\documentclass[12pt,reqno]{amsart}

\usepackage{a4wide}
\usepackage[unicode,final]{hyperref}
\usepackage[T1]{fontenc}    
\usepackage{enumitem}
\usepackage{graphicx}

\usepackage{amsmath}
\usepackage{amssymb}
\usepackage{amsxtra}
\usepackage{amsthm}

\usepackage[silent]{fixme}
\fxsetup{
  status=draft,
  layout={marginclue,noinline}
}
\fxsetface{inline}{\itshape}

\usepackage[mathscr]{eucal} 
\usepackage{mathrsfs}       
\usepackage{slashed}        

\usepackage{mathtools}
\usepackage{extpfeil}
\usepackage{proof}
\usepackage{ifnextok}
\usepackage{xstring}

\usepackage[all]{xy}
\SelectTips{cm}{}
\SilentMatrices
\ifx\pdfoutput\undefined
  \xyoption{dvips}
\fi
\usepackage{tikz-cd}

\newcommand{\ifundef}[1]{\expandafter\ifx\csname#1\endcsname\relax}

\DeclareMathAlphabet{\mathbbe}{U}{bbold}{m}{n}

\makeatletter

\def\re@DeclareMathSymbol#1#2#3#4{%
    \let#1=\undefined
    \DeclareMathSymbol{#1}{#2}{#3}{#4}}

\ifundef{Top}
  \DeclareSymbolFont{tcSyC}{U}{txsyc}{m}{n}
  \SetSymbolFont{tcSyC}{bold}{U}{txsyc}{bx}{n}
  \DeclareFontSubstitution{U}{txsyc}{m}{n}

  \re@DeclareMathSymbol{\Top}{\mathord}{tcSyC}{120}
  \re@DeclareMathSymbol{\Bot}{\mathord}{tcSyC}{121}
\fi

\ifundef{righthalfcup}
  \DeclareFontFamily{U}{MnSymbolC}{}
  \DeclareSymbolFont{mnSyC}{U}{MnSymbolC}{m}{n}
  \SetSymbolFont{mnSyC}{bold}{U}{MnSymbolC}{b}{n}
  \DeclareFontShape{U}{MnSymbolC}{m}{n}{
      <-6>  MnSymbolC5
     <6-7>  MnSymbolC6
     <7-8>  MnSymbolC7
     <8-9>  MnSymbolC8
     <9-10> MnSymbolC9
    <10-12> MnSymbolC10
    <12->   MnSymbolC12}{}
  \DeclareFontShape{U}{MnSymbolC}{b}{n}{
      <-6>  MnSymbolC-Bold5
     <6-7>  MnSymbolC-Bold6
     <7-8>  MnSymbolC-Bold7
     <8-9>  MnSymbolC-Bold8
     <9-10> MnSymbolC-Bold9
    <10-12> MnSymbolC-Bold10
    <12->   MnSymbolC-Bold12}{}

  \re@DeclareMathSymbol{\righthalfcup}{\mathord}{mnSyC}{184}
  \re@DeclareMathSymbol{\lefthalfcap}{\mathord}{mnSyC}{185}
\fi

\DeclareFontFamily{U}{MnSymbolA}{}
\DeclareSymbolFont{mnSyA}{U}{MnSymbolA}{m}{n}
\SetSymbolFont{mnSyA}{bold}{U}{MnSymbolA}{b}{n}
\DeclareFontShape{U}{MnSymbolA}{m}{n}{
    <-6>  MnSymbolA5
   <6-7>  MnSymbolA6
   <7-8>  MnSymbolA7
   <8-9>  MnSymbolA8
   <9-10> MnSymbolA9
  <10-12> MnSymbolA10
  <12->   MnSymbolA12}{}
\DeclareFontShape{U}{MnSymbolA}{b}{n}{
    <-6>  MnSymbolA-Bold5
   <6-7>  MnSymbolA-Bold6
   <7-8>  MnSymbolA-Bold7
   <8-9>  MnSymbolA-Bold8
   <9-10> MnSymbolA-Bold9
  <10-12> MnSymbolA-Bold10
  <12->   MnSymbolA-Bold12}{}

\re@DeclareMathSymbol{\twoheadedswarrow}{\mathord}{mnSyA}{30}

\makeatother

\newcommand{\mlaux}[3]{\setbox0=\hbox{$\mathsurround=0pt #2{#3}$}%
  \dimen0=\dp0\advance\dimen0 by \ht0\lower#1\dimen0\box0}

\newcommand{\makellapm}[2]{\hbox to 0pt{\hss$\mathsurround=0pt #1{#2}$}}
\newcommand{\llapm}{\relax\mathpalette\makellapm}
\newcommand{\makerlapm}[2]{\hbox to 0pt{$\mathsurround=0pt #1{#2}$\hss}}

\newcommand{\makelapm}[2]{\hbox to 0pt{\hss$\mathsurround=0pt #1{#2}$\hss}}

\newcommand{\makeushort}[3]{%
  \setbox0=\hbox{$\mathsurround=0pt #2{#3}$}%
  \hbox to 1\wd0{\hss\underbar{\hbox to #1\wd0{\hss\box0\hss}}\hss}}

\def\makebigger#1#2#3{\scalebox{#1}{$\mathsurround=0pt #2{#3}$}}
\def\bigger#1#2{{\relax\mathpalette{\makebigger{#1}}{#2}}}

\def\scaleuphalf{1.0954}

\newcommand{\op}{^{\mathord{\text{\rm op}}}}
\newcommand{\co}{^{\mathord{\text{\rm co}}}}

\renewcommand{\th}{^{\text{th}}}

\newcommand{\defeq}{\mathrel{:=}}
\newcommand{\eqdef}{\mathrel{=:}}

\makeatletter

\def\newmop{\@ifstar{\@newmop m}{\@newmop o}}
\def\@newmop#1{\@ifnextchar[{\@@newmop #1}{\@@@newmop #1}}
\def\@@newmop#1[#2]{\@declmathop #1#2}
\def\@@@newmop#1#2{\expandafter\@declmathop\expandafter #1\csname #2\endcsname{#2}}

\makeatother

\newdir{ |}{{}*!/-5pt/@{|}}

\newmop{im}
\newmop{coim}
\newmop{dom}
\newmop{cod}
\newmop{id}
\newmop{Map}

\newmop{obj}
\newmop{arr}
\newmop{sq}
\newmop{norm}

\newmop{el}

\newmop{ev}

\newmop{fun}

\newmop{Ext}
\newmop{icon}
\newmop{pbk}
\newmop{icom}

\newmop{cell}
\newmop{cof}
\newmop{fib}

\newmop{diag}

\newmop*{colim}
\newmop*{holim}

\newmop[\lan]{lan}
\newmop[\ran]{ran}

\newmop*{cone}
\newmop*{cocone}
\newmop*{base}

\newcommand{\comma}{\mathbin{\downarrow}}

\makeatletter
\newcommand{\rotatemath}[2]{\rotatebox[origin=c]{180}{$\m@th #1{#2}$}}
\makeatother

\newcommand{\yoneda}{\mathscr{Y}\!}

\newmop[\pr]{pr}
\newmop[\dm]{d}

\newmop[\cpr]{in}

\newcommand{\pocorner}{\hbox to 8pt{{\vrule height8pt depth0pt width0.5pt}%
    \vbox to 8pt{{\hrule height0.5pt width7.5pt depth0pt}\vfill}}}

\newcommand{\pbcorner}{\vbox to 0pt{\kern 4pt\hbox to 0pt{\kern 4pt%
      \vbox{{\hrule height0.5pt width7.5pt depth0pt}}%
      {\vrule height8pt depth0pt width0.5pt}\hss}\vss}}
\newcommand{\pbexcursion}{\save[]+<5pt,-5pt>*{\pbcorner}\restore}

\newcommand{\pwr}{\mathbin\pitchfork}

\newcommand{\wcolim}{\circledast}

\newcommand{\wlim}[2]{\{#1,#2\}}

\newmop{cls}

\newcommand{\leib}[1]{\mathbin{\widehat{#1}}}

\newcommand{\category}[1]{\underline{\smash[b]{\text{\rm{#1}}}}}

\newcommand{\cattwo}{{\bigger{1.12}{\mathbbe{2}}}}
\newcommand{\catone}{{\bigger{1.16}{\mathbbe{1}}}}
\newcommand{\iso}{{\mathbb{I}}}

\makeatletter

\def\Del@Sym{{\bigger\scaleuphalf{\mathbbe{\Delta}}}}

\def\del@fn{\futurelet\del@next}
\def\del@dn{\def\del@next}

\def\parsedel@{%
  \ifx +\del@next \del@dn+{\Del@Sym_{\mathord{+}}}%
  \else \del@dn {\del@fn\parsedel@@}%
  \fi\del@next}

\def\parsedel@@{%
  \ifx\space@\del@next \expandafter\del@dn\space{\del@fn\parsedel@@}%
  \else\ifx [\del@next \del@dn[{\del@fn\parsedel@@@}%
  \else\ifx _\del@next \del@dn{\Delta}%
  \else\ifx ^\del@next \del@dn{\Delta}%
  \else \del@dn{\Del@Sym}%
  \fi\fi\fi\fi\del@next}

\def\parsedel@@@{%
  \ifx\space@\del@next \expandafter\del@dn\space{\del@fn\parsedel@@@}%
  \else\ifx t\del@next \del@dn t{\Del@Sym_\infty\del@fn\parsedel@@@@}%
  \else\ifx b\del@next \del@dn b{\Del@Sym_{-\infty}\del@fn\parsedel@@@@}%
  \else \del@dn{\errmessage{unexpected modifier}}%
  \fi\fi\fi\del@next}

\def\parsedel@@@@{%
  \ifx\space@\del@next \expandafter\del@dn\space{\del@fn\parsedel@@@@}%
  \else\ifx ]\del@next \del@dn]{}%
  \else \del@dn{\errmessage{expecting close of option block}}%
  \fi\fi\del@next}

\def\Del{\del@fn\parsedel@}

\makeatother

\newcommand{\Horn}{\Lambda}

\newcommand{\Set}{\category{Set}}
\newcommand{\Cat}{\category{Cat}}

\newcommand{\sCat}{\sSet\text{-}\Cat}
\newcommand{\sSet}{\category{sSet}}
\newcommand{\qCat}{\category{qCat}}

\newcommand{\dmod}[3]{\xymatrix@=1.25em{{#2} \ar[r]|\mid^{ {#1}} & {#3}}}

\newcommand{\fbv}[1]{\{{#1}\}}

\newcommand{\join}{\mathbin\star}

\newmop{dec}
\newmop{fatdec}
\newmop{slc}
\newmop{fatslc}

\newcommand{\slice}{/}
\newcommand{\slicel}[2]{\vphantom{#2}^{{#1}\slice{}}\mkern-2mu{#2}}
\newcommand{\slicer}[2]{{#1\mkern-1mu}_{{}\slice{#2}}}

\newmop{ir} 
\newmop{incl}

\newcommand{\nrv}{N}
\newcommand{\ho}{h}

\newcommand{\nrvhc}{\nrv}
\newcommand{\hN}{\nrv}
\newcommand{\gC}{\mathfrak{C}}

\newcommand{\boundary}{\partial}

\def\reedyfilt#1_#2{#1_{\leq #2}}
\newmop{sk}
\newmop{cosk}
\newmop{res}

\newmop{map}
\newmop{Ho}

\newmop[\pth]{path}
\newmop{cyl}

\newmop[\const]{c}

\newcommand{\Kan}{\category{Kan}}

\newmop{coll}
\newmop{wgt}

\newdir{ >}{{}*!/-7pt/@{>}}
\newdir{u(}{{}*!/-4pt/@^{(}}
\newdir{d(}{{}*!/-4pt/@_{(}}
\newdir{|>}{%
  !/4.5pt/@{|}*:(1,-.2)@^{>}*:(1,+.2)@_{>}*+@{}}

\makeatletter
\def\makeslashed#1#2#3#4#5{#1{\mathpalette{\sla@{#2}{#3}{#4}}{#5}}}

\def\@mathlower#1#2#3{\setbox0=\hbox{$\m@th#2#3$}\lower#1\ht0\box0}
\def\mathlower#1#2{\mathpalette{\@mathlower{#1}}{#2}}
\makeatother

\newcommand{\inc}{\hookrightarrow}

\newcommand{\tfib}{\twoheadrightarrow}

\newcommand{\longtwoheadrightarrow}{\mathrel{\mathord{-}\mkern-3mu\mathord\twoheadrightarrow}}

\newcommand{\we}{\xrightarrow{\mkern10mu{\smash{\mathlower{0.6}{\sim}}}\mkern10mu}}

\newcommand{\trvfib}{\stackrel{\smash{\mkern-2mu\mathlower{1.5}{\sim}}}\longtwoheadrightarrow}

\newcommand{\To}{\Rightarrow}

\makeatletter

\def\tens@fn{\futurelet\tens@next}
\def\tens@dn{\def\tens@nextcont}
\newtoks\tens@toks
\def\addtotens@toks#1{\tens@toks=\expandafter{\the\tens@toks#1}}

\def\parsetens@@{%
    \ifx\space@\tens@next \expandafter\tens@dn\space{\tens@fn\parsetens@@}%
    \else\ifx ^\tens@next \tens@dn ^##1{\parsetens@procsep^\addtotens@toks{##1}%
      \tens@fn\parsetens@@}%
    \else\ifx _\tens@next \tens@dn _##1{\parsetens@procsep_\addtotens@toks{##1}%
      \tens@fn\parsetens@@}%
    \else\tens@dn{\ifx *\tens@last \else\addtotens@toks\egroup\fi\the\tens@toks}%
    \fi\fi\fi\tens@nextcont}

\def\parsetens@procsep#1{%
  \ifx *\tens@last \addtotens@toks{#1}\addtotens@toks\bgroup%
  \else\ifx \tens@last\tens@next \addtotens@toks,%
  \else \addtotens@toks\egroup\addtotens@toks\bgroup%
    \addtotens@toks\egroup\addtotens@toks{#1}\addtotens@toks\bgroup%
  \fi\fi\let\tens@last\tens@next}

\newcommand{\tn}[1]{\let\tens@last=*\tens@toks={#1}\tens@fn\parsetens@@}

\makeatother

\def\adjdisplay#1-|#2:#3->#4.{{%
    \xymatrix@R=0em@!C=2.5em{%
      *+[l]{#3} \ar@/_0.55pc/[rr]_-{#2} & {\bot} &
      *+[r]{#4}\ar@/_0.55pc/[ll]_-{#1}}}}

\def\adjdisplaytwo#1-|#2:#3->#4.{{%
\xymatrix@=1.2em{
      {#3}\ar@/_1.5ex/[rr]_-{#2}^-{}="one"
      & & {#4}
      \ar@/_1.5ex/[ll]_-{#1}^-{}="two"
      \ar@{}"one";"two"|{\bot}
    }}}

\def\tripleadjdisplay#1-|#2-|#3:#4->#5.{{%
\xymatrix@=2.4em{
{#4}\ar[r]|{#2} &
{#5} \ar@/_3ex/[l]_{#1}^{\bot} \ar@/^3ex/[l]_{\bot}^{#3}}
}}

\def\adjinline#1-|#2:#3->#4.{{#1}\dashv{#2}:#3\to #4}

\newcommand{\pent}[1]{
  \xybox{
    \POS (0,-15)*+{\a}="0",
         (-14,-5)*+{\b}="1",
         (-9,12)*+{\c}="2",
         (9,12)*+{\d}="3",
         (14,-5)*+{\e}="4"
    \POS"0" \ar "1"^{\labelstyle \ab}|{}="01"
    \POS"1" \ar "2"^{\labelstyle \bc}|{}="12"
    \POS"2" \ar "3"^{\labelstyle \cd}|{}="23"
    \POS"3" \ar "4"^{\labelstyle \de}|{}="34"
    \POS"0" \ar "4"_{\labelstyle \ae}|{}="04"
    \ifcase #1
    \POS"0" \ar "2"|{\labelstyle \ac}="02"
    \POS"0" \ar "3"|{\labelstyle \ad}="03"
    \POS"02";"1"**{}, ?(0.3) \ar@{=>} ?(0.7)^{\labelstyle \abc}
    \POS"03";"2"**{}, ?(0.25) \ar@{=>} ?(0.5)_{\labelstyle \acd}
    \POS"04";"3"**{}, ?(0.2) \ar@{=>} ?(0.4)_{\labelstyle \ade}
    \or
    \POS"1" \ar "3"|{\labelstyle \bd}="13"
    \POS"1" \ar "4"|{\labelstyle \be}="14"
    \POS"13";"2"**{}, ?(0.3) \ar@{=>} ?(0.7)_{\labelstyle \bcd}
    \POS"14";"3"**{}, ?(0.25) \ar@{=>} ?(0.5)_{\labelstyle \bde}
    \POS"04";"1"**{}, ?(0.25) \ar@{=>} ?(0.5)_{\labelstyle \abe}
    \or
    \POS"2" \ar "4"|{\labelstyle \ce}="24"
    \POS"0" \ar "2"|{\labelstyle \ac}="02"
    \POS"02";"1"**{}, ?(0.3) \ar@{=>} ?(0.7)^{\labelstyle \abc}
    \POS"04";"2"**{}, ?(0.2) \ar@{=>} ?(0.35)_{\labelstyle \ace}
    \POS"24";"3"**{}, ?(0.2) \ar@{=>} ?(0.6)^{\labelstyle \cde}
    \or
    \POS"1" \ar "3"|{\labelstyle \bd}="13"
    \POS"0" \ar "3"|{\labelstyle \ad}="03"
    \POS"04";"3"**{}, ?(0.2) \ar@{=>} ?(0.4)_{\labelstyle \ade}
    \POS"13";"2"**{}, ?(0.3) \ar@{=>} ?(0.7)_{\labelstyle \bcd}
    \POS"03";"1"**{}, ?(0.25) \ar@{=>} ?(0.5)^{\labelstyle \abd}
    \or
    \POS"2" \ar "4"|{\labelstyle \ce}="24"
    \POS"1" \ar "4"|{\labelstyle \be}="14"
    \POS"24";"3"**{}, ?(0.2) \ar@{=>} ?(0.6)^{\labelstyle \cde}
    \POS"04";"1"**{}, ?(0.25) \ar@{=>} ?(0.5)_{\labelstyle \abe}
    \POS"14";"2"**{}, ?(0.25) \ar@{=>} ?(0.5)^{\labelstyle \bce}
    \else\fi
  }
}

\newcommand{\pentofpent}[1]{
  \def\baselen{#1}
  \begin{xy}
    0;<\baselen,0mm>:
    *{\xybox{
        \POS(0,-4)*[o]{\pent 0}="zero"
        \POS(16,40)*[o]{\pent 3}="three"
        \POS(72,40)*[o]{\pent 1}="one"
        \POS(88,-4)*[o]{\pent 4}="four"
        \POS(44,-36)*[o]{\pent 2}="two"
        \ar@<1ex>"zero";"three"^-{\objectstyle\abcd}
        \ar@<1ex>"three";"one"^-{\objectstyle\abde}
        \ar@<1ex>"one";"four"^-{\objectstyle\bcde}
        \ar@<-1ex>"zero";"two"_-{\objectstyle\acde}
        \ar@<-1ex>"two";"four"_-{\objectstyle\abce}
        \ar@{=>}(44,-5);(44,+15)^{\objectstyle\abcde}
     }}
  \end{xy}
}

\makeatletter
\newcommand{\qc}[1]{\mathord{\text{\normalfont{\textsf{#1}}}}}
\newcommand{\qop}[1]{\mathord{\qc{#1}}}
\def\ec@#1#2<.>{{\mathcal{#1}\mkern-2mu\text{\normalfont{%
      \textsf{\slshape #2}}}\mkern2mu}}
\newcommand{\ec}[1]{\mathord{\ec@#1<.>}}
\newcommand{\eop}[1]{\mathord{\ec{#1}}}
\makeatother

\newcommand{\SSet}{\eop{SSet}}
\newcommand{\Flex}{\eop{Flex}}

\newcommand{\gr}{^{\mathrm{gr}}}

\newcommand{\ep}{_{\mathrm{epi}}}

\newcommand{\qA}{\qc{A}}
\newcommand{\qB}{\qc{B}}
\newcommand{\qC}{\qc{C}}
\newcommand{\qD}{\qc{D}}

\newcommand{\qK}{\qc{K}}

\newcommand{\qM}{\qc{M}}
\newcommand{\qS}{\qc{S}}

\newcommand{\qX}{\qc{X}}

\newcommand{\eK}{\ec{K}}

\newcommand{\eM}{\ec{M}}
\newcommand{\eA}{\ec{A}}
\newcommand{\eB}{\ec{B}}
\newcommand{\eC}{\ec{C}}
\newcommand{\eD}{\ec{D}}
\newcommand{\eE}{\ec{E}}

\newcommand{\Fun}{\qop{Fun}}
\newcommand{\hFun}{\ho\Fun}
\renewcommand{\Map}{\qop{Map}}

\renewcommand{\qCat}{\eop{QCat}}
\renewcommand{\Kan}{\eop{Kan}}

\newcommand{\Graph}{\category{Graph}}
\newcommand{\sCptd}{\sSet\text{-}\category{Cptd}}
\newcommand{\Cptd}{\category{Cptd}}

\newcommand{\Cube}{\mathord{\sqcap\llapm\sqcup}}
\newcommand{\CHorn}{\mathord{\sqcap\llapm\sqcap}}

\newmop{core}
\newmop{Fam}
\newmop{last}
\newmop{rv}
\newmop{free}

\usepackage{xr}

\newcommand{\extRef}[3]{%
  {\protect\IfBeginWith{#3}{itm:}{}{#2.}}\ref*{#1:#3}}

\newcommand{\refI}{\extRef{found}{I}}
\newcommand{\refII}{\extRef{cohadj}{II}}

\newcommand{\refIV}{\extRef{yoneda}{IV}}
\newcommand{\refV}{\extRef{equipment}{V}}
\newcommand{\refVI}{\extRef{comprehend}{VI}}

\setlist{}
\setenumerate{leftmargin=*,labelindent=\parindent}
\setitemize{leftmargin=*,labelindent=0.5\parindent}
\setdescription{leftmargin=1em}

\swapnumbers

\theoremstyle{plain}

\newtheorem{thm}{Theorem}[subsection]
\newtheorem{lem}[thm]{Lemma}
\newtheorem{cor}[thm]{Corollary}

\newtheorem{prop}[thm]{Proposition}

\theoremstyle{definition}
\newtheorem{defn}[thm]{Definition}
\newtheorem{ex}[thm]{Example}
\newtheorem{ntn}[thm]{Notation}

\theoremstyle{remark}
\newtheorem{obs}[thm]{Observation}

\newtheorem{rmk}[thm]{Remark}

\makeatletter
\let\c@equation\c@thm
\makeatother
\numberwithin{equation}{subsection}

\title[Recognizing quasi-categorical limits in homotopy coherent nerves]{Recognizing quasi-categorical limits and colimits in homotopy coherent nerves}

\author[Riehl]{Emily Riehl}
\address{
  Department of Mathematics \\
Johns Hopkins University \\
Baltimore, MD 21218\\
  USA
}
\email{eriehl@math.jhu.edu}

\author[Verity]{Dominic Verity}
\address{
  Centre of Australian Category Theory \\
  Macquarie University \\
  NSW 2109 \\
  Australia
}
\email{dominic.verity@mq.edu.au}

\date{\today}


\begin{document}

  \ifpdf
  \DeclareGraphicsExtensions{.pdf, .jpg, .tif}
  \else
  \DeclareGraphicsExtensions{.eps, .jpg}
  \fi

  \begin{abstract} In this paper we prove that various quasi-categories whose objects are $\infty$-categories in a very general sense are \emph{complete}: admitting limits indexed by all simplicial sets. This result and others of a similar flavor follow from a general theorem in which we characterize the data that is required to define a limit cone in a quasi-category constructed as a homotopy coherent nerve. Since all quasi-categories arise this way up to equivalence, this analysis covers the general case. Namely, we show that quasi-categorical limit cones may be modeled at the point-set level by \emph{pseudo homotopy limit cones}, whose shape is governed by the weight for pseudo limits over a homotopy coherent diagram but with the defining universal property up to equivalence, rather than isomorphism, of mapping spaces. Our applications follow from the fact that the $(\infty,1)$-categorical core of an $\infty$-cosmos admits weighted homotopy limits for all flexible weights, which includes in particular the weight for pseudo cones.
  \end{abstract}

  \maketitle
  \tableofcontents

\section{Introduction}

Previous work \cite{RiehlVerity:2012tt, RiehlVerity:2012hc, RiehlVerity:2013cp,RiehlVerity:2015fy,RiehlVerity:2015ke,RiehlVerity:2017cc}  has shown that much of the fundamental theory of $(\infty,1)$-categories may be developed in a natively model-independent fashion by applying techniques of formal category theory to the categorical universes in which $(\infty,1)$-categories live as objects. Our approach is ``synthetic'' in the sense that our proofs do not depend on precisely what these $(\infty,1)$-categories \emph{are}, but rather rely upon a relatively sparse axiomatisation of the universe in which they \emph{live}. To describe an appropriate ``universe,'' we introduce the notion of an $\infty$-\emph{cosmos}, a (large) quasi-categorically enriched category $\eK$ satisfying certain axioms recalled in Definition \ref{defn:cosmos}: roughly, these ask that an $\infty$-cosmos is an ``$(\infty,2)$-category with \emph{flexible} limits'' about which we will have more to say below. In contrast with the work of \cite{Toen:2005vu, BSP:2011ot}, the axioms for an $\infty$-cosmos do not characterize ``categories whose objects are $(\infty,1)$-categories.'' Nonetheless, so that our statements about $\infty$-cosmoi suggest their most natural interpretation, we refer to the objects of any $\infty$-cosmos as  $\infty$-\emph{categories}. 

The prototypical example of an $\infty$-cosmos is the $\infty$-cosmos whose objects are \emph{quasi-categories}, a model of $(\infty,1)$-categories as simplicial sets satisfying the weak Kan condition, and whose function complexes are the quasi-categories of functors between them; this example is reviewed in Example \ref{ex:qcat-cosmos}.  But there are other $\infty$-cosmoi whose objects are complete Segal spaces, Segal categories, or 1-complicial sets, each of these being models of $(\infty,1)$-\emph{categories}. These $\infty$-cosmoi are \emph{biequivalent} to $\qCat$, meaning that the ``underlying quasi-category functor $\Fun_{\eK}(1,-) \colon \eK \to \qCat$ is surjective on objects up to equivalence and induces a local equivalence on functor spaces.

The axioms of an $\infty$-cosmos are not intended to only describe categorical universes for $(\infty,1)$-categories; several models of $(\infty,n)$-categories form the objects of an $\infty$-cosmos for instance.\footnote{Known $\infty$-cosmoi of $(\infty,n)$-categories include  $\theta_n$-spaces,  iterated complete Segal spaces,  $n$-complicial sets, and $n$-quasi-categories.} In this work, we will make use of various ``exotic'' $\infty$-cosmoi that can be constructed from a given $\infty$-cosmos $\eK$, including the slice category $\eK_{/B}$ over an $\infty$-category in $B$, two $\infty$-cosmoi of arrows that we introduce for the first time in Propositions \ref{prop:isofib-cosmoi} and \ref{prop:K^2_eq-cosmos}, and the $\infty$-cosmos of groupoidal ``spaces'' in $\eK$ established in Proposition \ref{prop:gpdal-infty-cosmos}.

Using only the axioms of an $\infty$-cosmos, we can develop a fairly comprehensive theory of limits and colimits of diagrams valued in an $\infty$-category, defining these notions in a variety of equivalent ways and proving, for instance, that right adjoints preserve limits. But when it comes to constructing examples of $(\infty,1)$-categories or computing limits or colimits therein, ``analytic'' techniques associated to a specific model of $(\infty,1)$-categories are appropriate. As most practitioners already believe, and a future paper will justify, it makes no essential difference which model of $(\infty,1)$-categories is chosen as long as the associated $\infty$-cosmos $\eK$ is biequivalent to the $\infty$-cosmos for quasi-categories. For simplicity we choose to work in $\qCat$ itself.

At the conclusion of the previous work in this series \cite{RiehlVerity:2017cc}, we show that any quasi-category $\qC$ is equivalent to the homotopy coherent nerve of a Kan-complex-enriched category $\eC$, namely, the homotopy coherent nerve of the full simplicial subcategory of $\qCat_{\!/\qC}$ spanned by the \emph{representable cartesian fibrations} $p_0 \colon \qC\comma c \tfib \qC$ indexed by its elements $c \colon 1 \to \qC$.  So in this paper we answer the question of constructing limits and colimits of diagrams valued in a quasi-category $f \colon X \to \qC$ by considering the corresponding homotopy coherent diagram $F\colon\gC{X} \to \eC$ valued in a Kan-complex-enriched category. Our main theorem gives a condition that characterizes quasi-categorical limits or colimits in this context.

{
\renewcommand{\thethm}{\ref{thm:nerve-completeness}, \ref{thm:nerve-completeness-converse}}
\begin{thm}
For any Kan-complex-enriched category $\eC$ and simplicial set $X$, if a homotopy coherent diagram $D \colon \gC[X] \to \eC$ admits a pseudo homotopy limit in $\eC$, then the corresponding limit cone $\gC[\Del^0\join X] \to \eC$ transposes to define a limit cone over the transposed diagram $d \colon X \to \qC$ in the homotopy coherent nerve of $\eC$. Conversely, if the diagram $d$ admits a limit in the quasi-category $\qC$, then the limit cone $\Del^0\join X \to \qC$ transposes to define a pseudo homotopy limit cone over $D$ in $\eC$. 

  Consequently, the quasi-category $\qC$ is complete if and only if $\eC$ admits pseudo homotopy limits for all simplicial sets $X$.
\end{thm}
\addtocounter{thm}{-1}
}

The statement requires some explanation. The limit notions appropriate to simplicially enriched category theory are \emph{weighted} by a particular functor, which describes the ``shape'' of cones over the given diagram; the idea is that in the context of an ambient simplicial enrichment, these cone legs may contain higher dimensional simplices. If the weight is ``fat enough,'' then the corresponding weighted limit notion is homotopically well behaved. For instance, we show in Proposition \ref{prop:flexible-weights-are-htpical} that every $\infty$-cosmos admits all limits with \emph{flexible} weights, and such limits are invariant under pointwise equivalence between diagrams.

The quasi-category $\qK$ of $\infty$-categories in an $\infty$-cosmos $\eK$ is defined not as the homotopy coherent nerve of the $\infty$-cosmos itself but rather as the homotopy coherent nerve of its $(\infty,1)$-\emph{categorical core}. The inclusion of the $(\infty,1)$-categorical core creates those flexible limits whose weights are valued in Kan complexes but does not create those flexible limits whose weights are valued in general simplicial sets, at least not strictly. But if we relax the defining universal property of a weighted limit to demand an equivalence rather than an isomorphism of quasi-categories, then the $(\infty,1)$-categorical core of an $\infty$-cosmos admits all \emph{flexible weighted homotopy limits}, as we demonstrate in Corollary \ref{cor:infty-one-core-flex}. Similarly, the full subcategory of fibrant-cofibrant objects in a simplicial model category admits all flexible weighted homotopy limits or colimits.

The \emph{pseudo homotopy limits} in the statement of Theorem \ref{thm:nerve-completeness} refer to a flexible weighted homotopy limit with a particular weight, namely the weight $W_X$ for pseudo cones over a homotopy coherent diagram of shape $X$, which we define in \ref{defn:weight-for-pseudo-limits}. Since $\infty$-cosmoi and their $(\infty,1)$-categorical cores admit pseudo-weighted homotopy limits, we deduce the following completeness results as corollaries of Theorem \ref{thm:nerve-completeness}


{
\renewcommand{\thethm}{\ref{prop:qcat-of-cosmos-complete}}
\begin{prop} For any $\infty$-cosmos $\eK$, the large quasi-category $\qK$ of $\infty$-categories in $\eK$ is small complete.
\end{prop}
\addtocounter{thm}{-1}
}

{
\renewcommand{\thethm}{\ref{prop:qcat-of-space-complete}}
\begin{prop} For any $\infty$-cosmos $\eK$, the large quasi-category $\qS_{\eK}$ of groupoidal $\infty$-categories in $\eK$ is small complete and closed under small limits in the quasi-category $\qK$.
\end{prop}
\addtocounter{thm}{-1}
}

{
\renewcommand{\thethm}{\ref{prop:qcat-simplicial-model}}
\begin{prop}
  If $\eM$ is a simplicial model category then the quasi-category $\qM$, defined as the homotopy coherent nerve of the full simplicial subcategory of fibrant-cofibrant objects, is small complete and cocomplete.
\end{prop}
\addtocounter{thm}{-1}
}
The last of these corollaries recovers a result  first proven by Barnea, Harpaz, and Horel in \cite[2.5.9]{BarneaHarpazHorel:2017pc}, while Szumi\l{}o proves a result similar to the first pair of results in the context of (unenriched) cofibration categories  \cite{Szumilo:2014tm}.

This paper contains all of the background needed to fill in the details of this outline, with the proofs of these results appearing in  \S\ref{sec:complete}.  To concisely cite previous work in this program, we refer to the results of \cite{RiehlVerity:2012tt, RiehlVerity:2012hc, RiehlVerity:2013cp,RiehlVerity:2015fy,RiehlVerity:2015ke,RiehlVerity:2017cc}  as I.x.x.x., II.x.x.x, III.x.x.x, IV.x.x.x, V.x.x.x, or VI.x.x.x respectively, though the statements of the most important results are reproduced here for ease of reference. When an external reference accompanies a restated result, this generally indicates that more expository details can be found there.

In \S\ref{sec:cosmoi} we review the axioms of an $\infty$-cosmos and construct several new examples of $\infty$-cosmoi that will be used in this paper. We also introduce and investigate the $\infty$-cosmos of \emph{groupoidal objects} or ``spaces'' in an $\infty$-cosmos and explore the construction of the $(\infty,1)$-categorical core of an $\infty$-cosmos. In \S\ref{sec:limits}, we review the synthetic theory of limits in an $\infty$-cosmos, presenting only the minimal details that we will require here to prove our main theorem. In \S\ref{sec:flexible} we introduce the class of flexible weighted limits and flexible weighted homotopy limits, proving the existence results mentioned above. In \S\ref{sec:computads}, we develop the technical tools needed to prove our main theorems. We review the homotopy coherent nerve functor and its left adjoint, \emph{homotopy coherent realisation}, which describes the shape of homotopy coherent diagrams. We also define the weight for pseudo limits of homotopy coherent diagrams as a \emph{collage}, a simplicial category that indexes pseudo limit cones. With this theory in place, the proof of Theorem \ref{thm:nerve-completeness} in \S\ref{sec:complete} is just a matter of following one's nose.

A paper on the subject of computing limits and colimits in quasi-categories might have lead the reader to anticipate construction results of a different flavor, reducing limits or colimits of arbitrary shaped diagrams to simpler ones. In a sequel \cite{RiehlVerity:2018oc}, which depends in a crucial way on Theorem \ref{thm:nerve-completeness}, we prove a dual pair of such results, reducing limits indexed by a simplicial set $X$ to products and pullbacks. Using Yoneda lemma techniques, we liberate ourselves from the $\infty$-cosmos of quasi-categories and prove these results in the fully general context of $\infty$-categories in any $\infty$-cosmos.

The reader will note that a number of the results in this paper are well known in one form or another. In particular, this includes our main theorem, a close analogue of which was proven by Lurie in \cite[4.2.4.1]{Lurie:2009fk}.  The novelty of our narrative is that our constructions of these limits and colimits are concrete and direct, in that  we compute the required universal properties via direct analysis of homotopy coherent structures, rather than by constructing appropriate model structures on diagram categories. We find it aesthetically pleasing that the formulation of this result given here directly transposes a diagram $f \colon X \to \qC$ and a quasi-categorical cone into a homotopy coherent diagram $F \colon \gC[X] \to \eC$ and a pseudo cone and observes that these cones enjoy corresponding universal properties, without needing to ``straighten'' a quasi-categorical into an equivalent functor between Kan-complex-enriched categories; see \cite[4.2.4.7]{Lurie:2009fk}. These technical details are helpful when we apply Theorem \ref{thm:nerve-completeness} to the meta-theory of our $\infty$-cosmos framework, a topic we expand upon further in the immediate sequel to this paper.

For us a further motivation arises from the theory of (nee.~weak) \emph{complicial sets\/} \cite{Verity:2007:wcs1,Verity:2007rm}. The concrete arguments presented here generalise routinely to apply in the complicial context; all they require is a certain delicacy in tracking the \emph{markings} (or \emph{stratification}) placed on simplices in our homotopy coherent structures. In that way, we obtain constructions of higher $(\infty,\infty)$-limit types in the homotopy coherent nerves of complicially enriched categories. When we first started thinking about such issues, we went looking for a concrete presentation of the corresponding results in the quasi-categorical literature. We were unable to find one that generalised easily to complicial sets; hence the current paper. We could have proceeded directly to give the fully complicial arguments but felt it prudent, from an expository perspective, to expose the simpler unmarked case first. 

\subsection{Size conventions}

The quasi-categories defined as homotopy coherent nerves are typically large. All other quasi-categories or simplicial sets, particularly those used to index homotopy coherent diagrams, are assumed to be small. In particular,    when discussing the existence of limits and colimits we shall implicitly assume that these are indexed by small categories, and correspondingly, completeness and cocompleteness properties will implicitly reference the existence of small limits and small colimits.  Here, as is typical, ``small'' sets will usually refer to those members of a Grothendieck universe defined relative to a fixed inaccessible cardinal.

Our intent here is simply to provide a size classification which allows us state and prove results that require such a distinction for non-triviality, principally those of the form ``such and such a \emph{large} category admits all \emph{small} limits''. Our arguments mostly comprise elementary constructions, so in applications this size distinction need not invoke the full force of a Grothendieck universe, indeed it might be as simple as that between the finite and the infinite. At the other extreme it might involve the choice of two Grothendieck universes to prove results about large categories. One such result is Theorem~\ref{thm:nerve-completeness-converse} which, on interpreting its size distinction relative to a second larger Grothendieck universe, provides a result which applies to large simplicial categories. Indeed, with a little more care, the proof of that result may be adapted to apply to large and locally small simplicial categories without the introduction of a second universe (by restricting arguments only to \emph{accessible} simplicial presheaves) although we choose not to worry the reader with such arcana here. 

We use a common typeface --- e.g.,~$\qC$, $\qK$,  --- to differentiate small and large quasi-categories from generic $\infty$-categories $A$ and simplicial sets $X$. Throughout, we attempt to distinguish between enriched and unenriched categorical settings. In particular, we write $\SSet$ for the cartesian closed (and thus simplicially enriched) category of simplicial sets and $\sSet$ for the underlying 1-category of simplicial sets and simplicial maps.
 
\subsection{Acknowledgements}

The authors are grateful for support from the National Science Foundation (DMS-1551129 and DMS-1652600) and from the Australian Research Council (DP160101519). This work was commenced when the second-named author was visiting the first at Harvard and then at Johns Hopkins, continued while the first-named author was visiting the second at Macquarie, and completed after everyone finally made their way home. We thank all three institutions for their assistance in procuring the necessary visas as well as for their hospitality. The final published manuscript benefitted greatly from the  astute suggestions of an eagle-eyed referee. We are also grateful to the referee for the sequel \cite{RiehlVerity:2018oc} who noticed that the proof of the converse result Theorem \ref{thm:nerve-completeness-converse} originally given there could be simplified enough so that we could include it here.


\section{\texorpdfstring{$\infty$}{infinity}-cosmoi}\label{sec:cosmoi}

An $\infty$-\emph{cosmos} is a type of $(\infty,2)$-category satisfying a very sparse list of axioms appropriate for  the ``universe'' in which  $\infty$-categories live as objects. In \S\ref{ssec:cosmoi-background}, we briefly review this notion and then construct a number of new $\infty$-cosmoi from a given $\infty$-cosmos $\eK$ that we will make use of here and elsewhere. 

In \S\ref{ssec:core}, we consider two relevant substructures of an $\infty$-cosmos. The first is the full subcategory of ``spaces'' in $\eK$, objects that satisfy a condition of being \emph{groupoidal} that we introduce in several equivalent forms. The second construction is of the maximal $(\infty,1)$-category contained within an $\infty$-cosmos, this having the same objects but with the quasi-categorical homs replaced by their maximal Kan complex ``groupoid cores.''

\subsection{\texorpdfstring{$\infty$}{infinity}-cosmoi and their homotopy 2-categories}\label{ssec:cosmoi-background}

An $\infty$-cosmos is a category $\eK$ whose objects $A, B$ we call $\infty$-\emph{categories} and whose function complexes $\Fun_{\eK}(A,B)$ are quasi-categories of \emph{functors} between them. The handful of axioms imposed on the ambient quasi-categorically enriched category $\eK$ permit the development of a general theory of $\infty$-categories ``synthetically,'' i.e., only in reference to this axiomatic framework. We work in an $\infty$-cosmos $\eK$ with all objects cofibrant, in contrast to the more general notion first introduced in \S\refIV{sec:cosmoi}.

\begin{defn}[$\infty$-cosmos]\label{defn:cosmos}
An $\infty$-\emph{cosmos} is a simplicially enriched category $\eK$ whose 
\begin{itemize}
\item objects we refer to as the \emph{$\infty$-categories} in the $\infty$-cosmos, whose
\item hom simplicial sets $\Fun_{\eK}(A,B)$ are  quasi-categories, 
\end{itemize} and that is equipped with a specified subcategory of \emph{isofibrations}, denoted by ``$\tfib$'',
satisfying the following axioms:
 \begin{enumerate}[label=(\alph*),series=defn:cosmos]
    \item\label{defn:cosmos:a} (completeness) As a simplicially enriched category,  $\eK$ possesses a terminal object $1$, small products, cotensors $A^U$ of  objects $A$ by all small simplicial sets $U$, inverse limits of countable sequences of isofibrations, and pullbacks of isofibrations along any functor.
    \item\label{defn:cosmos:b} (isofibrations) The class of isofibrations contains the isomorphisms and all of the functors $!\colon A \tfib 1$ with codomain $1$; is stable under pullback along all functors; is closed under inverse limit of countable sequences; and if $p\colon E\tfib B$ is an isofibration in $\eK$ and $i\colon U\inc V$ is an inclusion of  simplicial sets then the Leibniz cotensor $i\leib\pwr p\colon E^V\tfib E^U\times_{B^U} B^V$ is an isofibration. Moreover, for any object $X$ and isofibration $p \colon E \tfib B$, $\Fun_{\eK}(X,p) \colon \Fun_{\eK}(X,E) \tfib \Fun_{\eK}(X,B)$ is an isofibration of quasi-categories.
\end{enumerate}
\end{defn}

For ease of reference, we refer to the limit types listed in axiom \ref{defn:cosmos:a} as the \emph{cosmological limit types}, these referring to diagrams of a particular shape with certain maps given by isofibrations.

\begin{rmk} As is revealed by our previous papers in this series, and suggested by the constructions appearing in \S\ref{ssec:commas}, for much of the development of the formal theory of $\infty$-categories only finite flexible weighted limits are needed. For this reason, the original definition of an $\infty$-cosmos only asks for the finite instances of \ref{defn:cosmos}\ref{defn:cosmos:a}. In the present treatment, we find it convenient to allow for cotensors with all simplicial sets, not just the finitely presented ones, and employ inductive arguments, such as appearing in the proof of Proposition \ref{prop:flexible-weights-are-htpical}, that make use of arbitrary small products and countable inverse limits of sequences of isofibrations, but we trust that the reader will have no difficulty adapting these results to their finite or countable variants in an $\infty$-cosmos admitting a more restricted class of weighted limits.
\end{rmk}

The underlying category of an $\infty$-cosmos $\eK$ has a canonical subcategory of representably-defined equivalences, denoted by ``$\we$'', satisfying the 2-of-6 property: a functor $f \colon A \to B$ is an \emph{equivalence} just when the induced functor $\Fun_{\eK}(X,f) \colon \Fun_{\eK}(X,A) \to \Fun_{\eK}(X,B)$ is an equivalence of quasi-categories for all objects $X \in \eK$.  The  \emph{trivial fibrations}, denoted by ``$\trvfib$'', are those functors that are both equivalences and isofibrations. These axioms imply that the underlying 1-category of an $\infty$-cosmos is a category of fibrant objects in the sense of Brown. Consequently, many familiar homotopical properties follow from \ref{defn:cosmos:a} and \ref{defn:cosmos:b}.   In particular it follows, from~\ref{defn:cosmos}\ref{defn:cosmos:b}, that if $p\colon E\trvfib B$ is a trivial fibration in $\eK$ then for all objects $A$ the simplicial map $\Fun_{\eK}(A,p)\colon\Fun_{\eK}(A,E)\to\Fun_{\eK}(A,B)$ is a trivial fibrations of quasi-categories. Consequently we have:  
   \begin{enumerate}[label=(\alph*), resume=defn:cosmos]
   \item\label{defn:cosmos:c} (cofibrancy) All objects are \emph{cofibrant}, in the sense that they enjoy the left lifting property with respect to all trivial fibrations in $\eK$.
     \[
       \xymatrix{
         & E \ar@{->>}[d]^{\rotatebox{90}{$\displaystyle\sim$}} \\
         A \ar[r] \ar@{-->}[ur]^{\exists} & B}
     \] 
  \end{enumerate}
  which was asserted as a (redundant) axiom in Definition \refV{qcat.ctxt.cof.def}.

  

\begin{ex}[$\infty$-cosmos of quasi-categories]\label{ex:qcat-cosmos}
The prototypical example is the $\infty$-cosmos $\qCat$ of quasi-categories, with function complexes defined to be the exponential objects in  the  cartesian closed category of simplicial sets. An \emph{isofibration} is an inner fibration that has the right lifting property with respect to the inclusion $\catone\inc\iso$ of either endpoint of the (nerve of the) free-standing isomorphism. That is, a map $f \colon\qA \to \qB$ of quasi-categories is an \emph{equivalence} just when there exists a map $g \colon \qB \to \qA$ together with maps $\qA \to \qA^\iso$ and $\qB \to \qB^\iso$ that restrict along the vertices of $\iso$ to the maps $\id_{\qA}$, $gf$, $fg$, and $\id_{\qB}$ respectively: 
\[ f \colon \qA \we \qB\quad \mathrm{iff}\quad \exists g \colon \qB \we \qA\quad \mathrm{and} \quad \vcenter{\xymatrix{ & \qA \\ \qA \ar[r] \ar[ur]^{\id_{\qA}} \ar[dr]_{gf} & \qA^\iso \ar@{->>}[u]_{p_0} \ar@{->>}[d]^{p_1} \\ & \qA}} \quad\mathrm{and}\quad \vcenter{\xymatrix{ & \qB \\ \qB \ar[r]\ar[ur]^{fg} \ar[dr]_{\id_{\qB}} & \qB^{\iso} \ar@{->>}[u]_{p_0} \ar@{->>}[d]^{p_1} \\ & \qB}}\]
\end{ex}

\begin{rmk}\label{rmk:cosmoi-from-model-cats}
The $\infty$-cosmos of quasi-categories can also be described using the language of model categories. It is the full subcategory of fibrant objects, with the isofibrations and equivalences respectively the fibrations and weak equivalences between fibrant objects, in a model category that is enriched over the Joyal model structure on simplicial sets and in which all fibrant objects are cofibrant. It is easy to verify that any category of fibrant objects arising in this way defines an $\infty$-cosmos (see Lemma \refIV{lem:model-categories-cosmoi}). This is the source of  many of our examples, which are described in \S\refIV{sec:cosmoi}.
\end{rmk}

For any $\infty$-category $A$ in any $\infty$-cosmos $\eK$, the strict slice category $\eK_{/A}$ of isofibrations over $A$  is again an $\infty$-cosmos. It follows that all of our theorems in this axiomatic framework immediately have fibred analogues.

\begin{prop}[{sliced $\infty$-cosmoi, \refV{defn:sliced-cosmoi}}]\label{prop:sliced-cosmoi}
If $\eK$ is any $\infty$-cosmos and $A \in \eK$ is any object, then there is an $\infty$-cosmos $\eK_{/A}$, the \emph{sliced $\infty$-cosmos of $\eK$ over $A$}, whose:
\begin{itemize}
\item objects are isofibrations $p \colon E \tfib A$ with codomain $A$;
\item functor space from $p \colon E \tfib A$ to $q \colon F \tfib A$ is defined by  taking the pullback
\[
    \xymatrix@=1.5em{
      {\Fun_{\eK_{/A}}(p,q)}\pbexcursion\ar[r]\ar@{->>}[d] &
      {\Fun_{\eK}(E,F)}\ar@{->>}[d]^{\Fun_{\eK}(E,q)} \\
      {1}\ar[r]_-{p} & {\Fun_{\eK}(E,A)}
    }
\]
  in simplicial sets;
\item isofibrations, equivalences, and trivial fibrations are created by the forgetful functor $\eK_{/A} \to \eK$;
\end{itemize}
and in which the simplicial limits are defined in the usual way for sliced simplicial categories.
\end{prop}

We may assemble the sliced $\infty$-cosmoi into a single $\infty$-cosmos $\eK^{\cattwo}$:

\begin{prop}[$\infty$-cosmoi of isofibrations]\label{prop:isofib-cosmoi} For any $\infty$-cosmos $\eK$, there is an $\infty$-cosmos $\eK$ which has:
  \begin{itemize}
  \item objects all isofibrations $p \colon E \tfib A$ in $\eK$;
  \item functor space from $p \colon E \tfib A$ to $q \colon F \tfib B$ defined by taking the pullback
    \[
      \xymatrix@R=1.5em@C=4em{
        {\Fun_{\eK^{\cattwo}}(p,q)}\pbexcursion\ar[r]\ar@{->>}[d] &
        {\Fun_{\eK}(E,F)}\ar@{->>}[d]^{\Fun_{\eK}(E,q)} \\
        {\Fun_{\eK}(A,B)}\ar[r]_-{\Fun_{\eK}(p,B)} & {\Fun_{\eK}(E,B)}
      }
    \]
    in simplicial sets so, in particular, the $0$-arrows from $p$ to $q$ are commutative squares
    \begin{equation}\label{eq:K-cattwo-squares}
      \xymatrix@=2em{
        {E}\ar@{->>}[d]_{p}\ar[r]^{g} & {F}\ar@{->>}[d]^{q} \\
        {A}\ar[r]_{f} & {B}
      }
    \end{equation}
    in $\eK$;
  \item equivalences those squares~\eqref{eq:K-cattwo-squares} whose components $f$ and $g$ are equivalences in $\eK$ and isofibrations (resp.\ trivial fibrations) those squares for which the map $f$ and the induced map $E\dashrightarrow A\times_B F$ (and thus also $g$) are isofibrations (resp.\ trivial fibrations) in $\eK$.
  \end{itemize}
  The simplicial limits of \ref{defn:cosmos}\ref{defn:cosmos:a} are defined object-wise in $\eK$, or in other words are jointly created by the domain and codomain projections $\dom,\cod\colon\eK^{\cattwo}\to\eK$.
\end{prop}

If the reader is comfortable thinking of $\infty$-cosmoi as categories of fibrant objects arising from model categories with the form described in Remark \ref{rmk:cosmoi-from-model-cats} and is sufficiently well-acquainted with the model category literature, then this result is more or less obvious. Note that the definitions of equivalences, isofibrations, and trivial fibrations in $\eK^\cattwo$ coincide with the Reedy weak equivalences, Reedy fibrations, and Reedy trivial fibrations when $\cattwo$ is considered as an inverse category. Such readers are encouraged to skip to Proposition \ref{prop:K^2_eq-cosmos}. For those for whom this sort of abstract homotopy theory is less familiar, we give the following justification.

\begin{proof}
As the proof will reveal, it is judicious to separate the simplicially enriched aspects of the axioms \ref{defn:cosmos}\ref{defn:cosmos:a}-\ref{defn:cosmos:b} from the unenriched aspects. To start, observe that the cotensor axioms are easy. If $p\colon E\tfib A$ is an object of $\eK^\cattwo$, that is to say an isofibration in $\eK$, then the map $p^X\colon E^X\tfib A^X$ induced between cotensors by a simplicial set $X$ in $\eK$ is an isofibration, by axiom~\ref{defn:cosmos}\ref{defn:cosmos:b}, and is easily seen to be the required cotensor in $\eK^\cattwo$.

The other limit types are a matter of unenriched category theory, so for the remainder of this proof we identify $\eK$ with its underlying category.  We shall temporarily adopt the notation $\bar\eK^{\cattwo}$ for the category whose objects are all arrows in $\eK$ and whose arrows are commutative squares. 

The pair of projections $\dom,\cod\colon\bar\eK^\cattwo\to \eK$ jointly create limits  simply because $\bar\eK^\cattwo$ is  a category of functors. In particular, if $F\colon \eD\to \eK^\cattwo$ is a diagram of one of the conical cosmological limit types discussed in \ref{defn:cosmos}\ref{defn:cosmos:a}---a pullback, product or countable tower in which certain maps are asked to be isofibrations in the sense described in the statement---then the projected diagrams $\dom F,\cod F\colon\eD\to\eK$ are also of that kind in the $\infty$-cosmos $\eK$, so they admit limits which extend in the usual way to give a limit of the diagram $F$ in the functor category $\bar\eK^\cattwo$.

  Our task now is twofold, first we must demonstrate that the limit we have constructed in $\bar\eK^\cattwo$ is an isofibration, which suffices to show that it gives a limit for the diagram $F$ in the full subcategory $\eK^\cattwo$. Then we must also show that the maps in its limit cone that are required to be isofibrations by \ref{defn:cosmos}\ref{defn:cosmos:b} are indeed isofibrations in $\eK^\cattwo$ of the kind described in the statement. We do this using a modicum of  fibred category theory.

  The codomain projection $\cod\colon \eK^{\cattwo}\to\eK$ is a cartesian fibration of categories; its cartesian lifts are constructed by taking pullbacks of isofibrations in $\eK$. Note, however, that the expanded codomain functor $\cod\colon\bar\eK^\cattwo \to\eK$ is not itself a cartesian fibration, because $\eK$ does not admit all pullbacks; but the fact that the cartesian maps of $\cod\colon\eK^\cattwo\to\eK$ are given as pullbacks in $\eK$ can equally well be stated as posting that the full inclusion $\eK^\cattwo\subset\bar\eK^\cattwo$ carries them to cartesian maps for $\cod\colon\bar\eK^\cattwo\to\eK$.

  We may now rephrase the definition of the equivalences, isofibrations, and trivial fibrations of $\eK^\cattwo$ in this language. Specifically we know that the total category of any cartesian fibration admits a factorisation of its arrows into a composite of an arrow in a fibre (sometimes called a \emph{vertical arrow}) and a cartesian arrow (sometimes called a \emph{horizontal arrow}) which is unique up to unique vertical isomorphism. In this case, given an arrow in the total space of $\cod\colon\eK^\cattwo\to\eK$ as depicted in~\eqref{eq:K-cattwo-squares} its horizontal factor is simply the universal square associated with the pullback $A\times_B F$ and its vertical factor is the induced comparison map $E\dashrightarrow A\times_BF$ in the slice $\eK_{/A}$, this being the fibre of $\cod\colon\eK^\cattwo\to\eK$ over $A$. It follows that such an arrow is an isofibration (resp.\ trivial fibration, equivalence) in $\eK^{\cattwo}$ iff its codomain component $\cod(g,f)=f$ is an isofibration (resp.\ trivial fibration, equivalence) in $\eK$ and its vertical factor is an isofibration (resp.\ trivial fibration, equivalence) in the sliced cosmos $\eK_{/A}$.

  It is clear from the $\infty$-cosmos axioms that for any arrow $f\colon A\to B$ in $\eK$ the pullback functor $f^*\colon\eK_{/B}\to\eK_{/A}$ preserves isofibrations and trivial fibrations. This result and some basic facts relating horizontal-vertical factorisations of maps to those of their composites suffices to prove that the class of isofibrations in $\eK^\cattwo$ is closed under composition and is stable under pullback along all arrows.
  
  It remains to verify that the appropriate components of the limits constructed above in $\bar\eK^\cattwo$ are isofibrations and define limits in $\eK^\cattwo$. We do this in Proposition \ref{prop:K^2-cosmos-limits} by applying an abstract result from fibred category theory recalled in Lemma \ref{lem:fibration.limit.fact}. This will complete the proof.
\end{proof}

\begin{lem}\label{lem:fibration.limit.fact}
  Suppose that $P\colon \eE\to\eB$ is a functor of categories and that $F\colon \eD\to\eE$ is a diagram which admits a limit $E\in\eE$ displayed by a limit cone $\pi\colon \Delta_E\To F$. Furthermore assume that this limit is preserved by the functor $P$ and that the limit cone $P\pi\colon \Delta_{PE}\To PF$ admits a $P$-cartesian lift $\chi\colon F'\To F$: a natural transformation with $P\chi = P\pi$ whose components are $P$-cartesian arrows. Then we may factor the cone $\pi$ through the $P$-cartesian lift $\chi$ to give a cone $\pi'\colon \Delta_{E}\To F'$ in the fibre $\eE_{PE}$ of $P$ over $PE$, which displays $E$ as a limit of the diagram $F'$ in the fibre $\eE_{PE}$.
\end{lem}

\begin{proof}
  Given a cone $\alpha\colon \Delta_{E'}\To F'$ in $\eE_{PE}$ we have that the universal property of the limiting cone $\pi\colon\Delta_E\To F$ applied to the composite cone $\chi\cdot\alpha\colon\Delta_{E'} \To F$ induces a unique arrow $f\colon E'\to E$ for which $\chi\cdot\alpha = \pi\cdot\Delta_f $. Now, on applying $P$ to that defining equation for $f$, we obtain the equation $P\pi\cdot\Delta_{Pf} = P\chi\cdot\Delta_{Pf} = P(\chi\cdot\Delta_f) = P(\chi\cdot\alpha) = P\chi\cdot P\alpha = P\chi = P\pi$ in which the penultimate step follows because $\alpha$ is a cone in the fibre $\eE_{PE}$; so the uniqueness portion of the universal property of the limit cone $P\pi$ implies that $Pf = \id_{PE}$ and thus that $f$ is an arrow of the fibre $\eE_{PE}$. It is clear now that $f$ is the unique such arrow, thereby completing our verification that the cone $\pi'\colon \Delta_{E}\To F'$ has the required universal property in the fibre $\eE_{PE}$. 
\end{proof}

We'll apply this result to the codomain-projection functor $\cod \colon \eK^\cattwo \to \eK$. Since $\eK$ has pullbacks of isofibrations, this functor is a Grothendieck fibration, with $\cod$-cartesian arrows in $\eK^\cattwo$ given by pullback squares.

\begin{prop}\label{prop:K^2-cosmos-limits}
  Suppose that $F\colon\eD\to\eK^\cattwo$ is a diagram of one of the conical cosmological limit types described in \ref{defn:cosmos}\ref{defn:cosmos:a} with respect to the class of isofibrations specified in the statement of Proposition~\ref{prop:isofib-cosmoi}. Then $\eK^\cattwo$ is closed in the category of all arrows in $\eK$ under the limit for this diagram and the various components of the limiting cone that are required to be isofibrations under \ref{defn:cosmos}\ref{defn:cosmos:b} are members of the specified class of isofibrations in $\eK^\cattwo$.
\end{prop}

\begin{proof}
Recall that the limit of the given diagram is created in the category $\bar\eK^\cattwo$ of all arrows in $\eK$ by the pair of projections $\dom,\cod\colon\bar\eK^\cattwo\to \eK$. Suppose that this limit $\lim(F)$ is displayed by a limit cone $\pi\colon \Delta_{\lim(F)}\To F$. Let $B\defeq\cod(\lim(F))$ and consider the projected limit cone $\cod\pi\colon\Delta_B\To \cod F$ which admits a cartesian lift $\chi\colon F'\To F$ along the projection $\cod\colon\eK^\cattwo\to\eK$. Since $\cod$-cartesian arrows are pullback squares, the diagram $F'$ is formed by pulling back the arrows in the diagram $F$ along the codomain components of the limit cone. We may factor $\pi$ through $\chi$ to give a cone $\pi'\colon\Delta_{\lim(F)}\To F'$ in the fibre $\bar\eK_{/B}$, and on application of Lemma \ref{lem:fibration.limit.fact}, we see that this is a limit cone in that extended slice.

  Observe now that if $k \colon d' \to d$ is an arrow in $\eD$ then the arrow $F'k$ in the fibre $\eK_{/B}$ is otherwise obtained as a pullback of the vertical factor of $Fk$, which lies in the slice over $\cod(Fd)$, along the projection $\pi_d\colon B\to \cod(Fd)$. These pullback functors between slice $\infty$-cosmoi preserve isofibrations, so it follows that if $Fk$ is an isofibration of the form described in Proposition~\ref{prop:isofib-cosmoi} then its vertical factor is an isofibration in the slice $\eK_{/\cod(Fd)}$ as is its pullback $F'k$ in $\eK_{/B}$. Consequently $F'$ is a cosmological limit type diagram in the sliced $\infty$-cosmos $\eK_{/B}$ because $F$ is a diagram of that kind relative to the class of isofibrations given in the statement Proposition \ref{prop:isofib-cosmoi}. We know, however, that the slice $\eK_{/B}$ is closed in the extended slice $\bar\eK_{/B}$ under the cosmological limits, so it follows that $\lim(F)$ is an object of $\eK_{/B}$, and thus of $\eK^\cattwo$, and that those components of the limit cone $\pi'\colon\Delta_{\lim(F)}\To F'$ that are expected to be isofibrations by axiom \ref{defn:cosmos}\ref{defn:cosmos:b} are such in the slice $\eK_{/B}$.

  Consequently, the limit cone $\pi\colon \Delta_{\lim(F)}\To F$ restricts to the subcategory $\eK^\cattwo\subset \bar\eK^\cattwo$. Suppose further that $d$ is an object of $\eD$ for which axiom \ref{defn:cosmos}\ref{defn:cosmos:b} asks that the corresponding component $\pi_d$ of this limiting cone is an isofibration. That projection factors as a composite $\pi_d=\chi_d\pi'_d$ in which $\pi'_d$ is an isofibration in the slice $\eK_{/B}$, by the comment at then end of the last paragraph, and is thus an isofibration in $\eK^\cattwo$. Furthermore, by definition $\chi_d$ is a $\cod$-cartesian lift of the limit projection $\cod \pi_c\colon B\to \cod (Fd)$ which is an isofibration in the $\infty$-cosmos $\eK$; however, any such $\cod$-cartesian lift of an isofibration in $\eK$ is an isofibration in $\eK^\cattwo$. Consequently their composite $\pi_d=\chi_d\pi'_d$ is an isofibration in $\eK^\cattwo$ as required.
\end{proof}

\begin{prop}[the $\infty$-cosmos of trivial fibrations]\label{prop:K^2_eq-cosmos} Let $\eK$ be an $\infty$-cosmos.
\begin{enumerate}[label=(\roman*)]
\item For any $\infty$-category $A$ in $\eK$, the full simplicial subcategory $\eK^{\simeq}_{/A}\inc\eK_{/A}$ spanned by the trivial fibrations in $\eK$ is an $\infty$-cosmos, with limits, isofibrations, equivalences, and trivial fibrations  created by the inclusion, and with all functors in $\eK^{\simeq}_{/A}$ equivalences.
\item  The full simplicial subcategory  $\eK^\cattwo_\simeq \inc\eK^\cattwo$ spanned by the trivial fibrations defines an $\infty$-cosmos, with limits, isofibrations, equivalences, and trivial fibrations created by the inclusion.
\end{enumerate}
\end{prop}
\begin{proof}
The full simplicial subcategory $\eK^{\simeq}_{/A}$ of the slice $\eK_{/A}$ is easily seen to be closed in $\eK_{/A}$ under the limit types named in the $\infty$-cosmos axioms. Consequently $\eK^{\simeq}_{/A}$ inherits an $\infty$-cosmos structure from $\eK_{/A}$, and in there all functors are equivalences by the 2-of-3 property in $\eK$. Furthermore this result allows us to apply the proof of Proposition~\ref{prop:K^2-cosmos-limits} to show that the full simplicial subcategory $\eK^\cattwo_\simeq$ of $\eK^\cattwo$ spanned by the trivial fibrations is also closed under the limits named in the $\infty$-cosmos axioms and thus inherits an $\infty$-cosmos structure from $\eK^\cattwo$. 
\end{proof}

\begin{defn}\label{defn:functor-inf-cosmoi} A \emph{cosmological functor} is a simplicial functor that preserves the classes of isofibrations and each of the limits specified in Definition~\ref{defn:cosmos}\ref{defn:cosmos:a}.
\end{defn}

\begin{ex}
Cosmological functors include:
\begin{enumerate}[label=(\roman*)]
\item the representable functor $\Fun_{\eK}(X,-) \colon\eK \to \qCat$ for any object $X \in \eK$;
\item as a special case, the \emph{underlying quasi-category functor} $\Fun_{\eK}(1,-) \colon \eK \to \qCat$; 
\item the simplicial cotensor $(-)^U \colon \eK \to \eK$ with any simplicial set $U$; 
\item the pullback functor $f^* \colon \eK_{/B} \to \eK_{/A}$ for any functor $f \colon A \to B \in \eK$; and
\item the domain and codomain projections $\dom,\cod\colon \eK^\cattwo\to\eK$ and $\dom,\cod\colon\eK^\cattwo_\simeq\to\eK$.
\end{enumerate}
among others.
\end{ex}

We refer to the $n$-simplices of the functor space $\Fun_{\eK}(A,B)$ of an $\infty$-cosmos as $n$-\emph{arrows}. Each $\infty$-cosmos has an underlying 1-category whose objects are the $\infty$-categories of that $\infty$-cosmos and whose morphisms, which we call $\infty$-\emph{functors} or more often simply \emph{functors}, are the 0-arrows (the vertices) of the functor spaces. Any $\infty$-cosmos has as quotient \emph{homotopy 2-category} that is built from this same underlying 1-category by adding 2-cells, which are represented by 1-arrows in the corresponding functor space up to a suitable equivalence relation.

\begin{defn}[the homotopy 2-category of an $\infty$-cosmos]\label{defn:hty-2-cat} The \emph{homotopy 2-category} of an $\infty$-cosmos $\eK$ is a strict 2-category $\ho_*\eK$ so that 
\begin{itemize}
\item the objects of $\ho_*\eK$ are the objects of $\eK$, i.e., the $\infty$-categories;
\item the 1-cells $f \colon A \to B$ of $\ho_*\eK$ are the vertices $f \in \Fun_{\eK}(A,B)$ in the functor spaces of $\eK$, i.e., the $\infty$-functors;
\item a 2-cell  $\xymatrix{ A \ar@/^2ex/[r]^f \ar@/_2ex/[r]_g \ar@{}[r]|{\Downarrow\alpha}& B}$ in $\ho_*\eK$ is represented by a 1-arrow $\alpha \colon f \to g \in \Fun_{\eK}(A,B)$, where a parallel pair of 1-arrows in $\Fun_{\eK}(A,B)$ represent the same 2-cell if and only if they bound a 2-arrow whose remaining outer face is degenerate.
\end{itemize}
Put concisely, the homotopy 2-category is the 2-category $\ho_*\eK$ defined by applying the homotopy category functor $\ho \colon \qCat \to \eop{Cat}$, the left adjoint to the nerve embedding $\eop{Cat}\inc \qCat$, to the functor spaces of the $\infty$-cosmos. We write $\hFun_{\eK}(A,B)$ for the hom-category of arrows from $A$ to $B$ in the 2-category $\ho_*\eK$.
\end{defn}

Proposition \refIV{prop:equiv.are.weak.equiv} proves that the equivalences between $\infty$-categories admit another important characterisation: they are precisely the equivalences in the homotopy 2-category of the $\infty$-cosmos. The upshot is that equivalence-invariant 2-categorical constructions are appropriately ``homotopical,'' characterising $\infty$-categories up to equivalence, and that we may use the term ``equivalence'' unambiguously in both the quasi-categorically enriched and 2-categorical contexts.

Similarly, the reason we have chosen the term ``isofibrations'' for the designated class of $\infty$-functors $A \tfib B$ is because these maps define isofibrations in the homotopy 2-category.  An \emph{isofibration} in a 2-category is a 1-cell that has a lifting property for isomorphisms with one chosen endpoint; see \refIV{rec:trivial-fibration} and \refIV{lem:isofib.are.representably.so}.

\begin{defn}[dual $\infty$-cosmoi]\label{defn:dual-cosmoi} For any $\infty$-cosmos $\eK$, write $\eK\co$ for
 the $\infty$-cosmos with the same objects but with the opposite functor spaces
\[ \Fun_{\eK\co}(A,B) \defeq\Fun_{\eK}(A,B)\op\] defined by applying the product preserving dual
  functor $(-)\op\colon \sSet\to \sSet$, which in turn is obtained by precomposing with the functor $(-)^\circ \colon \Del \to \Del$ that reverses the ordering of the elements in each ordinal $[n]$.  The homotopy 2-category of $\eK\co$ is the ``co'' dual of the homotopy 2-category of $\eK$, reversing the 2-cells but not the 1-cells. This explains our notation. 
 \end{defn}

  Any quasi-categorically enriched category $\eK$ also has an opposite category $\eK\op$ with the same objects but with
\[ \Fun_{\eK\op}(A,B) \defeq \Fun_{\eK}(B,A).\] However, if $\eK$ is an $\infty$-cosmos then $\eK\op$ need not be an $\infty$-cosmos; see Observation \refIV{obs:duals-of-cosmoi} however.

\subsection{Groupoidal objects and the \texorpdfstring{$(\infty,1)$}{(infty,1)}-core of an \texorpdfstring{$\infty$}{infty}-cosmos}\label{ssec:core}

An object $E$ in an $\infty$-cosmos $\eK$ is \emph{groupoidal\/} if the following equivalent conditions are satisfied:

\begin{lem}\label{lem:groupoidal-object}
For any object $E$ in an $\infty$-cosmos $\eK$ the following are equivalent:
\begin{enumerate}[label=(\roman*)]
\item\label{itm:groupoidal-object:i} $E$ is a groupoidal object in the homotopy 2-category $\ho_*\eK$, that is,  every 2-cell with codomain $E$ is invertible.
\item\label{itm:groupoidal-object:ii} For each $X \in \eK$, the hom-category $\hFun_{\eK}(X,E)$ is a groupoid. 
\item\label{itm:groupoidal-object:iii}  For each $X \in \eK$, the functor space $\Fun_{\eK}(X,E)$ is a Kan complex.
\item\label{itm:groupoidal-object:iv} The isofibration $E^\iso \tfib E^\cattwo$, induced by the inclusion of the walking arrow into the walking isomorphism $\cattwo\inc\iso$, is a trivial fibration.
\end{enumerate}
\end{lem}
\begin{proof}
Here \ref{itm:groupoidal-object:ii} is an unpacking of \ref{itm:groupoidal-object:i}. The equivalence of \ref{itm:groupoidal-object:ii} and \ref{itm:groupoidal-object:iii} is a well-known result of Joyal \cite[1.4]{Joyal:2002:QuasiCategories}. Condition \ref{itm:groupoidal-object:iv} is equivalent to the assertion that $\Fun_{\eK}(X,E)^\iso \tfib \Fun_{\eK}(X,E)^\cattwo$ is a trivial fibration between quasi-categories for all $X$. If this is a trivial fibration, then surjectivity on vertices implies that every 1-simplex in $\Fun_{\eK}(X,E)$ is an isomorphism, proving \ref{itm:groupoidal-object:iii}. As $\cattwo \inc\iso$ is a weak homotopy equivalence, \ref{itm:groupoidal-object:iii} implies  \ref{itm:groupoidal-object:iv}.
\end{proof}

A common theme is that $\infty$-categorical definitions that admit ``internal'' characterisations in an $\infty$-cosmos are also preserved by cosmological functors. For instance:

\begin{cor}\label{cor:groupoidal-preservation} Cosmological functors preserve groupoidal objects.
\end{cor}
\begin{proof}
A cosmological functor preserves simplicial cotensors and trivial fibrations, so this follows directly from the characterisation of Lemma \ref{lem:groupoidal-object}\ref{itm:groupoidal-object:iv}.
\end{proof}

\begin{prop}[the $\infty$-cosmos of groupoidal objects]\label{prop:gpdal-infty-cosmos}
The full simplicial subcategory $\eK\gr\inc\eK$ of groupoidal objects in an $\infty$-cosmos $\eK$ defines an $\infty$-cosmos  with limits, isofibrations, equivalences, and trivial fibrations  created by the inclusion. Moreover, all of the functor spaces in $\eK\gr$ are Kan complexes.
\end{prop}
\begin{proof}
To prove the claimed result, we need only demonstrate that the full simplicial subcategory $\eK\gr\inc\eK$ is closed under the cosmological limit types. For this observe from Lemma \ref{lem:groupoidal-object}\ref{itm:groupoidal-object:iv} that the category $\eK\gr$ is defined by the pullback
\[
\xymatrix{
\eK\gr \ar@{^(->}[d] \ar[r] \pbexcursion & \eK^\cattwo_\simeq \ar@{^(->}[d] \\ \eK \ar[r]_E & \eK^\cattwo
}\]
where $E \colon \eK \to \eK^\cattwo$ is the simplicial functor defined on objects by $A \mapsto (A^{\iso} \tfib A^\cattwo)$. The functor $E \colon \eK \to \eK^\cattwo$ is cosmological, preserving the cosmological limits of \ref{defn:cosmos}\ref{defn:cosmos:a}, and by Proposition \ref{prop:K^2_eq-cosmos}, the inclusion $\eK^\cattwo_\simeq\inc\eK^\cattwo$ creates the cosmological limits. Both vertical inclusions are \emph{replete}---meaning that any object in $\eK$ that is equivalent to a groupoidal object is itself groupoidal---and fully faithful, so it follows from the fact that $\eK^\cattwo_\simeq\inc\eK^\cattwo$ creates the cosmological limits and $E$ preserves them that $\eK\gr\inc\eK$ also creates the cosmological limits. Thus $\eK\gr$ inherits an $\infty$-cosmos structure from $\eK$ as claimed.
\end{proof}

\begin{ex}[the $\infty$-cosmos of Kan complexes]\label{ex:Kan-infty-cosmos}
  A quasi-category is a groupoidal object in the $\infty$-cosmos $\qCat$ if and only if it is a Kan complex. This follows from the fact that a natural transformation of functors of quasi-categories is an isomorphism iff its components are isomorphisms and that a quasi-category is a Kan complex iff all of its edges are invertible. So the subcategory $\qCat\gr$ spanned by the groupoidal objects in the $\infty$-cosmos of quasi-categories is simply the enriched category $\Kan$ of Kan complexes, which we see by Proposition \ref{prop:gpdal-infty-cosmos} defines an $\infty$-cosmos.
\end{ex}

Example \ref{ex:Kan-infty-cosmos} encourages us to think of the groupoidal objects in an $\infty$-cosmos $\eK$ as being the \emph{spaces\/} in that universe. 

An $\infty$-cosmos is a type of $(\infty,2)$-category since it is a category enriched over a model of $(\infty,1)$-categories. We now introduce the \emph{$(\infty,1)$-categorical core} of an $\infty$-cosmos:

\begin{defn}[groupoid cores of a quasi-category]\label{defn:gpd-cores}
  We use the notation $g\qA$ to denote the \emph{groupoid core\/} of a quasi-category $\qA$, this being the maximal Kan complex it contains. More explicitly, $g\qA$ is the largest simplicial subset of $\qA$ spanning its invertible edges.
  \end{defn}
  
   Functors of quasi-categories preserve isomorphisms, so a functor $f\colon\qA\to \qB$ restricts to a functor $f\colon g\qA\to g\qB$; in this way the groupoidal core construction acts functorially on the underlying category of $\qCat$ and it is indeed right adjoint $g \colon \qCat \to \Kan$ to the inclusion $\Kan\inc\qCat$ in the unenriched sense. In particular it preserves finite products, so we may apply it to the functor spaces of a quasi-categorically enriched category $\eK$ to construct a Kan-complex-enriched subcategory that we now introduce:
   
   \begin{defn}[$(\infty,1)$-core of an $\infty$-cosmos]\label{defn:infinity,1-core}
   For any $\infty$-cosmos $\eK$, write   $g_*\eK \subseteq \eK$ for the subcategory with the same objects and with homs defined to be the groupoid cores of the functor spaces of $\eK$. We refer to $g_*\eK$ as the \emph{$(\infty,1)$-core} of $\eK$ and think of it as being the core $(\infty,1)$-category inside this $(\infty,2)$-category. 
\end{defn}

The construction of Definition \ref{defn:infinity,1-core} applies equally to construct the $(\infty,1)$-core of any quasi-categorically enriched category, whether or not it defines an $\infty$-cosmos.

\begin{rmk}\label{rmk:core-not-enriched}
  We  note that the groupoid core functor does not admit a simplicial enrichment with respect to the usual enrichment of $\qCat$ over itself. To understand this failure, observe that a natural transformation between functors of quasi-categories will only restrict to groupoid cores if each of its components is invertible. 
  
  This does, however, lead us to the conclusion that, so long as we are willing to consider only invertible natural transformations, the groupoid core functor does admit a canonical enrichment to a simplicial functor $g\colon g_*\qCat\to\Kan$ and this is right adjoint to the inclusion $\Kan\inc g_*\qCat$ in the simplicially enriched sense:
  \[ \adjdisplay {}-| g : g_*\qCat -> \Kan.\]
\end{rmk}



\section{Limits and colimits in an \texorpdfstring{$\infty$}{infinity}-category}\label{sec:limits}

In this section, we review the relevant aspects of the \emph{synthetic} theory of limits and colimits of diagrams valued in an $\infty$-category, and connect them back to the more familiar ``analytic'' definitions of Joyal in the $\infty$-cosmos of quasi-categories. The presentation we give here is much sparser than the full story presented in \S\refI{sec:limits}, as we save ourselves time by presenting only the definition that we will require in \S\ref{sec:complete}, rather than its many equivalent forms. Limits and colimits are defined in \S\ref{ssec:limits}  using \emph{arrow} and \emph{comma} $\infty$-categories, which we first review in  \S\ref{ssec:commas}.

\subsection{Arrow and comma constructions}\label{ssec:commas}

The axioms of an $\infty$-cosmos permit us to construct \emph{arrow} and \emph{comma} $\infty$-categories as particular simplicially enriched limits. In this section, we briefly review these constructions and their associated homotopical properties.

\begin{defn}[arrow $\infty$-categories] For any $\infty$-category $A$, the simplicial cotensor 
\[ \xymatrix@C=30pt{ A^\cattwo \defeq A^{\Del^1} \ar@{->>}[r]^-{(p_1,p_0)} & {A^{\boundary\Delta^1}} \cong A \times A}\] defines the \emph{arrow $\infty$-category} $A^\cattwo$, equipped with an isofibration $(p_1,p_0)\colon A^\cattwo \tfib A \times A$, where $p_1 \colon A^\cattwo \tfib A$ denotes the codomain projection and $p_0 \colon A^\cattwo \tfib A$ denotes the domain projection.
\end{defn}

Using the simplicially enriched pullbacks of isofibrations that exist by virtue of axiom \ref{defn:cosmos}\ref{defn:cosmos:a}, arrow $\infty$-categories can be used to define a general \emph{comma $\infty$-category} associated to a cospan of functors.

\begin{defn}[comma $\infty$-categories]\label{defn:comma} Any pair of functors  $f\colon B\to A$ and $g\colon C\to A$ in an $\infty$-cosmos $\eK$ has an associated \emph{comma $\infty$-category}, constructed by the following pullback in $\eK$:
\[
    \xymatrix@=2.5em{
      {f\comma g}\pbexcursion \ar[r]\ar@{->>}[d]_{(p_1,p_0)} &
      {A^\cattwo} \ar@{->>}[d]^{(p_1,p_0)} \\
      {C\times B} \ar[r]_-{g\times f} & {A\times A}
    }
\]
 Note that, by construction, the map $(p_1,p_0) \colon f \comma g \tfib C \times B$ is an isofibration.
\end{defn}

\begin{ex}[comma and slice quasi-categories]\label{ex:comma-vs-slice} An \emph{object} (or perhaps better \emph{element}) of an $\infty$-category can be represented by a functor $a \colon 1 \to A$. When paired with the identity, the comma construction gives rise to a pair of isofibrations
\[ p_1 \colon a \comma A \tfib A \qquad \mathrm{and} \qquad p_0 \colon A \comma a \tfib A\]
which should be thought of as the codomain and domain projections from the sub $\infty$-categories of the arrow $\infty$-category consisting of those arrows whose source or target, respectively, is $a$.

In the $\infty$-cosmos of quasi-categories, these comma quasi-categories are equivalent over $\qA$ to the \emph{slice quasi-categories} introduced by Joyal
\[ a \comma \qA \simeq \slicel{a}{\qA} \qquad \mathrm{and} \qquad \qA \comma a \simeq \slicer{\qA}{a}.\] See Lemma \refI{lem:slice-equiv-comma} for a proof.
\end{ex}

\begin{prop}[{maps between commas, \refI{lem:comma-obj-maps}}]\label{prop:trans-comma}
A natural transformation of co-spans on the left of the following display gives rise to the diagram of pullbacks on the right
\begin{equation*}
  \vcenter{
    \xymatrix@R=1.5em@C=2.5em{
      {C}\ar[r]^{g}\ar[d]_{c} & {A}\ar[d]_{a} & {B}\ar[l]_{f}\ar[d]^{b} \\
      {C'}\ar[r]_{g'} & {A'} & {B'}\ar[l]^{f'}
    }}\mkern30mu\rightsquigarrow\mkern30mu
  \vcenter{
    \xymatrix@=1em{
      {f\comma g}\pbexcursion\ar[rr]\ar@{->>}[dd] &&
      {A^{\cattwo}}\ar@{->>}[dd]|!{[rd];[ld]}\hole & \\
      & {f'\comma g'}\pbexcursion\ar[rr]\ar@{->>}[dd] &&
      {(A')^{\cattwo}}\ar@{->>}[dd]\\
      {C\times B}\ar[rr]|!{[ru];[rd]}\hole_(0.65){g\times f} && {A\times A} & \\
      & {C'\times B'}\ar[rr]_{g'\times f'} &&
      {A'\times A'}
      \ar@{-->} "1,1";"2,2"
      \ar "3,1";"4,2"_-{c\times b}
      \ar "1,3";"2,4" ^{a^{\cattwo}}
      \ar "3,3";"4,4" ^{a\times a}
    }}
\end{equation*}
in which the uniquely induced dashed map completing the commutative cube is denoted \[\comma(b,a,c) \colon f \comma g \to f' \comma g'.\] Moreover,  $\comma(b,a,c)$ is an isofibration (resp.~trivial fibration, equivalence) whenever the components $a$, $b$ and $c$ are all maps of that kind.
\end{prop}
\begin{proof}
The result of Lemma \refI{lem:comma-obj-maps} is stated only for quasi-categories but its proof applies in any $\infty$-cosmos. Alternatively, the comma construction is a particular kind of \emph{flexible weighted limit}, using a notion to be introduced in Definitions \ref{defn:simp-weight} and \ref{defn:flexible-weight}, where the weight is given by the cospan of simplicial sets
\[ \xymatrix@R=0em@C=3em{ \Del^0 \ar@{^(->}[r]^{\fbv{1}} & \Del^1 & \Del^0 \ar@{_(->}[l]_{\fbv{0}}}\]
so we may apply Proposition~\ref{prop:flexible-weights-are-htpical}\ref{itm:flexible-htpical}, to show that $\comma(b,a,c)$ is an isofibration (resp. trivial fibration, equivalence) whenever the components $a$, $b$ and $c$ are all maps of that kind.
\end{proof}

\subsection{Limits and colimits in an \texorpdfstring{$\infty$}{infinity}-category}\label{ssec:limits}

\begin{defn}[{terminal elements, \refI{prop:terminalconverse}}]\label{defn:terminal} An element $t \colon 1 \to A$ defines a \emph{terminal element} of the $\infty$-category $A$ if and only if  the domain projection $p_0 \colon A \comma t \tfib A$ is a trivial fibration. Dually, $i \colon 1 \to A$ defines an \emph{initial element} of the $\infty$-category $A$ if and only if the codomain projection $p_1 \colon i \comma A \tfib A$ is a trivial fibration.
\end{defn}

Equivalently, terminal and initial elements, respectively, define right and left adjoints to the unique functor $!\colon A \to 1$; see \S\refI{subsec:terminal}.

\begin{ex}[terminal objects in a quasi-category] If $\qA$ is a quasi-category, the equivalence $\qA \comma t \simeq \slicer{\qA}{t}$ of Example \ref{ex:comma-vs-slice} provides another characterisation of terminal objects, the original one due to Joyal. The defining lifting property for trivial fibrations of quasi-categories
\[
\xymatrix{ \partial\Del^n\ar[r] \ar@{^(->}[d] & \slicer{\qA}{t} \ar@{->>}[d]^{p_0} \\ \Del^n \ar[r] \ar@{-->}[ur] & \qA}
\] for $n \geq 0$ transposes to
\[ \xymatrix{ \Del^{\fbv{n+1}} \ar[r] \ar@/^2ex/[rr]^t& \partial\Del^{n+1}\ar@{^(->}[d]  \ar[r] & \qA \\ & \Del^{n+1} \ar@{-->}[ur]}\]
i.e., a vertex $t \in \qA$ is terminal if and only if any sphere in $\qA$ whose final vertex is $t$ has a filler.
\end{ex}

Via the nerve embedding, diagrams indexed by small categories are among the diagrams indexed by small simplicial sets. The simplicial cotensors of axiom \ref{defn:cosmos}\ref{defn:cosmos:a} are used to define $\infty$-categories of diagrams.

\begin{defn}[diagram $\infty$-categories]\label{defn:diagram-cats} If $J$ is a (small) simplicial set and $A$ is an $\infty$-category, then the $\infty$-category $A^J$ is naturally thought of as being the \emph{$\infty$-category of $J$-indexed diagrams in $A$}.
\end{defn}

For any $\infty$-category $A$ and simplicial set $J$, the \emph{constant diagram functor} $\Delta \colon A \to A^J$ is constructed by applying the contravariant functor $A^{(-)}$ to the unique simplicial map $J\to \Del^0$.

\begin{defn}[$\infty$-categories of cones]\label{defn:cones} For any diagram $d \colon 1 \to A^J$ of shape $J$ in an $\infty$-category $A$, the \emph{$\infty$-category of cones over $d$} is the comma $\infty$-category $p_0 \colon \Delta \comma d \tfib A$ formed by the pullback 
\[
    \xymatrix@=2.5em{
      {\Delta \comma d}\pbexcursion \ar[r]\ar@{->>}[d]_{(p_1,p_0)} &
      {A^{J \times \cattwo}} \ar@{->>}[d]^{(p_1,p_0)} \\
      {1 \times A} \ar[r]_-{d\times \Delta} & {A^J\times A^J}
    }
\]
Dually, the \emph{$\infty$-category of cones under $d$} is the comma $\infty$-category $p_1 \colon d \comma \Delta \tfib A$.
\end{defn}

\begin{defn}[{limits and colimits in an $\infty$-category, \refI{prop:limits.as.terminal.objects}}]\label{defn:general-limits}
For a diagram $d \colon 1 \to A^J$ of shape $J$ in an $\infty$-category $A$, the following are equivalent and define what it means for $d$ to \emph{have a limit in $A$}.
\begin{enumerate}[label=(\roman*)]
\item The $\infty$-category of cones $\Delta \comma d$ has a terminal element $\lambda \colon 1 \to \Delta \comma d$. The element $\ell\defeq p_0 \lambda \colon 1 \to A$ is then the \emph{limit object}, while $\lambda$ is the \emph{limit cone}.
\item There exists an element $\ell \colon 1 \to A$ and an equivalence $A \comma \ell \simeq \Delta \comma d$ over $A$.
\end{enumerate}
Dually, $d \colon 1 \to A^J$ \emph{has a colimit} if and only if the $\infty$-category of cones $d \comma \Delta$ has an initial element, or equivalently, if there exists an element $\ell \colon 1 \to A$ and an equivalence $\ell \comma A \simeq d \comma \Delta$ over $A$.
\end{defn}

\begin{ex}[limits in quasi-categories]\label{ex:limits-in-quasi-categories} In the $\infty$-cosmos of quasi-categories, the quasi-category of cones $\Delta \comma d$ over a diagram $d \colon 1 \to \qA^J$ is equivalent over $\qA$
 to the Joyal slice $\slicer{\qA}{d}$ of $\qA$ over the transposed map $d \colon J \to \qA$; see Lemma \refI{lem:cone-equiv-fatcone}. For any $\ell \colon 1 \to \qA$, it is easy to see that $\qA \comma \ell$ admits a terminal object, induced by the identity 2-cell at $\ell$; see Example \refI{ex:slice-terminal}. The equivalence $\qA\comma \ell \simeq \Delta\comma d$ of Definition \ref{defn:general-limits} then implies that the limit cones defines a terminal object in $\Delta \comma d \simeq \slicer{\qA}{d}$. Proposition \refI{prop:limits.are.limits} proves the converse, demonstrating that the general notion of limits introduced here specialises to recapture precisely the notion of limits and colimits for quasi-categories introduced by Joyal in  \cite[4.5]{Joyal:2002:QuasiCategories}.
\end{ex}



\section{Flexible homotopy limits in an \texorpdfstring{$\infty$}{infinity}-cosmos and its \texorpdfstring{$(\infty,1)$}{(infinity,1)}-categorical core}\label{sec:flexible}

In this section, we  demonstrate that an $\infty$-cosmos permits a large variety of limit constructions, namely those with \emph{flexible weights}. These flexible weighted limits are introduced in \S\ref{ssec:flexible}. The weighted limit for a fixed flexible weight is defined up to isomorphism but the construction is appropriately homotopical, as we prove in Proposition \ref{prop:flexible-weights-are-htpical}, our main result. This leads us to consider a weaker notion of \emph{homotopy} weighted limit in an $\infty$-cosmos, whose defining universal property expresses an equivalence, rather than an isomorphism of quasi-categories. The reason for our interest in this weaker notion is revealed in \S\ref{ssec:htpy-limit}, where we prove that flexibly weighted homotopy limits in the groupoidal core $g_*\eK$ of an $\infty$-cosmos are computed as limits in $\eK$ weighted by the pointwise Kan complex replacement of the original flexible weight.

\subsection{Flexible homotopy limits}\label{ssec:flexible}

The basic simplicially-enriched limit notions enumerated in axiom \ref{defn:cosmos}\ref{defn:cosmos:a} imply that an $\infty$-cosmos $\eK$ possesses a much larger class of weighted limits, which we now describe. A more comprehensive introduction to the theory of weighted limits in  simplicially enriched categories  can be found in Chapter~3 of Kelly~\cite{kelly:ect} or Chapter~7 of Riehl~\cite{Riehl:2014kx}.  We'll apply these notions in an arbitrary simplicially enriched category $\eC$ with hom-spaces denoted by $\Map_{\eC}(X,Y)$.
 
\begin{defn}\label{defn:simp-weight}
A \emph{weight} for a diagram indexed by a small simplicial category $\eA$ is a simplicial functor $W \colon \eA \to \SSet$. A $W$-\emph{cone} over a diagram $F \colon \eA \to \eC$ in a simplicial category $\eC$ is comprised of an object $L \in \eC$ together with a simplicial natural transformation $\Lambda \colon W \to \Map_{\eC}(L,F-)$. Such a cone displays $L$ as a $W$-\emph{weighted limit of} $F$ if and only if for all $X \in \eC$ the simplicial map
  \begin{equation}\label{eq:weighted-UP}
    \xymatrix@R=0em@C=5em{
      \Map_{\eC}(X,L) \ar[r]^-{\cong} & \Map_{\SSet^{\eA}}(W,\Map_{\eC}(X,F-))
    }
  \end{equation}
  given by post-composition with $\Lambda$ is an isomorphism, in which case the limit object $L$ is typically denoted by $\wlim{W}{F}_{\eA}$ or simply $\wlim{W}{F}$. In this notation, the universal property \eqref{eq:weighted-UP} of the weighted limit asserts an isomorphism
  \[ 
      \xymatrix@R=0em@C=5em{
      \Map_{\eC}(X,\wlim{W}{F}_{\eA}) \ar[r]^-{\cong} & \{W,\Map_{\eC}(X,F-)\}_{\eA}.
    }
    \]
\end{defn}

The dual notion of \emph{weighted colimit} is definable for any weight $W \colon \eA \to \SSet$ and diagram $F \colon \eA\op \to \eC$; note in the colimit case that the weight and the diagram have contrasting variance. A $W$-\emph{cone} under a diagram $F$ is comprised of an object $C \in \eC$ together with a simplicial natural transformation $\Lambda \colon W \to \Map_{\eC}(F-,C)$ and displays $C$ as a $W$-\emph{weighted colimit of} $F$ if and only if for all $X \in \eC$ the simplicial map
  \[ 
      \xymatrix@R=0em@C=5em{
      \Map_{\eC}(C,X) \ar[r]^-{\cong} & \{W,\Map_{\eC}(F-,X)\}_{\eA}.
    }
    \]
    is an isomorphism. In this case, the colimit object $C$ is typically denoted by $W\wcolim_{\eA}F$ of $W \wcolim F$.

\begin{obs}\label{obs:lim-cont-in-weight}
  Suppose that $\eA$ and $\eB$ are small simplicial categories, that $W\colon\eB\op\to\SSet^{\eA}$ is a diagram of weights, and that $U\colon\eB \to \SSet$ is another weight. Then the weighted colimit $U\wcolim_{\eB} W$ exists in $\SSet^{\eA}$ and is computed level-wise in $\SSet$ by the coend formula:
  \begin{equation*}
    U\wcolim_{\eB} W \cong \int^{B\in \eB} UB\times W(B,-)
  \end{equation*}
  Now suppose that $F\colon\eA\to \eC$ is a diagram and that for each object $B\in\eB$ the weighted limit $\wlim{W(B,-)}{F}_{\eA}$ exists in $\eC$. Then these limits assemble into a diagram $\wlim{W}{F}_{\eA}\colon \eB\to \eC$ and there exists an isomorphism
  \begin{equation*}
    \wlim{U\wcolim_{\eB} W}{F}_{\eA} \cong \wlim{U}{\wlim{W}{F}_{\eA}}_{\eB}
  \end{equation*}
  when these limits exist. Indeed, the limit on one side of this isomorphism exists if and only if the limit on the other side exists. This result expresses the sense in which the weighted limit construction is cocontinuous in its weight.
\end{obs}

\begin{defn}[flexible weights and projective cell complexes]\label{defn:flexible-weight}
  For a small simplicial category $\eA$ and pair of objects $[n] \in \Del$ and
  $A \in \eA$, the \emph{projective $n$-cell} associated with $A$ is the
  simplicial natural transformation:
  \[
    \boundary\Del^n\times \Map_{\eA}(A,-)\inc
    \Del^n\times \Map_{\eA}(A,-).
  \]
  We say that a monomorphism $V\inc W$ in $\SSet^{\eA}$ is a \emph{projective cell complex\/} if it may be expressed as a countable composite of pushouts of coproducts of projective cells. A weight $W\colon\eA \to \SSet$ is said to be a \emph{flexible weight\/} if the inclusion $\emptyset\inc W$ is a projective cell complex. 
\end{defn}

We shall let $\Flex^{\eA}$ denote the full simplicial subcategory of $\SSet^{\eA}$ spanned by the flexible weights. Adapting the arguments of Propositions \refII{prop:projwlims2} and \refII{prop:proj-wlim-homotopical} to the general $\infty$-cosmos context, it follows that flexible weighted limits exist in any $\infty$-cosmos admitting the limits axiomatised in \ref{defn:cosmos}\ref{defn:cosmos:a} and moreover such limits are appropriately ``homotopical'': 

\begin{prop}\label{prop:flexible-weights-are-htpical} Let $\eK$ be an $\infty$-cosmos and let $\eA$ be a small simplicial category.
  \begin{enumerate}[label=(\roman*)]
  \item\label{itm:flexible-exist} For any diagram $F \colon \eA\to\eK$ and flexible weight $W \colon \eA \to \SSet$, the weighted limit $\wlim{W}{F}$ exists in $\eK$.
  \item\label{itm:flexible-relative} For any diagram $F \colon \eA\to\eK$, flexible weights $W,W'\colon \eA\to\SSet$, and projective cell complex $i\colon W\inc W'$, the induced map $\wlim{i}{F}\colon\wlim{W'}{F}\to\wlim{W}{F}$ is an isofibration. It follows that we may extend the construction of these flexible limits to a simplicial bifunctor
  \begin{equation*}
    \xymatrix@R=0em@C=8em{
      {(\Flex^{\eA})\op\times\eK^{\eA}}\ar[r]^-{\wlim{-}{-}_{\eA}} & {\eK}
    }
  \end{equation*}
  which carries projective cell complexes to isofibrations.
  \item\label{itm:flexible-htpical} If $\kappa \colon F \To G$ is a simplicial natural transformation between two such diagrams whose components are equivalences, isofibrations, or trivial fibrations in $\eK$ and $W$ is a flexible weight, then the induced map
    \begin{equation*}
      \xymatrix@R=0em@C=6em{
        {\wlim{W}{F}}\ar[r]^-{\wlim{W}{\kappa}} & {\wlim{W}{G}}
      }
    \end{equation*}
    is an equivalence, isofibration, or trivial fibration (respectively) in $\eK$.
  \end{enumerate}
\end{prop}

\begin{proof}
We prove \ref{itm:flexible-exist} and \ref{itm:flexible-relative} simultaneously with the former serving as the base case for the latter. Note in the latter case, the hypothesis that $\wlim{W}{F}$ exists, and $W \inc W'$ is a projective cell complex will imply that $\wlim{W'}{F}$ exists, if we did not know this already.

We start by observing that $i\colon W\inc W'$ may be built up as a countable composite of pushouts of coproducts of projective cells and, by observation~\ref{obs:lim-cont-in-weight}, these are carried to the corresponding limit notions by $\{-,F\}$ whenever those limits exist in $\eK$. Note, by the Yoneda lemma, that for a simplicial set $U$ and object $A \in \eA$ the weighted limit $\{ U \times \Map_{\eA}(A,-), F\}$ exists and is isomorphic to the $U$-cotensor of the object $FA \in \eK$; consequently, the map $\wlim{\Del^n\times \Map_{\eA}(A,-)}{F}\to \wlim{\boundary\Del^n\times\Map_{\eA}(A,-)}{F}$ induced by the projective $n$-cell associated with $A$ is simply the isofibration $FA^{\Del^n} \tfib FA^{\boundary\Del^n}$. It follows that $\wlim{i}{F}\colon\wlim{W}{F}\to\wlim{W'}{F}$ exists and may be expressed as a countable inverse limit of pullbacks of products of isofibrations $FA^{\Del^n} \tfib FA^{\boundary\Del^n}$ so long as each one of those particular limits exist in $\eK$. This fact is, however, assured by Definition~\ref{defn:cosmos}\ref{defn:cosmos:b}, which implies that products of isofibrations exist and are again isofibrations, that in turn pullbacks of those exist and are again isofibrations, and that the sequence of those admits a countable inverse limit which is again an isofibration as posited. This completes our proof of~\ref{itm:flexible-exist} and \ref{itm:flexible-relative} as soon as we observe that the limit $\wlim{\emptyset}{F}$ weighted by the empty weight exists and is isomorphic to the terminal object of $\eK$.

For \ref{itm:flexible-htpical}, suppose that the natural transformation $\kappa\colon F\To G$ has the property that each of its components is an isofibration in $\eK$. Then it may equally well be regarded as being a single diagram $K\colon\eA\to\eK^{\cattwo}$ in the $\infty$-cosmos of isofibrations introduced in Proposition \ref{prop:isofib-cosmoi}. So we may apply the result of part~\ref{itm:flexible-exist} to show that the flexible weighted limit $\wlim{W}{K}$ exists in $\eK^\cattwo$. Furthermore we know, from Proposition~\ref{prop:K^2-cosmos-limits}, that the products, pullbacks and limits of towers used in that construction are jointly created by the projections $\dom,\cod\colon\eK^\cattwo\to\eK$ and so the same holds true for the flexible limit $\wlim{W}{K}$. In other words, the object $\wlim{W}{K}$ in $\eK^\cattwo$ is simply the induced map $\wlim{W}{\kappa}\colon\wlim{W}{F}\tfib\wlim{W}{G}$ in $\eK$ which is thus an isofibration as postulated. The corresponding result for trivial fibrations follows by applying the same argument in the $\infty$-cosmos of trivial fibrations $\eK^\cattwo_\simeq$ of Proposition~\ref{prop:K^2_eq-cosmos}. Now we may apply the argument of Ken Brown's lemma~\cite{Brown:1973zl}, which tells us that in the presence of a suitable functorial factorization, constructed in this case using the cylinder object in $\eK^{\eA}$ formed by the level-wise cotensor with the free-living isomorphism $\iso$, then the desired result for component-wise equivalences follows from the result for trivial fibrations.
\end{proof}

\begin{rmk}\label{rmk:simp-limit-construction}
The limit types named in the $\infty$-cosmos axioms are all constructed in $\qCat$ by taking the corresponding limits at the level of simplicial sets. It follows, from the construction of Proposition~\ref{prop:flexible-weights-are-htpical}, that the flexible weighted limits of $\qCat$ are computed as in $\SSet$ and thus that these may be given by the familiar end formula \begin{equation}\label{eq:simp-limit-construction}
\wlim{W}{F}_{\eA}\cong\int_{a\in\eA} Fa^{Wa} \cong \mathrm{eq}\left( \xymatrix{\prod\limits_{a \in \eA} Fa^{Wa} \ar@<.75ex>[r] \ar@<-.75ex>[r] & \prod\limits_{a,a' \in \eA} Fa'^{\Map_{\eA}(a,a') \times Wa}} \right)\end{equation} in there, which in turn may be expressed as an equaliser of products of functor spaces  in the underlying category of simplicial sets, and thus in its full subcategory of quasi-categories. 
\end{rmk}


When working in a quasi-category enriched category $\eK$ it is often the case that we are only interested in weighted (co)limits that are defined up to \emph{equivalence\/} rather than \emph{isomorphism}. To that end we have the following definition:

\begin{defn}[flexible weighted homotopy limits]\label{defn:flexible-hty-limit}
  Suppose that $W\colon\eA\to\SSet$ is a flexible weight and that $F\colon\eA\to\eK$ is a diagram in a quasi-category enriched category $\eK$. We say that a $W$-cone $\Lambda\colon W\to \Fun_{\eK}(L,F-)$ displays an object $L\in \eK$ as a {\em flexible weighted homotopy limit\/} of $F$ weighted by $W$ if for all objects $X\in\eK$ the map
\begin{equation}\label{eq:flex-hty-lim-comp}
  \xymatrix@R=0em@C=5em{
    \Fun_{\eK}(X,L) \ar[r] &
    \{W,\Fun_{\eK}(X,F-)\}_{\eA}.
  }
\end{equation}
induced by post-composition with $\Lambda$ is an equivalence of quasi-categories,\footnote{Here we show that the codomain of the comparison map in~\eqref{eq:flex-hty-lim-comp} is a quasi-category by applying Proposition~\ref{prop:flexible-weights-are-htpical} in the $\infty$-cosmos of quasi-categories.} in which case we denote the limit object by $\wlim{W}{F}^\simeq_{\eA}$.
\end{defn}

Proposition \ref{prop:flexible-weights-are-htpical}\ref{itm:flexible-htpical} implies that flexible weighted homotopy limits in an $\infty$-cosmos are appropriately homotopical: a natural equivalence $\kappa \colon F \To G$ of diagrams induces an equivalence
\[ \wlim{W}{F}^\simeq \we \wlim{W}{G}^\simeq.\]

\begin{defn}\label{defn:functorial-flexible}
We say that $\eK$ admits a \emph{functorial choice of flexible weighted homotopy limits\/} if there exists a simplicial bifunctor
\begin{equation*}
  \xymatrix@R=0em@C=8em{
    {(\Flex^{\eA})\op\times\eK^{\eA}}\ar[r]^-{\wlim{-}{-}^{\simeq}_{\eA}} &
    {\eK}
  }
\end{equation*}
and a family of weighted cones $\Lambda_{W,F}\colon W\to\Fun_{\eK}(\wlim{W}{F}^{\simeq}_{\eA},F-)$ which is simplicially natural in $F$ and $W$ and such that each $\Lambda_{W,F}$ displays $\wlim{W}{F}^{\simeq}_{\eA}$ as a $W$-weighted homotopy limit of $F$.
\end{defn}

\subsection{Homotopy limits in the \texorpdfstring{$(\infty,1)$}{(infty,1)}-core of an \texorpdfstring{$\infty$}{infty}-cosmos}\label{ssec:htpy-limit}

As an example of our use of flexible weighted homotopy limits we offer the following study of their construction in the $(\infty,1)$-core of an $\infty$-cosmos introduced in Definition \ref{defn:infinity,1-core}. Our main result, Proposition \ref{prop:infty-one-core-flex}, will give an explicit construction of the flexible weighted homotopy limit of any diagram valued in the $(\infty,1)$-core of an $\infty$-cosmos, which demonstrates in particular that such limits exist for any flexible weight.



We start with the following observation. Despite the fact observed in Remark \ref{rmk:core-not-enriched} that the groupoid core functor $g \colon \qCat \to \Kan$ is not simplicially enriched, it still preserves certain flexible weighted limits:

\begin{lem}\label{lem:infty-one-core-flex-1}
Consider a flexible weight $W\colon \eA\to\SSet$, taking values in the subcategory $\Kan\subset\SSet$ of Kan copmlexes,  and a diagram $F\colon\eA\to\qCat$, which factors through $g_*\qCat\subset\qCat$. On taking flexible limits in $\qCat$ we obtain a canonical comparison $g\wlim{W}{F}\to\wlim{W}{gF}$ of Kan complexes, and this is an isomorphism.
\end{lem}


\begin{proof}
From  \eqref{eq:simp-limit-construction} we know that the flexible weighted limit $\wlim{W}{F}$ (respectively $\wlim{W}{gF}$) may be constructed as an un-enriched limit of functor spaces in the underlying category of quasi-categories. The groupoid core construction is right adjoint as a functor on underlying categories, and so it preserves those un-enriched limits. Furthermore, if $\qA$ is a quasi-category and $\qX$ is a Kan-complex then $g(\qA^{\qX})\cong (g\qA)^{\qX}$, a fact which follows from the observations that any functor $f\colon\qX\to\qA$ lands in the groupoid core $g\qA$ and that a natural transformation is invertible in $\qA^{\qX}$ iff its components all lie in $g\qA$; see Remark \ref{rmk:core-not-enriched}. From these facts, it follows that we have the following sequence of natural isomorphisms
\begin{align*}
g\wlim{W}{F} &\cong g\left(\int_{a\in\eA} Fa^{Wa}\right) \\ &\cong \mathrm{eq}\left( \xymatrix{\prod\limits_{a \in \eA} (gFa)^{Wa} \ar@<.75ex>[r] \ar@<-.75ex>[r] & \prod\limits_{a,b \in \eA} (g(Fb^{\Map_{\eA}(a,b)}))^{Wa}} \right) \\ \intertext{Now if $\Map_{\eA}(a,b)$ is not a Kan complex, then $(gFb)^{\Map_{\eA}(a,b)}$ may be a proper subspace of $g(Fb^{\Map_{\eA}(a,b)})$. However, since the diagram $F$ takes values in $g_*\qCat$ the images of both of the maps
\[
 \xymatrix{\prod\limits_{a \in \eA} (gFa)^{Wa} \ar@<.75ex>[r] \ar@<-.75ex>[r] & \prod\limits_{a,b \in \eA} ((gFb)^{\Map_{\eA}(a,b)})^{Wa} \subset  \prod\limits_{a,b \in \eA}(g(Fb^{\Map_{\eA}(a,b)}))^{Wa}}\]
 factor through this subspace. Hence, we have further isomorphisms}
  &\cong  \mathrm{eq}\left( \xymatrix{\prod\limits_{a \in \eA} (gFa)^{Wa} \ar@<.75ex>[r] \ar@<-.75ex>[r] & \prod\limits_{a,b \in \eA} (gFb)^{\Map_{\eA}(a,b) \times Wa}} \right) \\ &\cong \wlim{W}{gF}
\end{align*}
whose composite is the comparison map of the statement. 
\end{proof}

\begin{obs}
  When $W\colon\eA\to\SSet$ is a weight and $F\colon\eA\to g_*\eK$ is a diagram in the $(\infty,1)$-core of a quasi-category enriched category $\eK$ then cones $\mu\colon W\to\Fun_{g_*\eK}(X,F-)$ in $g_*\eK$ correspond to cones in $\eK$ whose components $\mu_A\colon WA\to\Fun_{\eK}(X,FA)$ send every edge in $WA$ to an invertible edge. This latter condition is, however, vacuous whenever all of the edges in $WA$ are themselves invertible, since functors preserve isomorphisms. So in the case where $W \colon \eA \to \SSet$ is valued in Kan complexes,  it follows that the $W$-weighted limit of $F$ in $\eK$ provides the corresponding $W$-weighted limit in $g_*\eK$. This and the simplicial functor $g \colon g_*\qCat \to \Kan$ of Remark \ref{rmk:core-not-enriched} give the fundamental justification for the result of Lemma \ref{lem:infty-one-core-flex-1}.
\end{obs}

  So our approach to building $W$-weighted  limits in $g_*\eK$ for a general flexible weight $W$ will be to first complete the weight to some weight $W'$ whose values are Kan complexes, and whose edges are thus all invertible, and show that a $W'$-weighted limit of $F$ in $\eK$ provides us with a $W$-weighted \emph{homotopy\/} limit in $g_*\eK$ in the sense of Definition \ref{defn:flexible-hty-limit}. Under this construction it is only reasonable to hope for a homotopy limit, rather than a strict enriched limit, because the replacement $W\inc W'$ is given by the following homotopical construction:

\begin{lem}\label{lem:levelwise-Kan-replacement}$\quad$
\begin{enumerate}[label=(\roman*)]
\item Any  weight $W \in \SSet^{\eA}$ admits a level-wise Kan complex replacement $W'$, related via a projective anodyne extension $W \inc W'$.
\item If the weight $W$ is flexible then so is the associated weight $W'$.
\item\label{itm:levelwise-kan-iii} This construction can be given by a simplicial functor $(-)'\colon\Flex^{\eA}\to\Flex^{\eA}$ and a simplicially natural family of projective anodyne extensions $u_W\colon W\inc W'$.
\end{enumerate}
\end{lem}
\begin{proof}
  We may apply Quillen's small object argument in the category of weights $\SSet^{\eA}$ to the following set of \emph{projective horns}
  \begin{equation}\label{eq:kan-proj-cells}
    \left\{ \left. \Horn^{n,k}\times\Map_{\eA}(A,-)\inc \Del^n\times\Map_{\eA}(A,-) \right|\, n\geq 0 \text{ and } 0\leq k\leq n\right\}
  \end{equation}
  built from \emph{all\/} horns. In that way we obtain a weak factorisation system which factorises each simplicial natural transformation $w\colon U\to V$ of weights into a composite $w = v\circ u$ where:
  \begin{itemize}
  \item the map $u$ is a countable composite of pushouts of coproducts of projective cells in the set in display~\eqref{eq:kan-proj-cells}, or in other words a \emph{projective anodyne extension}, and
  \item the map $v$ is \emph{level-wise Kan fibration}.
  \end{itemize}
  It follows that any weight $W\colon\eA\to\SSet$ has a replacement $W'\colon\eA\to\SSet$ which is a level-wise Kan complex equipped with a projective anodyne extension $u_W\colon W\inc W'$. Since horn inclusions are build as composites of pushouts of sphere boundary inclusions, any projective anodyne extension is, in particular, a relative cell complex, so it follows that if $W$ is a flexible weight then so is its replacement $W'$.

  Indeed, by applying the simplicially enriched variant of the small object argument discussed in Chapter 13 of~\cite{Riehl:2014kx}, we may construct a factorisation system which is functorial in the simplicially enriched sense. Consequently, this replacement construction may be assumed to deliver a simplicial functor $(-)'\colon\Flex^{\eA}\to\Flex^{\eA}$ and a simplicially natural family of projective anodyne extensions $u_W\colon W\inc W'$.
\end{proof}

\begin{lem}\label{lem:infty-one-core-flex-2}
  Suppose that $u\colon U\inc W$ is a projective anodyne extension of weights in $\SSet^{\eA}$ and that $F\colon \eA\to \Kan$ is a diagram valued in Kan complexes. Then the induced simplicial map $\wlim{u}{F}\colon\wlim{W}{F}\to\wlim{U}{F}$ is a trivial fibration.
\end{lem}

\begin{proof}
  Our argument follows that of the proof of Proposition~\ref{prop:flexible-weights-are-htpical}, but this time we start from the observation that the the map $\wlim{\Del^n\times \Map_{\eA}(A,-)}{F}\to \wlim{\Horn^{n,k}\times\Map_{\eA}(A,-)}{F}$ induced by the projective horn $\Horn^{n,k}\times\Map_{\eA}(A,-) \inc \Del^n\times\Map_{\eA}(A,-)$ is simply the trivial fibration $FA^{\Del^n} \tfib FA^{\Horn^{n,k}}$.
\end{proof}

Combining the results of this section, we now compute the flexible weighted homotopy limit of any diagram $F \colon \eA \to g_*\eK$ valued in the $(\infty,1)$-core of a quasi-categorically enriched category $\eK$ that admits flexible weighted limits. The following result reveals that for any flexible weight $W \colon \eA \to \SSet$ the homotopy limit object may be computed as the homotopy weighted limit of $F \colon \eA \to \eK$ weighted by the level-wise Kan complex replacement $W'$ of $W$, and the $W$-weighted homotopy limit cone may similarly be extracted from the $W'$-weighted homotopy limit cone.

\begin{prop}\label{prop:infty-one-core-flex}
  Suppose that $W\colon \eA\to\SSet$ is a flexible weight and that $u_W\colon W\inc W'$ is its replacement by a level-wise Kan complex. Consider a diagram $F\colon\eA\to g_*\eK$  in the $(\infty,1)$-core of a quasi-category enriched category $\eK$.
  Any cone $\Lambda'\colon W'\to \Fun_{\eK}(L,F-)$ which displays $L$ as a $W'$-weighted homotopy limit in $\eK$ factors through the inclusion $\Fun_{g_*\eK}(L,F-)\subseteq\Fun_{\eK}(L,F-)$, and the composite cone
  \begin{equation*}
    \Lambda = \mkern20mu
    \xymatrix@R=0em@C=4em{
      {W}\ar@{^(->}[r]^-{u_W} & {W'}\ar[r]^-{\Lambda'} &
      {\Fun_{g_*\eK}(L,F-)}
    }
  \end{equation*}
  displays $L$ as a $W$-weighted homotopy limit of $F$ in the $(\infty,1)$-core $g_*\eK$.
\end{prop}

\begin{proof}
  The cone $\Lambda'$ factors as stated because each $W'(A)$ is a Kan complex, so its component $\Lambda'_A\colon W'(A)\to\Fun_{\eK}(L,FA)$ factors through the groupoid core $g\Fun_{\eK}(L,FA)=\Fun_{g_*\eK}(L,FA)$. For the remainder consider the following composite:
  \begin{equation*}
    \xymatrix@C=1.3em{
      {g\Fun_{\eK}(X,L)}\ar[r]^-{\simeq} &
      {g\wlim{W'}{\Fun_{\eK}(X,F-)}}\ar[r]^-{\cong} &
      {\wlim{W'}{g\Fun_{\eK}(X,F-)}}\ar@{->>}[r]^-{\simeq} &
      {\wlim{W}{g\Fun_{\eK}(X,F-)}}
    }
  \end{equation*}
  Here the left-hand map is constructed by taking the equivalence induced by post-comp\-osi\-tion with homotopy limiting cone $\Lambda'$ in $\eK$ and restricting it to groupoid cores. We infer that it is an equivalence by appealing to the fact that equivalences of quasi-categories restrict to equivalences of their groupoid cores. The middle isomorphism arises by application of Lemma~\ref{lem:infty-one-core-flex-1} and the right-hand trivial fibration by application of Lemma~\ref{lem:infty-one-core-flex-2} to the replacement $u_W\colon W\inc W'$. Now observe that this composite equivalence is simply the map obtained by post-composing with the composite cone $\Lambda$ in $g_*\eK$ and the stated result follows.
\end{proof}

Putting together the results of this section, we conclude that for any flexible weight $W$, the $(\infty,1)$-core of an $\infty$-cosmos admits a simplicially functorial choice of flexible weighted homotopy limits, formed by taking a simplicially functor levelwise Kan replacement $W'$ of the flexible weight as in Lemma \ref{lem:levelwise-Kan-replacement}, forming the strict $W'$-weighted limit as permitted by Lemma \ref{lem:infty-one-core-flex-1}, and applying Proposition \ref{prop:infty-one-core-flex}.

\begin{cor}\label{cor:infty-one-core-flex}
If $\eK$ is a quasi-categorically enriched category which admits a simplicially functorial choice of flexible homotopy limits, then so does its $(\infty,1)$-core $g_*\eK$. In particular, the $(\infty,1)$-core of an $\infty$-cosmos admits a simplicially functorial choice of flexible homotopy limits.
\end{cor}

\begin{proof}
The result follows directly from Lemma \ref{lem:infty-one-core-flex-1},  Propositions~\ref{prop:flexible-weights-are-htpical} and~\ref{prop:infty-one-core-flex}  and Lemma~\ref{lem:levelwise-Kan-replacement}\ref{itm:levelwise-kan-iii}.\end{proof}



\section{Homotopy coherent realisation, simplicial computads, and collages}\label{sec:computads}

Many interesting examples of large quasi-cat\-e\-gories arise as homotopy coherent nerves of Kan complex enriched categories. In this section, we develop tools that will be used   in \S\ref{sec:lims-in-nerves} to prove that appropriately-defined homotopy limits in a Kan complex enriched category descend to limits in the quasi-category defined as its homotopy coherent nerve.

We can probe the homotopy coherent nerve of a Kan complex enriched category by making use of the \emph{homotopy coherent realisation functor}
\[ \adjdisplay \gC -| \nrvhc : \sSet\text{-}\Cat -> \sSet.\] which is  left adjoint  to the homotopy coherent nerve functor of Cordier \cite{Cordier:1982:HtyCoh}. This left adjoint takes values in the subcategory of ``freely generated'' simplicial categories, which go by the name of \emph{simplicial computads}.\footnote{The simplicial computads are the cofibrant objects \cite[\S 16.2]{Riehl:2014kx} in  the model structure on simplicial categories due to  Bergner \cite{Bergner:2007fk}.} By adjointness, an $X$-shaped diagram in a homotopy coherent nerve transposes to define a simplicial functor whose domain is the simplicial computad $\gC{X}$, and the ``freeness'' implies that such diagrams are specified by relatively little data, enumerated in Proposition \ref{prop:gothic-C}. This analysis gives us a more explicit presentation of the generating data of a so-called \emph{homotopy coherent diagram} $\gC{X} \to \eC$ than is widely known.

In this section, we briefly review the notions needed here. A more leisurely account with considerably more details appears as \S\refVI{sec:coherent-nerve}. Simplicial computads and the homotopy coherent realisation functor are defined in \S\ref{ssec:coherent}. These techniques are applied in \S\ref{ssec:collage} to construct a particular weight $W_X \colon \gC{X} \to \SSet$ associated to homotopy coherent diagrams of shape $X$ via a simplicial computad that defines the \emph{collage} of $W_X$ and dictates the shape of $W_X$-cones over a diagram. Finally, in \S\ref{ssec:joins}, we compute the homotopy coherent realisation $\gC[X\join Y]$ of the join of two simplicial sets, a result which will be needed in \S\ref{sec:complete}.

\subsection{Simplicial computads and homotopy coherent realisation}\label{ssec:coherent}
 
 Let $\Graph$ denote the  presheaf
  category of reflexive graphs and graph morphisms. We identify the category of graphs with the essential image of the free category functor $\free \colon \Graph \inc \Cat$. The objects in the essential image are \emph{computads}, those categories in which every arrow admits a unique factorisation into \emph{atomic}\footnote{For bookkeeping reasons it is convenient to adopt the convention that atomic arrows are not identities, though in a simplicial computad the identities will also admit no non-trivial factorisations. With this convention, an identity arrow factors uniquely as an empty composite of atomic arrows.} arrows that admit no non-trivial factorisations, and the morphisms in the essential image are \emph{computad functors}, that preserve atomic arrows (or send them to identities).
  
  Recall that a simplicial category $\eA$ may be regarded as a simplicial object $[n] \mapsto \eA_n$ in the category of categories with a common set of objects $\mathrm{ob}\eA$ and identity-on-objects functors.  The arrows of $\eA_n$ are referred to as $n$-arrows of $\eA$.

  \begin{defn}[{the category of simplicial computads \refVI{defn:simplicial-computad}, \refVI{lem:computad-colimits}}]\label{defn:simplicial-computad}
  The category of simplicial computads $\sCptd$ is defined by the intersection of the 
sub-categories
  $\sCat$ and $\Cptd^{\Del\ep\op}$ of $\Cat^{\Del\ep\op}$:
  \[
\xymatrix{ \sCptd \ar@{^(->}[r] \ar[d] \pbexcursion & \sCat \ar[d] \\ \Graph^{\Del\ep\op} \ar@{^(->}[r]_{\free} & \Cat^{\Del\ep\op}}\]
which is to say that a simplicial category $\eA$ is a \emph{simplicial computad} if and only if each category $\eA_n$ of $n$-arrows is freely generated by the reflexive directed graph of atomic $n$-arrows and degenerate images of atomic arrows are atomic. A simplicial functor is a \emph{simplicial computad morphism} if it preserves atomic arrows (or send them to identities).
\end{defn}

\begin{lem}[{\refVI{lem:computad-colimits}}]\label{lem:computad-colimits} $\sCptd$ is closed
  under  colimits in $\sCat$.
  \end{lem}
  \begin{proof}
  Both of these
  subcategories   $\sCat$ and $\Cptd^{\Del\ep\op}$ are closed under  colimits in $\Cat^{\Del\ep\op}$ and
  $\Cptd^{\Del\ep\op}$ is replete in there, so it follows that $\sCptd$ is closed
  under  colimits in $\sCat$ and in $\Cptd^{\Del\ep\op}$.
  \end{proof}
  
  \begin{ex}\label{ex:directed-suspension}
  For any simplicial set $X$, we write $\cattwo[X]$ for the simplicial computad with two objects, denoted $-$ and $+$, and a single non-trivial hom-space $\Map_{\cattwo{X}}(-,+)\defeq X$.
\end{ex}

\begin{defn}[{relative simplicial computads \refVI{defn:cptd-morphism}}]\label{defn:subcomputad}\hfil
\begin{enumerate}[label=(\roman*)]
\item For simplicial computads $\eA$ and $\eB$, a simplicial computad morphism $F\colon\eA\to\eB$  is a \emph{simplicial subcomputad inclusion} if and only if its image under the functor $\sCptd \to \Graph^{\Del\ep\op}$ is a monomorphism, i.e., if and only if it  is injective on objects and faithful.
  \item For simplicial categories $\eA$ and $\eB$, a simplicial functor $F \colon \eA \to \eB$ is a \emph{relative simplicial computad} if it is expressible as a sequential composition of pushouts of coproducts of the maps
  \[ \emptyset\inc\catone \qquad \text{and} \qquad \cattwo[\boundary\Del^n]\inc\cattwo[\Del^n]\qquad \text{for}\ n \geq 0.\]
  \end{enumerate}
\end{defn}

We leave it to the curious reader to verify that if $\eA$ is a simplicial computad, then $F\colon \eA \to \eB$ is a simplicial subcomputad inclusion if and only if it is a relative simplicial computad. 

\begin{lem}[{\refVI{lem:sub-computad-stability}}]\label{lem:sub-computad-stability} Simplicial subcomputads are
  stable under pushout, coproduct, and colimit of countable sequences in
  $\sCptd$.
\end{lem}
\begin{proof}
The category $\Graph^{\Del\op\ep}$ is a presheaf category, in which colimits and monomorphisms are
  determined pointwise in $\Set$, so the desired result follows from the fact
  that it clearly holds for monomorphisms in $\Set$.
\end{proof}

\begin{ntn}[cubes, boundaries, and cubical horns]\label{ntn:cubes}$\quad$
\begin{itemize}
\item  We shall adopt the notation $\Cube^k$ for the simplicial cube
  $(\Del^1)^{\times k}$ for each $k \geq 0$.
  \item We write $\boundary\Cube^k\subset\Cube^k$ for the boundary of the cube, formed as the domain of the iterated Leibniz product
  $(\boundary\Del^1\subset\Del^1)^{\leib\times k}$.\footnote{For more details about the Leibniz or ``pushout-product'' construction see \cite[\S 4]{RiehlVerity:2013kx}.} That is, an $r$-simplex of $\Cube^k$, represented as a $k$-tuple of maps   $(\rho_1,...,\rho_k)$ with each $\rho_i \colon [r] \to [1]$, is a member of $\boundary\Cube^k$ if and
  only if there is some $i$ for which $\rho_i\colon[r]\to[1]$ is constant at either $0$ or $1$ (in which case $\rho_i$ defines an $r$-simplex in $\boundary\Del^1\subset\Del^1$).

\item  We also define the \emph{cubical horn} $\CHorn^{k,j}_{e}\subset\Cube^k$,
  for $1 \leq j \leq k$ and $e \in \{0,1\}$, to be the domain of the following Leibniz product:
  \begin{equation*}
    (\boundary\Del^1\subset\Del^1)^{\leib\times(j-1)}\leib\times
    (\Del^{\fbv{e}}\subset\Del^1)\leib\times
    (\boundary\Del^1\subset\Del^1)^{\leib\times(k-j)}
  \end{equation*}
  So an $r$-simplex $(\rho_1,...,\rho_k)$ of $\Cube^k$ is in $\CHorn^{k,j}_e$ if
  and only if $\rho_i$ is a constant operator for some $i\neq j$
  or $\rho_j$ is the constant operator which maps everything to $e$.
  \end{itemize}
\end{ntn}

  \begin{ex}[{homotopy coherent simplices as simplicial computads; \S\refVI{sec:htpy-coh-simplex}}]\label{ex:homotopy-coherent-simplex} Recall the simplicial category $\mathfrak{C}\Del^n$ whose objects are integers $0,1,\ldots, n$ and whose mapping spaces are the cubes \[ \Map_{\mathfrak{C}\Del^n}(i,j) = \begin{cases} \Cube^{j-i-1} & i < j \\ \Cube^0 & i = j \\ \emptyset & i > j \end{cases}\] 
For $i < j$, the vertices of $\Map_{\mathfrak{C}\Del^n}(i,j)$ are naturally identified with subsets of the closed interval $[i,j] = \{ i \leq t \leq j\}$ containing both endpoints, a set whose cardinality is $j-i-1$; more precisely, $\Map_{\mathfrak{C}\Del^n}(i,j)$ is the nerve of the poset with these elements, ordered by inclusion. Under this isomorphism, the composition operation  corresponds to the simplicial map 
  \[
  \xymatrix@R=1em{ \Map_{\gC\Del^n}(i,j)\times\Map_{\gC\Del^n}(j,k)\ar[r]^-\circ \ar@{}[d]|{\rotatebox{90}{$\cong$}}& \Map_{\gC\Del^n}(i,j) \ar@{}[d]|{\rotatebox{90}{$\cong$}} \\
\Cube^{\times
    (j-i-1)}\times\Cube^{\times (k-j-1)}\ar[r] & \Cube^{\times (k-i-1)}}\]
     which
  maps the pair of vertices $T \subset [i,j]$ and $S \subset [j,k]$ to $T \cup S \subset [i,k]$.
  
  Again for $i < j$, an $r$-arrow $T^\bullet$ in $\Map_{\gC\Del^n}(i,j)$ corresponds to a sequence
  \[ T^0 \subset T^1 \subset \cdots \subset T^r\] of subsets of $[i,j] = \{ i \leq t \leq j\}$ and is non-degenerate if and only if each of these inclusions are proper. The composite of a pair of $r$-arrows $T^\bullet \colon i \to j$ and $S^\bullet \colon j \to k$ is the level-wise union $T^\bullet \cup S^\bullet \colon i \to k$ of these sequences.
  
 From this description, it is easy to see that the simplicial category $\gC\Del^n$  is a simplicial computad (Lemma \refVI{lem:simplex-computad}), in which an $r$-arrow $T^\bullet$ from $i$ to $j$ is atomic if and only if the set $T^0 = \{i,j\}$; the only atomic $r$-arrows from $j$ to $j$ are identities. Geometrically,   the atomic arrows in each functor space $\Map_{\gC\Del^n}(i,j) \cong \Cube^{j-i-1}$ are precisely those simplices that contain the initial vertex in the poset whose nerve defines the simplicial cube.
\end{ex}

The homotopy coherent simplices define a cosimplicial object in $\sCat$ from which one defines the homotopy coherent realisation and homotopy coherent nerve functors first studied by Cordier \cite{Cordier:1982:HtyCoh}.

\begin{defn}[homotopy coherent realization and the homotopy coherent nerve] Applying Kan's construction \cite[1.5.1]{Riehl:2014kx} to the functor $    \gC{\Del^\bullet}\colon \Del \longrightarrow \sCat$ yields an adjunction
  \begin{equation*}
    \adjdisplay \gC -| \hN : \sCat -> \sSet .
  \end{equation*}
the right adjoint of which is called the
  \textit{homotopy coherent nerve} and the left adjoint of which, defined by pointwise
  left Kan extension along the Yoneda embedding:
\[
    \xymatrix@=1.5em{
      {\Del}\ar[rr]^{\Del^{\bullet}}\ar[dr]_{\gC{\Del^\bullet}} &
      \ar@{}[d]|(0.4){\cong} & {\sSet}\ar@{-->}[dl]^{\gC} \\
      & {\sCat} &
    }
  \]
 we refer to as \emph{homotopy coherent realisation}. An $n$-simplex of a
  homotopy coherent nerve $\hN\eC$ is simply a simplicial functor
  $c\colon\gC{\Del^{n}}\to \eC$ and the action of an operator
  $\beta\colon[m]\to[n]$ on that simplex is given by precomposition with 
$\gC{\Del^{\beta}}\colon\gC{\Del^{m}}\to\gC{\Del^{n}}$.
\end{defn}

Our aim is now to concretely describe the simplicial categories $\gC{X}$ arising as homotopy coherent realisations.  The objects of $\gC{X}$ are the vertices of $X$; a complete description of these simplicial computads is given in Proposition \ref{prop:gothic-C}. To build up to this result, we start by considering the case where $X$ is a simplicial subset of $\Del^n$. It follows from Lemma \ref{lem:sub-computad-stability} that:

\begin{lem}[{\refVI{lem:subcomputad-inclusion}}]\label{lem:subcomputad-inclusion} For any inclusion $i\colon X \inc Y$ of simplicial sets, the induced simplicial functor $\gC{i}\colon\gC{X}\inc\gC{Y}$ is a simplicial subcomputad.
\end{lem}

Lemma \ref{lem:subcomputad-inclusion} tells us that if $X \subset \Del^n$ then 
 its homotopy coherent realisation
  $\gC{X}$ is a simplicial subcomputad of $\gC\Del^n$ containing only those atomic arrows of $\gC\Del^n$ that are in the image of the simplicial functor $\gC{\Del^{\alpha}}\colon \gC{\Del^{m}}\inc\gC{\Del^{n}}$
  for some non-degenerate face $\alpha\colon[m]\inc [n]$ in $X$. It follows, in particular, that the homotopy coherent realisations of boundaries and horns have a simple description. 

\begin{ex}[homotopy coherent nerves of sub-simplices; \refVI{ex:subcomputad-of-simplex}]\label{ex:subcomputad-of-simplex}
Consider the following simplicial subsets of $\Del^n$. 
  \begin{enumerate}[label=(\roman*)]
  \item\label{itm:boundary} \textbf{boundaries:} The inclusion $\gC\boundary\Del^n\inc\gC\Del^n$ is the identity on objects. The only non-degenerate simplex in $\Del^n$ that is not present in its boundary $\partial\Del^n$ is the top dimensional $n$-simplex. Thus, an atomic $r$-arrow $T^\bullet$ of $\gC\Del^n$
    is not in $\gC\boundary\Del^n$ if and only  if it has $T^r = [0,n]$; in particular each of these missing atomic arrows
    \[ \{0,n\} = T^0 \subset T^1 \subset \cdots \subset T^n = [0,n]\] 
     lies in the functor space
    $\Map_{\gC\Del^n}(0,n)$ and $\Map_{\gC\boundary\Del^n}(i,j)=\Map_{\gC\Del^n}(i,j)$ for all but that particular functor space. 
The inclusion
         \[
  \xymatrix@R=1em{ \Map_{\gC\boundary\Del^n}(0,n)\ar@{^(->}[r] \ar@{}[d]|{\rotatebox{90}{$\cong$}}& \Map_{\gC\Del^n}(0,n)\ar@{}[d]|{\rotatebox{90}{$\cong$}} \\
\boundary\Cube^{n-1}\ar@{^(->}[r] & \Cube^{n-1}}\]
    is isomorphic to the cubical boundary inclusion.
    
  \item\label{itm:outer-horn}\textbf{outer horns:} The inclusion $\gC\Horn^{n,n}\inc\gC\Del^n$ is also the identity on objects. The only non-degenerate simplices in $\Del^n$ that are not present in the horn $\Horn^{n,n}$ are the top dimensional $n$-simplex and its $n\th$ face. In the former case, the missing atomic $r$-arrows are elements of $\Map_{\gC\Del^n}(0,n)$, and in the latter case they are elements of  $\Map_{\gC\Del^n}(0,n-1)$;
    so $\Map_{{\gC\Horn^{n,n}}}(i,j)= \Map_{\gC\Del^n}(i,j)$ for all but those two
    functor spaces. Under the isomorphism of Example \ref{ex:homotopy-coherent-simplex}
    the inclusions 
         \[
           \xymatrix@R=1em{\Map_{\gC\Horn^{n,n}}(0,n-1)\ar@{^(->}[r] \ar@{}[d]|{\rotatebox{90}{$\cong$}}&\Map_{\gC\Del^n}(0,n-1)\ar@{}[d]|{\rotatebox{90}{$\cong$}} \\
\boundary\Cube^{n-2}\ar@{^(->}[r] &\Cube^{n-2}}
\qquad
  \xymatrix@R=1em{ \Map_{{\gC\Horn^{n,n}}}(0,n)\ar@{^(->}[r] \ar@{}[d]|{\rotatebox{90}{$\cong$}}& \Map_{\gC\Del^n}(0,n)\ar@{}[d]|{\rotatebox{90}{$\cong$}} \\
\CHorn^{n-1,n-1}_0\ar@{^(->}[r] & \Cube^{n-1}}\] are isomorphic to the cubical boundary and cubical horn inclusions. 
  \end{enumerate}
\end{ex}

\begin{defn}[bead shapes]
  We shall call those atomic arrows $T^\bullet\colon 0 \to n$ of $\gC{\Del^{n}}$
  which are not members of $\gC\boundary\Del^n$ \emph{bead shapes}. By Examples \ref{ex:homotopy-coherent-simplex} and \ref{ex:subcomputad-of-simplex}, an $r$-dimensional bead shape $T^\bullet\colon
  0\to n$ is given by a sequence of subsets
  \[ \{0,n\} = T^0 \subset T^1 \subset \cdots \subset T^r = [0,n]\]
  with $T^0 = \{0,n\}$  and $T^r$ equal to the full interval $[0,n] = \{ 0 \leq t \leq n\}$. 
\end{defn}

\begin{prop}[{$\gC X$ as a simplicial computad; \refVI{prop:gothic-C}}]\label{prop:gothic-C}
The homotopy coherent realisation $\gC{X}$ of a simplicial set $X$ is a simplicial computad with
  \begin{itemize}
  \item   objects the vertices of $X$ and 
  \item non-degenerate atomic $r$-arrows given by pairs $(x,T^\bullet)$, wherein
  $x$ is a non-degenerate $n$-simplex of $X$ for some $n > r $ and
  $T^\bullet\colon 0\to n$ is an $r$-dimensional bead shape. 
  \end{itemize}
  The domain of $(x,T^\bullet)$ is the initial vertex $x_0$ of $x$ and its codomain is the terminal vertex $x_n$.\qed
\end{prop}
 
The pairs $(x,T^\bullet)$ appearing in Proposition \ref{prop:gothic-C} called \emph{beads in $X$}. As a consequence of this result we find that $r$-simplices of $\gC{X}$ correspond to sequences of abutting beads, structures which are called \emph{necklaces} in the work of Dugger and Spivak~\cite{DuggerSpivak:2011ms} and Riehl~\cite{Riehl:2011ot}. In this terminology, $\gC X$ is a simplicial computad in which the atomic arrows are those necklaces that consist of a single bead with non-degenerate image.

\subsection{Collages of weights}\label{ssec:collage}

In this section we explore an alternate presentation of a simplicial weight $W \colon \eA\to\SSet$ as a simplicial category we refer to as the \emph{collage} of the weight. The idea is that $\coll(W)$ describes the shape of $W$-cones over a diagram. The main result is that a weight $W$  is flexible if and only if $\catone + \eA\inc\coll(W)$ is a relative simplicial computad. The main example will be the collage that defines the shape of pseudo limit cones over a homotopy coherent diagram, a weight that will feature prominently in the the main theorem of \S\ref{sec:complete}.

A simplicial functor $W\colon\eA\to\SSet$ may otherwise be described as
comprising a family of simplicial sets $\{Wa\}_{a\in\eA}$ along with left
actions
\begin{equation}\label{eq:action-of-weight}
  \xymatrix@R=0em@C=6em{
    Wa\times\Map_{\eA}(a,a')\ar[r]^-{*} & Wa'
  }
\end{equation}
of the hom-spaces of $\eA$ that must collectively satisfy the customary
axioms with respect to the identities and composition of $\eA$. This
description leads us to define a simplicially enriched category $\coll(W)$,
called the \emph{collage} of $W$.

\begin{defn}[collages]\label{defn:collage}
  For any weight $W \colon \eA\to\SSet$, the \emph{collage} of $W$ is a
  simplicial category $\coll(W)$ that contains $\eA$ as a full simplicial
  subcategory along with precisely one extra object $\bot$ whose endomorphism
  space is the point. The simplicial sets $\Map_{\coll(W)}(a,\bot)$ are all taken
  to be empty and we define:
  \begin{equation*}
    \Map_{\coll(W)}(\bot,a) \defeq Wa\mkern40mu \text{for objects $a\in\eA$.}
  \end{equation*}
  The composition operations between hom-spaces with domain $\bot$ and hom-spaces in  $\eA$ 
  are given by the actions depicted in~\eqref{eq:action-of-weight}.
\end{defn}

\begin{prop}[the collage adjunction]\label{prop:collage-adjunction}$\quad$
\begin{enumerate}[label=(\roman*)]
\item\label{itm:collage-fun} The collage construction defines a fully faithful functor 
  \begin{equation*}
    \xymatrix@R=0em@C=6em{
      \sSet^{\eA}\ar[r]^-{\coll} & \prescript{\catone+\eA/}{}{\sCat}
    }
  \end{equation*}
  from the category of $\eA$-indexed weights to the category of simplicial categories under $\catone+\eA$ whose essential image is comprised of those simplicial functors $F\colon \catone+\eA\to\eC$ that are bijective on objects, fully faithful when restricted to $\eA$ and $\catone$, and have the property that any arrow with codomain $F\bot$ is the identity.
\item\label{itm:collage-adj} The collage functor admits a right adjoint
  \begin{equation*}
    \adjdisplay \coll -| \wgt :
    \prescript{\catone+\eA/}{}{\sCat} -> \sSet^{\eA}.
  \end{equation*}
  which carries a simplicial functor $F\colon\catone+\eA\to\eC$ to the following weight:
  \begin{equation*}
    \xymatrix@R=0em@C=8em{
      {\eA}\ar[r]^{\Map_{\eC}(F\bot,F-)} & {\SSet}
    }
  \end{equation*}
\end{enumerate}
\end{prop}

Here $\sSet^{\eA}$ denotes the underlying category of the simplicially enriched category $\SSet^{\eA}$.
  
\begin{proof}
  The subcategory inclusion of $\eA$ into $\coll(W)$ together with the
  functor $\catone\to\coll(W)$ mapping the unique object of $\catone$ to
  $\bot$ define a functor $\catone+\eA\inc \coll(W)$. Extending this
  construction to simplicial natural transformations in the obvious way we
  obtain a fully faithful functor
  \begin{equation*}
    \xymatrix@R=0em@C=6em{
      \sSet^{\eA}\ar[r]^-{\coll} & \prescript{\catone+\eA/}{}{\sCat}
    }
  \end{equation*}
 From this vantage point, the characterisation of the essential image is clear.

 Given a map of weights $\Lambda\colon W\to \Map_{\eC}(F\bot, F-)$, we may construct a simplicially enriched functor $F_\Lambda\colon\coll(W)\to\eC$ which carries $\bot$ to $F\bot$, acts as $F$ on the subcategory $\eA$, and which acts on the hom-space $\Map_{\coll(W)}(\bot,a)=Wa$ in the way given by the component $\Lambda_a\colon Wa\to \Map_{\eC}(F\bot, Fa)$. The simplicial functoriality of $F_\Lambda$ is an immediate consequence of the simplicial naturality of $\Lambda$. Conversely given a simplicial functor
  \begin{equation*}
    \xymatrix@=1.5em{
      & {\catone+\eA}\ar@{^(->}[dl]!R!U(0.5)\ar[dr]^{F} & \\
      {\coll(W)}\ar[rr]_-{G} && {\eC}
    }
  \end{equation*}
  in $\prescript{\catone+\eA/}{}{\sCat}$ we may construct a map of weights $\Lambda_G\colon W\to\Map_{\eC}(F\bot,F-)$ whose component at an object $a\in\eA$ is simply the action of $G$ on the hom-space $\Map_{\coll(W)}(F\bot,Fa)=Wa$. These operations are clearly mutually inverse, thereby demonstrating the postulated right adjointness property of the weight $\Map_{\eC}(F\bot,F-)$ with respect to the collage construction; see also
  Observation \refII{obs:coll-right-adj}.
  \end{proof}

  This adjunction has a useful and important interpretation, when read in light of Definition \ref{defn:simp-weight}:
  
  \begin{cor}\label{cor:collage-bijection} The collage
    $\coll(W)$ realises the shape of $W$-cones, in the sense that simplicial functors $G \colon \coll(W) \longrightarrow \eC$ with domain $\coll(W)$ stand in bijection to $W$-cones with apex $G(\bot)$ over the restricted diagram $G\colon\eA\to\eC$.\qed
\end{cor}

\begin{ntn}
From hereon we shall tend to blur the distinction between $W$-cones in $\eC$ and the simplicial functors $\coll(W)\to\eC$ to which they correspond. Consequently we will generally use the term ``weighted cone'' to mean either structure interchangeably.
\end{ntn}

We have the following recognition principle for flexible weights on simplicial
computads,  a mild variant of Proposition~\refII{prop:projcofchar}.

\begin{thm}[flexible weights and collages]\label{thm:flexible-collage} A natural transformation $\alpha \colon V \inc W$ between weights in $\sSet^{\eA}$ is a projective cell complex if and only if $\coll{(\alpha)} \colon \coll{V} \inc \coll{W}$ is a relative simplicial computad. In particular, $W$ is a flexible weight if and only if $\catone+\eA\inc\coll{W}$ is a relative simplicial computad.
\end{thm}
\begin{proof}
If  $\alpha \colon V \inc W$ is a projective cell complex, then it can be presented as a countable composite of pushouts of coproducts of projective cells of varying dimensions indexed by the objects $a \in \eA$. Since the collage construction is a left adjoint, it preserves these colimits, and hence the map $\coll{(\alpha)} \colon \coll{V} \inc \coll{W}$ as a sequential composite of pushouts of coproducts of simplicial functors $\coll{(\partial\Delta[n]\times\eA(a,-))} \inc \coll{(\Delta[n]\times\eA(a,-))}$ in the category ${}^{\catone+\eA/}\sCat$. This composite colimit diagram is connected --- note $\coll{\emptyset}=\catone+\eA$ --- so this cell complex presentation is also preserved by the forgetful functor ${}^{\catone+\eA/}\sCat \to \sCat$ and the simplicial functor $\coll{(\alpha)} \colon \coll{V} \inc \coll{W}$ can be understood as a countable composite of pushouts of coproducts of  $\coll{(\partial\Delta[n]\times\eA(a,-))} \inc \coll{(\Delta[n]\times\eA(a,-))}$ in $\sCat$.

This is advantageous because there is a pushout square in $\sCat$ displayed below left corresponding to the maps of simplicial sets displayed below right
\begin{equation}\label{eq:collage-computad-pushout}
\begin{tikzcd}
\cattwo[\partial\Delta[n]] \arrow[d, hook] \arrow[r] \arrow[dr, phantom, "\ulcorner" very near end] & \coll{(\partial\Delta[n]\times\eA(a,-))} \arrow[d, hook] & \arrow[d, phantom, "\leftrightsquigarrow"] & \partial\Delta[n] \arrow[r, hook, "\id_a"] \arrow[d, hook] & \partial\Delta[n] \times \eA(a,a) \arrow[d, hook] \\ \cattwo[\Delta[n]] \arrow[r] & \coll{(\Delta[n]\times\eA(a,-))}  & ~ & \Delta[n] \arrow[r, hook, "\id_a"] & \Delta[n] \times \eA(a,a)
\end{tikzcd}
\end{equation}
whose horizontals sends the two objects $-$ and $+$ of the simplicial computads defined in Example \ref{ex:directed-suspension} to $\top$ and $a$ and act on the non-trivial hom-spaces via the inclusions whose component in $\eA(a,a)$ is constant at the identity element at $a$. To verify that $\coll{(\partial\Delta[n]\times\eA(a,-))} \inc \coll{(\Delta[n]\times\eA(a,-))}$ is a pushout of $\cattwo[\partial\Delta[n]]\inc\cattwo[\Delta[n]]$, observe from the adjunction of Proposition \ref{prop:collage-adjunction} and the Yoneda lemma  that a commutative triangle of simplicial functors as below left-transposes to a commutative triangle of simplicial maps as below-right
\[
\begin{tikzcd}
 \coll{(\partial\Delta[n]\times\eA(a,-))} \arrow[d, hook]  \arrow[dr]  & & \arrow[d, phantom, "\leftrightsquigarrow"] & \partial\Delta[n] \arrow[d, hook] \arrow[dr]  \\  \coll{(\Delta[n]\times\eA(a,-))} \arrow[r] & \eC  & ~ & \Delta[n] \arrow[r] & \Map_{\eC}(F\bot,Fa)
\end{tikzcd}
\]
where $F \colon \catone+\eA \inc \coll{(\Delta[n]\times\eA(a,-))} \to \eC$ is the composite functor. Thus, we see that to extend a simplicial functor $ \coll{(\partial\Delta[n]\times\eA(a,-))} \to \eC$ along the inclusion $\coll{(\partial\Delta[n]\times\eA(a,-))} \inc \coll{(\Delta[n]\times\eA(a,-))}$  is to attach an $n$-simplex to an $n$-simplex boundary in a particular hom-space in $\eC$, i.e., to extend a functor $\cattwo[\partial\Delta[n]] \to \eC$ along $\cattwo[\partial\Delta[n]]\inc\cattwo[\Delta[n]]$. From this we see that $\coll{(\alpha)} \colon \coll{V} \inc \coll{W}$ is a transfinite composite of pushouts of coproducts of simplicial functors $\cattwo[\partial\Delta[n]]\inc\cattwo[\Delta[n]]$, which proves that this map is a relative simplicial computad.

Conversely, if $\coll{(\alpha)} \colon \coll{V} \inc \coll{W}$ is a relative simplicial computad, then it can be presented as a countable composite of pushouts of coproducts of simplicial functors $\cattwo[\partial\Delta[n]]\inc\cattwo[\Delta[n]]$; since this inclusion is bijective on objects the inclusion $\emptyset\inc\catone$ is not needed. Since the only arrows of $\coll{W}$ that are not present in $\coll{V}$ have domain $\top$ and codomain $a \in \eA$, the characterisation of the essential image of the collage functor of Proposition \ref{prop:collage-adjunction}\ref{itm:collage-fun} allows us to identify each stage of the countable composite
\[
\begin{tikzcd}[column sep=small]
\coll{V} \arrow[r, hook] & \coll{(W^1)}  \arrow[r, hook] & \cdots \arrow[r, hook] &\coll{(W^i)} \arrow[r, hook] & \coll{(W^{i+1})} \arrow[r, hook] & \cdots \arrow[r, hook] & \coll{W}
\end{tikzcd}
\]
as the collage of some weight $W^i \colon \eA \to\SSet$. Each attaching map $\cattwo[\partial\Delta[n]] \to \coll{W^i}$ in the cell complex presentation acts on objects by mapping $-$ and $+$ to $\top$ and $a$ for some $a \in \eA$, and hence factors through the top horizontal of the pushout square \eqref{eq:collage-computad-pushout}. Hence, the inclusion $\coll{(W^i)}\inc\coll{(W^{i+1})}$ is a pushout of a coproduct of the maps 
$\coll{(\partial\Delta[n]\times\eA(a,-))} \inc \coll{(\Delta[n]\times\eA(a,-))}$, one for each cell $\cattwo[\partial\Delta[n]]\inc\cattwo[\Delta[n]]$ whose attaching map sends $+$ to $a \in \coll{(W^i)}$. As the collage functor is fully faithful, we have now expressed $\coll{(\alpha)} \colon \coll{V} \inc \coll{W}$ as a countable composite of pushouts of coproducts of simplicial functors $\coll{(\partial\Delta[n]\times\eA(a,-))} \inc \coll{(\Delta[n]\times\eA(a,-))}$. A fully faithful functor that preserves colimits also reflects them, so in this way we see that $\alpha \colon V \inc W$ is a countable composite of pushouts of coproducts of projective cells, proving that it is a projective cell complex as claimed.
\end{proof}

We now apply Proposition \ref{prop:collage-adjunction} and Theorem \ref{thm:flexible-collage} to identify a flexible weight for homotopy coherent diagrams of shape $X$ by means of its collage.

\begin{defn}[weights for pseudo limits]\label{defn:weight-for-pseudo-limits}
For any simplicial set $X$, the coherent realisation of the canonical inclusion $\Del^0+X \inc \Del^0\join X$ defines a collage $\catone+\gC{X}\inc\gC{[\Del^0\join X]}$. Hence, via the counit isomorphism of the collage adjunction Proposition \ref{prop:collage-adjunction}\ref{itm:collage-fun}, this simplicial category is isomorphic to the collage of its corresponding weight which we denote by $W_X \colon \gC{X} \to \SSet$ and call the \emph{weight for the pseudo limit of a homotopy coherent diagram of shape $X$}. The $W_X$-weighted limit of a homotopy coherent diagram of shape $X$ is then referred to as the \emph{pseudo limit} of that diagram.

Since the left adjoint of the collage adjunction is fully faithful, its unit is an isomorphism, and this permits us to define the weight $W_X$ explicitly: for a vertex $x \in X$,
\[ W_X(x) \defeq \gC{[\Del^0\join X]}(\bot,x).\]
So the functor
  \begin{equation*}
    \xymatrix@R=0em@C=5em{
      \gC{X}\ar[r]^{W_X} & {\SSet}
    }\qquad\text{is~given~by}\qquad W_X(x) \defeq {\gC[\Del^0\join X]}(\bot,x). 
  \end{equation*}
  \end{defn}

\begin{lem}\label{lem:pseudo-is-flexible}
  For all simplicial sets $X$ the weight $W_X\colon\gC{X}\to\SSet$ for homotopy limits of diagrams of shape $\gC{X}$ is a flexible weight.
\end{lem}

\begin{proof}
  The collage of $W_X$ is $\catone + \gC{X} \cong \gC[\Del^0+X] \inc \gC[X\join\Del^0]$, which is a simplicial computad inclusion by Lemma \ref{lem:subcomputad-inclusion}, so this result follows from Theorem \ref{thm:flexible-collage}.
\end{proof}

  By Corollary \ref{cor:collage-bijection}, a simplicial functor $F\colon \gC[\Del^0\join X] \to \eC$ defines a $W_X$-shaped cone over the restricted diagram $F\colon\gC[X]\to\eC$. In the case where $F$ encodes a weighted limit cone, we refer to the simplicial functor as the \emph{pseudo limit cone over $F$}. We call $W_X$ the weight for a ``pseudo'' limit because we anticipate considering homotopy coherent diagrams valued in Kan-complex enriched categories, $\eC$ in which all arrows in positive dimension are automatically invertible. If we were to consider diagrams valued in quasi-categorically enriched categories,  it would be more appropriate to refer to $W_X$ as the weight for \emph{lax} limits, since in that context the 1-arrows of $\gC[\Del^0\join X]$ will likely map to non-invertible morphisms.

\subsection{Homotopy coherent realisation of joins}\label{ssec:joins}

Our aim in this section is to establish a characterisation of the homotopy coherent realisation $\gC[X\join Y]$ of the join of two simplicial sets.   By Proposition \ref{prop:gothic-C}, the objects in $\gC[X\join Y]$ are vertices in either $X$ or $Y$, and  there are no $n$-arrows in $\gC[X \join Y]$ from a vertex in $Y$ to a vertex in $X$. In this way, we have a canonical functor $\gC[X\join Y] \to \cattwo$ so that $\gC[X]$ is the fibre over $0$ and $\gC[Y]$ is the fibre over 1. It remains, then, to identify those $n$-arrows whose source is a vertex in $X$ and whose target is a vertex in $Y$. In Theorem \ref{thm:Phi-iso}, which we state more precisely at the end of this section, we will establish a natural isomorphism of simplicial computads
\[ \gC[X \join Y] \cong \gC[X\join\Del^0] \times_\cattwo \gC[\Del^0\join Y].\]

We start by formalising our analysis of the construction just given.

\begin{lem}\label{lem:join-gC-bifunctor} There is a canonical bifunctor
  \begin{equation}\label{eq:join-gC-bifunctor-cptd}
    \xymatrix@R=0ex@C=8em{
      {\sSet\times\sSet}\ar[r]^{\gC[{-}\join{-}]} & {\sCptd_{/\cattwo}}}
  \end{equation}
  which preserves  connected colimits in each variable independently.
  \end{lem}
  \begin{proof}
  Given a pair of simplicial sets $X$ and $Y$, we promote the arrow category
  $\cattwo = \{ 0 \to 1 \} $ to a trivially enriched simplicial category and
  define a canonical simplicial functor $P_{X,Y}\colon \gC[X\join Y]\to\cattwo$
  uniquely determined by the stipulation that it maps objects arising from
  vertices of $X$ to the object $0$ and those arising from vertices of $Y$ to
  the object $1$. Note that this is a functor of
  simplicial computads, since any simplicial functor from a computad into $\cattwo$ is
  necessarily a functor of simplicial computads.

  Corresponding observations reveal that if $f\colon X\to X'$ and $g\colon Y\to
  Y'$ are simplicial maps, then the corresponding functor $\gC[f\join g]\colon
  \gC[X\join Y]\to \gC[X'\join Y']$ makes the following triangle commute:
  \begin{equation*}
    \xymatrix@=1em{
      {\gC[X\join Y]}\ar[rr]^-{\gC[f\join g]}\ar[dr]_{P_{X,Y}} &&
      {\gC[X'\join Y']}\ar[dl]^{P_{X',Y'}} \\
      & {\cattwo} & }
  \end{equation*}
  In this way we have shown that the construction $(X,Y)\mapsto \gC[X\join Y]$
  extends naturally to give a bifunctor:
\begin{equation}\label{eq:join-gC-bifunctor}
    \xymatrix@R=0ex@C=8em{
      {\sSet\times\sSet}\ar[r]^{\gC[{-}\join{-}]} & {\sCat_{/\cattwo}}}
  \end{equation}
  
  The join of simplicial sets preserves  connected colimits in each
  variable independently \cite[3.2]{Joyal:2002:QuasiCategories} and the homotopy coherent realisation functor is a left
  adjoint so it preserves all  colimits. Furthermore colimits in the
  slice $\sCat_{/\cattwo}$ are constructed as in $\sCat$, so it follows that this bifunctor also preserves  connected
  colimits in each variable independently. Since the homotopy coherent
  realisation functor $\gC$ factors through the subcategory $\sCptd\inc\sCat$,
  by Lemma~\refVI{lem:realization-in-computads}, and that inclusion creates
   colimits, by Lemma~\ref{lem:computad-colimits}. It follows
  that we can restrict the codomain of~\eqref{eq:join-gC-bifunctor} to define the bifunctor of the statement  which preserves  connected colimits in each variable independently.
\end{proof}

\begin{rmk}\label{rmk:join-gC-bifunctor-var}
  By specialising the bifunctor \eqref{eq:join-gC-bifunctor-cptd} we obtain a pair of functors
  \begin{equation}\label{eq:join-gC-bifunctor-var-1}
    \xymatrix@R=0ex@C=5em{{\sSet}\ar[r]^-{\gC[{-}\join \Del^0]} &
      {\sCat_{/\cattwo}}}\mkern40mu
    \xymatrix@R=0ex@C=5em{{\sSet}\ar[r]^-{\gC[\Del^0\join{-}]} &
      {\sCat_{/\cattwo}}}
  \end{equation}
  and these give rise to a bifunctor
  \begin{equation}\label{eq:join-gC-bifunctor-var-2}
    \xymatrix@R=0ex@C=5em{
      {\sSet\times\sSet}\ar[rr]^-{\gC[{-}\join \Del^0]\times
        \gC[\Del^0\join{-}]} &&
      {\sCat_{/\cattwo}\times\sCat_{/\cattwo}} \ar[r]^-{{-}\times_{\cattwo}{-}} &
      {\sCat_{/\cattwo}}
    }
  \end{equation}
  where the functor labelled ${-}\times_{\cattwo}{-}$ is the binary product
  (pullback over $\cattwo$) of the slice category $\sCat_{/\cattwo}$.

  On applying the bifunctor $\gC[{-}\join{-}]$ to the unique map $!\colon Y\to
  \Del^0$ (respectively $!\colon X\to \Del^0$) we obtain a family of simplicial
  projection functors $\gC[X\join !]\colon \gC[X\join Y]\to\gC[X\join\Del^0]$
  natural in $X\in \sSet$ (respectively $\gC[!\join Y]\colon \gC[X\join Y]\to
  \gC[\Del^0\join Y]$ natural in $Y\in\sSet$) in the slice $\sCat_{/\cattwo}$.
  As a consequence of the uniqueness of maps into the terminal simplicial set
  $\Del^0$ these families of functors are also natural in their second variable,
  in the sense that the following triangles commute:
  \begin{equation*}
    \xymatrix@=1em{
      {\gC[X\join Y]} \ar[rr]^-{\gC[X\join g]}\ar[dr]_-{\gC[X\join !]} &&
      {\gC[X\join Y']} \ar[dl]^-{\gC[X\join !]} \\
      & {\gC[X\join\Del^0]} & }
    \mkern60mu
    \xymatrix@=1em{
      {\gC[X\join Y]} \ar[rr]^-{\gC[f\join Y]}\ar[dr]_-{\gC[!\join Y]} &&
      {\gC[X'\join Y]} \ar[dl]^-{\gC[!\join Y]} \\
      & {\gC[\Del^0\join Y]} & }
  \end{equation*}
  These projection maps induce a comparison map
  \begin{equation*}
    \Phi_{X,Y}\defeq (\gC[X\join !],\gC[!\join Y]) \colon
    \gC[X\join Y]\longrightarrow \gC[X\join\Del^0] \times_{\cattwo}
    \gC[\Del^0\join Y]
  \end{equation*}
  and the naturality properties of the projections discussed above imply that
  these are natural in $X$ and $Y$. It follows that these comparisons assemble
  into a natural transformation
  \begin{equation}\label{eq:join-gC-bifunctor-var-3}
    \xymatrix@R=0ex@C=9em{
      {\sSet\times\sSet}
      \ar@/^1.5ex/[]!R!UR(0.5);[r]!L!UL(0.5)^-{\gC[{-}\join{-}]}
      \ar@/_1.5ex/[]!R!DR(0.5);[r]!L!DL(0.5)_-{\gC[{-}\join\Del^0]\times_{\cattwo}
        \gC[\Del^0\join{-}]}\ar@{}[r]|{\Downarrow\Phi} &
      {\sCat_{/\cattwo}}
    }
  \end{equation}
  between the bifunctors introduced here and in
Lemma~\ref{lem:join-gC-bifunctor}.
\end{rmk}

While we made a point of observing that the bifunctor introduced in
Lemma~\ref{lem:join-gC-bifunctor} landed in the subcategory of computads in
$\sCat_{/\cattwo}$ we have, as yet, made no such claim for the corresponding
bifunctor introduced in Remark~\ref{rmk:join-gC-bifunctor-var}. This result will be proven as Corollary \ref{cor:phi-lands-in-computad}, but it requires a little justification, as provided in the following
sequence of lemmas:

 Given an object $\eA\to\cattwo$ of $\Cat_{/\cattwo}$ or $\sCat_{/\cattwo}$ we
  shall use the notation $\eA^i$ to denote its fibre over the object
  $i\in\cattwo$, the subscript already being taken.

\begin{obs}\label{obs:fibre-over-2}
  Suppose that $\eA\to\cattwo$ is an object of $\Cat_{/\cattwo}$ (respectively
  $\sCat_{/\cattwo}$). Then any arrow that is not contained in one of its fibres
  must have domain in $\eA^0$ and codomain in $\eA^1$. So an arrow in one of
  those fibres cannot be factored through an object in the other fibre. It
  follows that an arrow in a fibre $\eA^i$ is atomic (respectively has a unique
  decomposition as a composite of atomic arrows) in there if and only if it is
  atomic (respectively has a unique atomic decomposition) in $\eA$ itself. 
\end{obs}

Recall we identify $\Graph$ with the essential image of the free category functor $\Graph\inc\Cat$ and refer to free categories as \emph{computads}.

\begin{lem}\label{lem:fibre-prod-over-2}
  Suppose that $\eA\to\cattwo$ and $\eB\to\cattwo$ are objects of $\Graph_{/\cattwo}$ with $\eA^1\cong\catone$ and $\eB^0 \cong \catone$. Then the fibred product $\eA \times_{\cattwo} \eB$ in $\Cat$ is contained in the subcategory $\Graph$ and defines the fibred product in that subcategory.
\end{lem}

When $\eA^1\cong\catone$ and $\eB^0\cong\catone$, we use the notation $\top$ (respectively $\bot$) to denote the unique object of $\eA^1$ (respectively $\eB^0$).

\begin{proof}
  We have that $(\eA\times_{\cattwo}\eB)^0=\eA^0\times\eB^0\cong\eA^0$ and
  $(\eA\times_{\cattwo}\eB)^1=\eA^1\times\eB^1\cong\eB^1$, but
  Observation~\ref{obs:fibre-over-2} tells us that the fibres $\eA^0$ and
  $\eB^1$ are computads and that the fibre inclusions
  $(\eA\times_{\cattwo}\eB)^i \inc \eA\times_{\cattwo}\eB$ preserve atomic
  arrows. So an arrow $(f,\id_\bot)$ (respectively $(\id_\top,g)$) is atomic in
  $\eA\times_{\cattwo}\eB$ if and only if $f$ is atomic in $\eA$ (respectively
  $g$ is atomic in $\eB$). The unique atomic decomposition of a general such
  arrow $(f,\id_\bot)$ (respectively $(\id_\top,g)$) is obtained by taking the
  atomic decomposition of $f$ in $\eA$ (respectively $g$ in $\eB$). This
  establishes the required atomic decomposition result for arrows in the fibres
  of $\eA\times_{\cattwo}\eB$.

  Observe now that the arrows of $\eA\times_{\cattwo}\eB$ which are not in one
  of its fibres are of the form $(f,g)\colon (A,\bot)\to(\top,B)$ with
  $A\in\eA^0$ and $B\in\eB^1$. Clearly, such an arrow factors through some
  object $(A',\bot)$ in $\eA\times_{\cattwo}\eB$ if and only if $f$ factors
  through $A'$ as a composite $f_2\circ f_1$ in $\eA$, and that decomposition
  uniquely determines a corresponding decomposition
  $(f,g)=(f_2,g)\circ(f_1,\id_\bot)$. Dual comments apply to factorisations of
  $(f,g)$ through some object $(\top,B')$; it follows easily that the map
  $(f,g)$ is atomic in $\eA\times_{\cattwo}\eB$ if and only if $f$ is atomic in
  $\eA$ and $g$ is atomic in $\eB$. To decompose a general such map $(f,g)$ into
  atomic arrows we form the atomic decompositions $f=f_n\circ\cdots\circ f_1$ in
  $\eA$ and $g=g_m\circ\cdots\circ g_1$ in $\eB$ and write down the manifest
  atomic decomposition $(f,g)=(\id_\top,g_m)\circ\cdots\circ(f_n,g_1)\circ\cdots
  \circ(f_1,\id_\bot) $ which is easily seen to be the unique such.

  Having shown that $\eA\times_{\cattwo}\eB$ is a computad, it remains only to
  check that it possesses the universal property of the fibred product in
  $\Graph$. However, our analysis of $\eA\times_{\cattwo}\eB$ actually
  showed that its graph of atomic arrows is the fibred product of the graphs of
  the atomic arrows of $\eA$ and $\eB$.
\end{proof}

\begin{lem}\label{lem:fibre-prod-over-2-again}
  Suppose that $\eA\to\cattwo$ and $\eB\to\cattwo$ are objects of $\sCptd_{/\cattwo}$
  with $\eA^1\cong\catone$ and $\eB^0\cong\catone$. Then the subcategory $\sCptd$ is
  closed in $\sCat$ under the fibred product $\eA \times_{\cattwo} \eB$.
\end{lem}

\begin{proof}
By Definition \ref{defn:simplicial-computad}, $\sCptd=\sCat\cap\Graph^{\Del\ep\op}$, with the intersection formed in $\Cat^{\Del\ep\op}$. The subcategory $\sCat$ is closed in $\Cat^{\Del\ep\op}$ under all limits, in particular it is closed under the fibred product $\eA\times_{\cattwo}\eB$. Furthermore, we have that $\eA$ and $\eB$ satisfy the condition given in the statement if and only if for all $n\in\mathbb{N}$ their categories $\eA_n$ and $\eB_n$ of $n$-arrows satisfy the conditions of Lemma~\ref{lem:fibre-prod-over-2}. It follows that $\Graph$ is closed in $\Cat$ under the fibred product $\eA_n\times_{\cattwo}\eB_n$ for each $n\in\mathbb{N}$ and thus that $\Graph^{\Del\ep\op}$ is also closed in $\Cat^{\Del\ep\op}$ under the fibred product $\eA\times_{\cattwo}\eB$. In this way we have shown that $\eA\times_{\cattwo}\eB$ and its projections are in $\sCptd=\sCat\cap\Graph^{\Del\ep\op}$; a routine argument demonstrates that it also has the desired universal property in there, using the fact that it does so in each of the subcategories $\sCat$ and $\Graph^{\Del\ep\op}$.
\end{proof}



\begin{defn}\label{defn:bot-top-over-2}
  Let $\sCat_{/\cattwo}^{\bot}$ (respectively $\sCat_{/\cattwo}^{\top}$) denote
  the full subcategory of $\sCat_{/\cattwo}$ spanning those $\eA\to\cattwo$
  whose fibre $\eA^0$ (respectively $\eA^1$) is isomorphic to $\catone$. Also
  let $\sCptd_{/\cattwo}^\bot\defeq \sCptd_{/\cattwo}\cap
  \sCat_{/\cattwo}^{\bot}$ and $\sCptd_{/\cattwo}^\top\defeq \sCptd_{/\cattwo}
  \cap \sCat_{/\cattwo}^{\top}$.
\end{defn}

\begin{lem}\label{lem:closed-conn-colim}
  The subcategories $\sCat_{/\cattwo}^{\bot}$ and $\sCat_{/\cattwo}^\top$ are
  closed in $\sCat_{/\cattwo}$ under connected colimits, and similarly the subcategories
  $\sCptd_{/\cattwo}^{\bot}$ and $\sCptd_{/\cattwo}^\top$ are closed in
  $\sCptd_{/\cattwo}$ under connected colimits.
\end{lem}

\begin{proof}
  Note that the fibre functor $(-)^0\colon\sCat_{/\cattwo}\to\sCat$, which
  carries an object $\eA\to\cattwo$ to its fibre over $0\in\cattwo$, has a right
  adjoint and so preserves all colimits. Specifically, that right adjoint
  carries $\eB$ to the object $\eB^\top\to\cattwo$ where $\eB^\top$ is the
  simplicial category obtained by adjoining an terminal object $\top$ to $\eB$
  and its projection to $\cattwo$ is determined by the object mappings $B\mapsto
  0$ for objects $B\in\eB\subset\eB^\top$ and $\top\mapsto 1$. Now observe that
  $(-)^0$ maps the colimit $\eC$ of a connected diagram in
  $\sCat^{\bot}_{/\cattwo}$ to the colimit of a constant connected diagram at
  the terminal simplicial category $\catone$, but that colimit is isomorphic to
  $\catone$ so we have that $\eC^0\cong\catone$ and thus that $\eC$ is in the
  full subcategory $\sCat^{\bot}_{/\cattwo}$. This establishes the stated result
  for the full subcategory $\sCat_{/\cattwo}^{\bot}$ and the manifest dual
  argument also establishes it for $\sCat_{/\cattwo}^{\top}$. The second clause follows because $\sCptd_{/\cattwo}$ is closed in $\sCat_{/\cattwo}$ under all colimits.
\end{proof}

\begin{prop}\label{prop:fib-prod-restr}
  The fibred product functor on $\sCat_{/\cattwo}$ restricts to give a bifunctor
  \begin{equation}\label{eq:fib-prod-restr}
    \xymatrix@R=0em@C=5em{
      {\sCptd^\top_{/\cattwo}\times\sCptd^\bot_{/\cattwo}}
      \ar[r]^-{{-}\times_{\cattwo}{-}} &
      {\sCptd_{/\cattwo}}
    }
  \end{equation}
  which preserves connected colimits in each variable independently.
\end{prop}

\begin{proof}
Lemma~\ref{lem:fibre-prod-over-2-again} shows that the fibred product functor on $\sCat$ restricts to a bifunctor as displayed in~\eqref{eq:fib-prod-restr}. Note, however, that $\sCat$ is closed in $\Graph^{\Del\ep\op}$ under pullbacks and that $\sCptd$ is closed in $\Graph^{\Del\ep\op}$ under all colimits. Consequently the fibred product bi-functor~\eqref{eq:fib-prod-restr} is computed as in $\Graph^{\Del\ep\op}_{/\cattwo}$, as are the connected colimits of $\sCptd^\top_{/\cattwo}$, $\sCptd^\bot_{/\cattwo}$ and $\sCptd_{/\cattwo}$. Now $\Graph^{\Del\ep\op}$ is  a presheaf category, so it follows that the slice $\Graph^{\Del\ep\op}_{/\cattwo}$ is cartesian closed and thus that its fibred product bifunctor preserves colimits in each variable independently. So the desired result follows since the bifunctor and colimits discussed in the statement are computed as in $\Graph^{\Del\ep\op}$.
\end{proof}

\begin{cor}\label{cor:phi-lands-in-computad}
The bifunctor introduced in~\eqref{eq:join-gC-bifunctor-var-2} takes values in simplicial computads and  the component $\Phi_{X,Y}$ of the natural transformation displayed
  in~\eqref{eq:join-gC-bifunctor-var-3} is a simplicial computad morphism.
  In other words, the natural transformation
  in~\eqref{eq:join-gC-bifunctor-var-3} restricts along
  $\sCptd_{/\cattwo}\inc\sCat_{/\cattwo}$ to give a natural transformation:
  \begin{equation}\label{eq:Phi-rest}
    \xymatrix@R=0ex@C=9em{
      {\sSet\times\sSet}
      \ar@/^1.5ex/[]!R!UR(0.5);[r]!L!UL(0.5)^-{\gC[{-}\join{-}]}
      \ar@/_1.5ex/[]!R!DR(0.5);[r]!L!DL(0.5)_-{\gC[{-}\join\Del^0]\times_{\cattwo}
        \gC[\Del^0\join{-}]}\ar@{}[r]|{\Downarrow\Phi} &
      {\sCptd_{/\cattwo}}
    }
  \end{equation}
\end{cor}
\begin{proof}
  We may now return to Remark~\ref{rmk:join-gC-bifunctor-var} and observe that
  the functors in~\eqref{eq:join-gC-bifunctor-var-1} restrict to give functors:
  \begin{equation*}
    \xymatrix@R=0ex@C=5em{{\sSet}\ar[r]^-{\gC[{-}\join \Del^0]} &
      {\sCptd_{/\cattwo}^\top}}\mkern40mu
    \xymatrix@R=0ex@C=5em{{\sSet}\ar[r]^-{\gC[\Del^0\join{-}]} &
      {\sCptd_{/\cattwo}^\bot}}
  \end{equation*}
landing in the categories defined in Definition \ref{defn:bot-top-over-2}. 
  We know that the original functors in~\eqref{eq:join-gC-bifunctor-var-1}
  preserved connected colimits and Lemma~\ref{lem:closed-conn-colim} reveals
  that the subcategories $\sCptd_{/\cattwo}^\top$ and $\sCptd_{/\cattwo}^\bot$
  are closed under connected colimits in $\sCptd_{/\cattwo}$. It follows,
  therefore, that the functors in the last display preserve connected colimits.
  
  As a consequence we may apply Proposition~\ref{prop:fib-prod-restr}, to show
  that the bifunctor introduced in~\eqref{eq:join-gC-bifunctor-var-2} restricts
  to a bifunctor:
  \begin{equation*}
    \xymatrix@R=0ex@C=4em{
      {\sSet\times\sSet}\ar[rr]^-{\gC[{-}\join \Del^0]\times
        \gC[\Del^0\join{-}]} &&
      {\sCptd_{/\cattwo}^\top\times\sCptd_{/\cattwo}^\bot} \ar[r]^-{{-}\times_{\cattwo}{-}} &
      {\sCptd_{/\cattwo}}
    }
  \end{equation*}
  What is more the preservation results for connected colimits established in
  the last paragraph and in Proposition~\ref{prop:fib-prod-restr} together imply
  that this restricted bifunctor preserves connected colimits independently in
  each variable.

  Finally, the component $\Phi_{X,Y}$ of the natural transformation displayed
  in~\eqref{eq:join-gC-bifunctor-var-3} is induced by a pair of simplicial
  computad morphisms $\gC[X\join !]$ and $\gC[!\join Y]$ under the universal
  property of the fibred product $\gC[X\join\Del^0]\times_{\cattwo}
  \gC[\Del^0\join Y]$ in $\sCat_{/\cattwo}$.
  Lemma \ref{lem:fibre-prod-over-2-again} tells us, however, that the subcategory
  $\sCptd_{/\cattwo}$ is closed in $\sCat_{/\cattwo}$ under that fibred product,
  so it follows that the induced map $\Phi_{X,Y}$ is actually a simplicial
  computad morphism as claimed.
\end{proof}

  We shall extend our notation for standard simplices slightly by letting
  $\Del^{-1}$ denote the empty simplicial set.

\begin{prop}\label{prop:Phi-iso}
  For all $n,m\geq -1$ the simplicial computad morphism
  \begin{equation*}
    \Phi_{\Del^n,\Del^m}\colon\gC[\Del^n\join\Del^m]\to
    \gC[\Del^n\join\Del^0]\times_{\cattwo}\gC[\Del^0\join\Del^m]
  \end{equation*}
  is an isomorphism.
\end{prop}

\begin{proof}
  The morphism $\Phi_{n,m}\defeq\Phi_{\Del^n,\Del^m}$ has domain and codomain
  isomorphic to $\gC\Del^{n+m+1}$ and $\gC\Del^{n+1}\times_{\cattwo}
  \gC\Del^{m+1}$ respectively. By definition $\gC\Del^{n+1}$ is the full
  simplicial subcategory of $\gC\Del^{n+m+1}$ spanning its objects $0,\ldots,n+1$
  and we shall simplify our notation in the following argument by identifying
  $\gC\Del^{m+1}$ with the full simplicial subcategory of $\gC\Del^{n+m+1}$
  spanning its objects $n,\ldots,n+m+1$. We might note here that under these
  identifications the homotopy coherent simplex $\gC\Del^{n+m+1}$ is fibred over
  $\cattwo$ by the unique functor which acts on objects to map $0,\ldots,n\mapsto
  0$ and $n+1,\ldots,n+m+1\mapsto 1$, and that the full subcategories
  $\gC\Del^{n+1}$ and $\gC\Del^{m+1}$ are fibred by restricting that functor.

  To describe the component $\Phi_{n,m}$ with respect to these presentations, we
  introduce a pair of simplicial operators $\alpha,\beta\colon[n+m+1]\to[n+m+1]$
  defined by:
  \begin{align*}
    \alpha(i) &=
      \begin{cases}
        i & \text{if $i\leq n$, and} \\
        n+1 & \text{if $i\geq n+1$}
      \end{cases} &
    \beta(i) &= 
      \begin{cases}
        n & \text{if $i\leq n$, and} \\
        i & \text{if $i\geq n+1$.}
      \end{cases}
  \end{align*}
  Observe that the derived simplicial computad morphisms
  $\gC[\alpha],\gC[\beta]\colon \gC\Del^{n+m+1}\to\gC\Del^{n+m+1}$ map all
  objects into the respective full subcategories $\gC\Del^{n+1}$ and
  $\gC\Del^{m+1}$ under the specified renumbering of the vertices of the latter.
  Consequently the induced morphism
  \begin{equation*}
    \xymatrix@R=0em@C=6em{
      {\gC\Del^{n+m+1}}\ar[r]^-{(\gC[\alpha],\gC[\beta])} &
      {\gC\Del^{n+m+1}\times_{\cattwo}\gC\Del^{n+m+1}}
    }
  \end{equation*}
  factors through the subcategory $\gC\Del^{n+1}\times_{\cattwo}\gC\Del^{m+1}$
  and the resulting morphism is $\Phi_{n,m}$. We can make the action of
  $\Phi_{n,m}$ more explicit by recalling that if $k$ and $l$ are
  objects of $\gC\Del^{n+m+1}$ with $k\leq l$ then $0$-arrows $U\colon k\to l$
  correspond to subsets $U\subseteq [k,l] = \{i\in\mathbb{N}\mid k\leq i\leq
  l\}$ with $\{k,l\}\subseteq U$, and that $r$-arrows $U^\bullet\colon i\to j$
  correspond to ordered sequences $U^0\subseteq U^1\subseteq \cdots\subseteq
  U^r$ of $0$-arrows. The functors $\gC[\alpha]$ and $\gC[\beta]$ act on the
  $0$-arrow $U\colon k\to l$ by mapping it to $\alpha(U)=\{\alpha(i)\mid i\in
  U\}$ and $\beta(U)=\{\beta(i)\mid i\in U\}$ respectively and it is easily
  checked, from the definitions of the operators $\alpha$ and $\beta$, that we
  have:
  \begin{equation}\label{eq:gC-ab-formulae}
    \begin{aligned}
      \alpha(U) &=
      \begin{cases}
        U & \text{if $l\leq n$, and} \\
        (U\cap[0,n])\cup\{n+1\} & \text{if $l\geq n+1$,}
      \end{cases}\\
      \beta(U) &=
      \begin{cases}
        U & \text{if $k\geq n+1$, and} \\
        \{n\}\cup(U\cap[n+1,n+m+1]) & \text{if $k\leq n$.}
      \end{cases}
    \end{aligned}
  \end{equation}
  The objects of $\gC\Del^{n+1}\times_{\cattwo}\gC\Del^{m+1}$ are pairs of
  integers of the form $(i,n)$ for $i=0,\ldots,n$ or $(n+1,j)$ for $j=n+1,\ldots, n+m+1$ and its $0$-arrows
  are pairs $(S,T)$ of $0$-arrows of $\gC\Del^{n+m+1}$ of one of the following
  three forms:
  \begin{itemize}
  \item $(S,T)\colon(i,n)\to(i',n)$, in which case $S\colon i\to i'$ is a
    $0$-arrow with $0\leq i\leq i'\leq n$ and $T=\{n\}\colon n\to n$ (the
    identity $0$-arrow on $n$),
  \item $(S,T)\colon(i,n)\to(n+1,j)$, in which case $S\colon i\to n+1$ is a
    $0$-arrow with $0\leq i\leq n$ and $T\colon n\to j$ is a $0$-arrow with
    $n+1\leq j\leq n+m+1$, or
  \item $(S,T)\colon(n+1,j)\to(n+1,j')$, in which case $T\colon j\to j'$ is a
    $0$-arrow with $n+1\leq j\leq j'\leq n+m+1$ and $S=\{n+1\}\colon n+1\to n+1$
    (the identity $0$-arrow on $n+1$).
  \end{itemize}
  Given such a $0$-arrow $(S,T)$ in $\gC\Del^{n+1}\times_{\cattwo}\gC\Del^{m+1}$
  define a set $S\oplus T\defeq (S\setminus\{n+1\})\cup(T\setminus\{n\})$ and
  then for each of the cases adumbrated above it is easily checked that $S\oplus
  T$ is a $0$-arrow of $\gC\Del^{n+m+1}$. Arguing case-by-case, it is also
  routine to apply the formulae in~\eqref{eq:gC-ab-formulae} to establish the
 equalities $\alpha(S\oplus T)=S$ and $\beta(S\oplus T)=T$ and to show that for
  all $0$-arrows $U$ of $\gC\Del^{n+m+1}$ we have $U=\alpha(U)\oplus\beta(U)$.
  In this way we have demonstrated that $S\oplus T$ is the \emph{unique}
  $0$-arrow of $\gC\Del^{n+m+1}$ which is mapped by $\Phi_{n,m}$ to the
  $0$-arrow $(S,T)$ of $\gC\Del^{n+1}\times_{\cattwo}\gC\Del^{m+1}$. Finally
  observe that the $r$-arrows of $\gC\Del^{n+m+1}$ and
  $\gC\Del^{n+1}\times_{\cattwo}\gC\Del^{m+1}$ are uniquely determined by the
  ordered sequences of $0$-arrows that comprise their vertices and that
  $\alpha$, $\beta$ and $\oplus$ all preserve the subset inclusion ordering
  between $0$-arrows. So we may infer from the fact that $\Phi_{n,m}$ acts
  bijectively on $0$-arrows that it also acts bijectively on $r$-arrows, and
  thus that it is an isomorphism of simplicial computads as required.
\end{proof}

We now extend the  result of Proposition \ref{prop:Phi-iso} to show that $\Phi_{X,Y}$ is an isomorphism
  for all pairs of simplicial sets $X,Y$. As one might expect, we do this by
  expressing $X$ and $Y$ as colimits of standard simplices $\Del^n$. We will,
  however, need to be careful to use representing colimits that are
  \emph{connected\/} since we will need them to be preserved by the bifunctors
  in~\eqref{eq:Phi-rest}.

\begin{thm}\label{thm:Phi-iso}
  For all simplicial sets $X,Y\in\sSet$ the simplicial computad morphism
  \begin{equation*}
    \Phi_{X,Y}\colon\gC[X\join Y]\to
    \gC[X\join\Del^0]\times_{\cattwo}\gC[\Del^0\join Y]
  \end{equation*} is an isomorphism.
\end{thm}


\begin{proof}
  Given a simplicial set $Z$, let $\el(Z)$ denote its category of elements and
  recall that the Yoneda lemma implies that we may canonically present $Z$ as a
  colimit of the diagram $d^Z\colon\el(Z)\stackrel{\text{proj}}\longrightarrow
  \Del \stackrel{\Del^\bullet} \longrightarrow\sSet$ of standard simplices.
  Indeed we may express $Z$ as a \emph{connected\/} colimit of standard
  simplices. To do so, form a connected category $\el(Z)_\bot$ by adjoining an
  initial object $\bot$ to $\el(Z)$ and extend the diagram $d^Z$ to a diagram on
  $\el(Z)_\bot$ by mapping $\bot$ to the initial object $\Del^{-1}=\emptyset$ in $\sSet$.
  Now observe that, since $\Del^{-1}$ is initial, any cocone under the original
  diagram extends uniquely to a cocone under the extended diagram, consequently
  the colimit of the original and extended diagrams coincide; they are both
  isomorphic to $Z$.

In this way, the simplicial sets $X$ and $Y$ may be
  expressed as connected colimits of diagrams $d_X\colon \el(X)_\bot\to\sSet$
  and $d_Y\colon\el(Y)_\bot\to\sSet$ of standard simplices. Composing these with
  the natural transformation in~\eqref{eq:Phi-rest} we obtain a natural
  transformation
  \begin{equation}\label{eq:Phi-iso-diags}
    \xymatrix@R=0ex@C=12em{
      {\el(X)_\bot\times\el(Y)_\bot}
      \ar@/^1.5ex/[]!R!UR(0.5);[r]!L!UL(0.5)^-{\gC[d_X(-)\join d_Y(-)]}
      \ar@/_1.5ex/[]!R!DR(0.5);[r]!L!DL(0.5)_-{\gC[d_X(-)\join\Del^0]\times_{\cattwo}
        \gC[\Del^0\join d_Y(-)]}\ar@{}[r]|(0.6){\Downarrow\Phi_{d_X(-),d_Y(-)}} &
      {\sCptd_{/\cattwo}}
    }
  \end{equation}
  whose components are instances of the morphisms analysed in
  Proposition~\ref{prop:Phi-iso} and are thus all isomorphisms.

  Now the bifunctors in~\eqref{eq:Phi-rest} preserve connected colimits in each
  variable independently and colimits commute with colimits. So it
  follows that in~\eqref{eq:Phi-iso-diags} the colimit of the upper functor is
  isomorphic to $\gC[X\join Y]$ and the colimit of the lower functor is
  $\gC[X\join\Del^0]\times_{\cattwo}\gC[\Del^0\join Y]$. What is more, it is a
  consequence of the naturality of $\Phi$ that the morphism induced between
  those colimits by the isomorphic transformation in~\eqref{eq:Phi-iso-diags} is
  is the component $\Phi_{X,Y}$, which is thus an isomorphism as postulated.
\end{proof}

The form of the simplicial category $\gC[X\join\Del^0]\times_{\cattwo}\gC[\Del^0\join Y]$ appearing on the right-hand side of the isomorphism of Theorem \ref{thm:Phi-iso} admits a simplified concrete description that we shall exploit in the next section:

\begin{rmk}\label{rmk:fib-prod-simpl}
For any $\eA$ in $\sCat^\top_{/\cattwo}$ and $\eB$ in
  $\sCat^\bot_{/\cattwo}$, we have that $(\eA\times_{\cattwo}\eB)^0\cong
  \eA^0$ and $(\eA^1\times_{\cattwo}\eB^1)\cong\eB^1$, and it follows that we
  can give the following simplified concrete description of
  $\eA\times_{\cattwo}\eB$, it has:
  \begin{itemize}
  \item Objects the disjoint union of the sets of objects of $\eA^0$ and
    of $\eB^1$,
  \item Hom spaces
    \begin{gather*}
      \Map_{\eA\times_{\cattwo}\eB}(a',a) \defeq \Map_{\eA}(a',a)\mkern60mu
      \Map_{\eA\times_{\cattwo}\eB}(b,b') \defeq \Map_{\eB}(b,b')\\
      \Map_{\eA\times_{\cattwo}\eB}(a,b) \defeq \Map_{\eA}(a,\top)\times
      \Map_{\eB}(\bot,b)
    \end{gather*}
    with all others empty, and
  \item Composites computed in the manifest way in $\eA$ or $\eB$.
  \end{itemize}
\end{rmk}



\section{Complete and cocomplete quasi-categories}\label{sec:complete}

Any quasi-category is equivalent to the homotopy coherent nerve of a Kan-complex-enriched category, so when discussing limits or colimits in a quasi-category $\qC$, it suffices to consider the case where $\qC=\hN\eC$ for a Kan-complex-enriched category $\eC$. In \S\ref{sec:lims-in-nerves}, we give a condition for a diagram $d \colon Y \to \qC$ to admit a limit (or dually a colimit): namely, $\eC$ needs to admit a \emph{pseudo homotopy limit}, a homotopy weighted limit in the sense of Definition \ref{defn:flexible-hty-limit} for the weight $W_Y$ of Definition \ref{defn:weight-for-pseudo-limits}, of the corresponding diagram $D \colon \gC{Y} \to \eC$. This is Theorem \ref{thm:nerve-completeness}.

In particular, if $\qC$ is the homotopy coherent nerve of a Kan-complex-enriched category $\eC$ that admits all flexible weighted homotopy limits then this condition holds for any diagram, and we conclude that the quasi-category $\qC$ is complete. Several families of examples of this form are established in \S\ref{ssec:complete-qcats}. 

One example we consider is the projectively fibrant and cofibrant simplicial presheaves indexed by a small Kan-complex-enriched category $\eC$, which gives rise to a complete and cocomplete quasi-category $\hat\qC$. By analysing the simplicially enriched Yoneda embedding $\yoneda\colon\eC\to\hat\eC$, we see that the sufficient condition for a diagram valued in $\qC$ to admit a limit or colimit is also necessary, as we show in Theorem  \ref{thm:nerve-completeness-converse}. Together Theorems \ref{thm:nerve-completeness} and \ref{thm:nerve-completeness-converse} recover Lurie's \cite[4.2.4.1]{Lurie:2009fk}, but our proof is somewhat more explicit. In the proof below, we employ a very particular model for the point-set level homotopy limit cone in a Kan-complex-enriched category, the data of which precisely translates to a limit cone in the quasi-category defined as the homotopy coherent nerve.

\subsection{Computing (co)limits in homotopy coherent nerves}\label{sec:lims-in-nerves}

In this subsection let $\eC$ denote a (locally small) \emph{Kan-complex-enriched
  category}, that is to say a simplicially enriched category that is ``locally
Kan'' in the sense that each of its mapping spaces $\Map_{\eC}(A,B)$ is a Kan complex.
Let $\qC$ denote the quasi-category obtained as the homotopy
coherent nerve of $\eC$. For any small simplicial set $Y$ and a diagram $d\colon
Y\to \qC$ our aim will be to develop a condition on the corresponding homotopy coherent diagram $d \colon \gC{Y} \to \eC$ which ensures that $\qC$ admits a limit of the given diagram.

  Example~\ref{ex:limits-in-quasi-categories} tells us that the diagram $d\colon Y\to \qC$ admits a limit in $\qC$ if and only if there exists some cone $\lambda$ over $d$ as depicted in diagram below-left, which enjoys the lifting property below-right:
  \begin{equation}\label{eq:limit-lift-prob'}
    \xymatrix@R=2em@C=3em{ Y \ar[r]^d \ar@{^(->}[d]  & \qC & & 
      {\Del^0\join Y}\ar@/^2ex/[]!R!UR(0.5);[rr]!L!UL(0.5)^-{\lambda}
      \ar[r]_-{\fbv{n}\join Y} &
      {\boundary\Del^n\join Y}\ar[r]_-{f}\ar@{^(->}[d] & {\qC} \\  \Del^0 \join Y \ar@{-->}[ur]_\lambda & & & 
      {} & {\Del^n\join Y}\ar@{-->}[ur] & {}
    }
  \end{equation}
  for all $n\geq 1$ and all $f\colon \boundary\Del^n\join Y\to \qC$ with
  $f\vert_{\fbv{n}\join Y} = \lambda$. Taking transposes under the homotopy
  coherent nerve $\dashv$ realisation adjunction, we obtain a cone $\Lambda\colon
  \gC[\Del^0\join Y]\to \eC$ and an equivalent lifting property:
  \begin{equation}\label{eq:limit-lift-prob}
    \xymatrix@R=2em@C=3.5em{ \gC[Y] \ar[r]^D \ar@{^(->}[d] & \eC & 
      {\gC[\Del^0\join Y]}\ar@/^2ex/[]!R!UR(0.5);[rr]!L!UL(0.5)^-{\Lambda}
      \ar[r]_-{\gC[\fbv{n}\join Y]} &
      {\gC[\boundary\Del^n\join Y]}\ar[r]_(0.55){F}\ar@{^(->}[d] &
      {\eC} \\ \gC[\Del^0\join Y] \ar@{-->}[ur]_{\Lambda} & &
      {} & {\gC[\Del^n\join Y]}\ar@{-->}[ur] & {}
    }
  \end{equation}
  
Recall from Corollary \ref{cor:collage-bijection} that every simplicial functor $\Lambda\colon\gC[\Del^0\join Y]\to \eC$ gives rise to a $W_Y$-weighted cone 
  \begin{equation}\label{eq:pseudo-cone-as-nat}
    \mu_\Lambda \defeq
    \xymatrix@R=0em@C=4em{
      {W_Y = {\gC[\Del^0\join Y]}(\bot,-)}\ar[r]^-{\Lambda_{\bot,-}} &
      {\Map_{\eC}(\Lambda\bot,\Lambda-)}
    }
  \end{equation}
  with apex $\Lambda\bot$ over the restricted diagram $\gC[Y]\subset\gC[\Del^0\join Y] \stackrel{\Lambda}{\longrightarrow} \eC$, where $W_Y$ is the weight whose collage is $\gC[\Del^0\join Y]$ introduced in Definition \ref{defn:weight-for-pseudo-limits}. This construction provides a bijection between simplicial functors $\Lambda\colon\gC[\Del^0\join Y]\to \eC$ and pairs $(\Lambda_Y,\mu_\Lambda)$ comprising a diagram $\Lambda_Y\colon\gC[Y]\to\eC$ and a $W_Y$-weighted cone $\mu_\Lambda$ over that diagram. Henceforth, we refer to the simplicial natural transformation \eqref{eq:pseudo-cone-as-nat} and the associated simplicial functor $\Lambda \colon \gC[\Del^0\join Y] \to \eC$ equally as \emph{pseudo weighted cones}.
  
We are now prepared to state our main theorem:

\begin{thm}\label{thm:nerve-completeness}
For any Kan-complex-enriched category $\eC$, simplicial set $Y$, and homotopy coherent diagram $D \colon \gC[Y] \to \eC$, if $D$ admits a pseudo homotopy limit in $\eC$ then the $W_Y$-weighted limit cone $\Lambda \colon \gC[\Del^0\join Y] \to \eC$ transposes to define a limit cone over the transposed diagram $d \colon Y \to \qC \defeq \hN\eC$ in the homotopy coherent nerve of $\eC$. Consequently if $\eC$ admits pseudo homotopy limits for all simplicial sets $Y$ then the quasi-category $\qC$ admits all limits.
\end{thm}

To prove Theorem \ref{thm:nerve-completeness}, we show that the extension \eqref{eq:limit-lift-prob} can be constructed whenever the cone $\Lambda \colon \gC[\Del^0\join Y] \to \eC$ is a pseudo homotopy limit cone, satisfying the condition of Definition \ref{defn:flexible-hty-limit} for the weight $W_Y$ for pseudo limits. The next result gives a first step towards analysing such situations:

\begin{prop}\label{prop:simp-func-from-join} If $\eC$ admits the $W_Y$-weighted limit of a diagram $F_Y\colon\gC Y\to\eC$, then simplicial functors $F\colon\gC[X\join Y]\to\eC$ that extend $F_Y\colon\gC[Y]\to\eC$ stand in bijective correspondence to simplicial functors $F_X\colon\gC[X\join\Del^0]\to \eC$ that map $\top$ to the weighted limit $\wlim{W_Y}{F_Y}$.
\end{prop}

\begin{proof}  Under the isomorphism $\Phi_{X,Y}\colon\gC[X\join Y]\cong \gC[X\join\Del^0]\times_{\cattwo}\gC[\Del^0\join Y]$ of Theorem \ref{thm:Phi-iso} and the description of $\gC[X\join\Del^0]\times_{\cattwo}\gC[\Del^0\join Y]$ given in Remark~\ref{rmk:fib-prod-simpl}, we find that a simplicial functor $F\colon\gC[X\join Y]\to\eC$ is uniquely determined by giving the following data:
  \begin{enumerate}[label=(\roman*)]
  \item a pair of simplicial functors $F_X\colon\gC[X]\to\eC$ and $F_Y\colon\gC[Y]\to\eC$, and
  \item a family of simplicial maps
    \begin{equation}\label{eq:simp-func-from-join}
      \xymatrix@R=0ex@C=4em{
        {{\gC[X\join\Del^0]}(x,\top)\times\Map_{\gC[\Del^0\join Y]}(\bot,y)}
        \ar[r]^-{F_{x,y}} & {\Map_{\eC}(F_Xx, F_Yy)}
      }
    \end{equation}
    that is simplicially natural in $x\in\gC X$ and $y\in\gC Y$.
  \end{enumerate}
Transposing the family of maps in~\eqref{eq:simp-func-from-join} and taking the manifest end with respect to $y\in\gC Y$ reduces it to a family of simplicial maps
  \begin{equation}\label{eq:simp-func-from-join'}
    \xymatrix@R=0em@C=2em{
      {{\gC[X\join\Del^0]}(x,\top)}\ar[r]^-{F_{x,*}} & \int_{y \in \gC{Y}} \Map_{\eC}(F_Xx,F_Yy)^{{\gC[\Del^0\join Y]}(\bot, y)} \eqdef \wlim{W_Y}{ \Map_{\eC}(F_Xx,F_Y-)}   }
  \end{equation}
  which is simplicially natural in $x\in\gC X$. In the particular case where $\eC$ admits a $W_Y$-weighted limit of the diagram $F_Y\colon\gC Y\to\eC$, the codomain of the family in~\eqref{eq:simp-func-from-join'} is, by definition, isomorphic to $\Map_{\eC}(F_Xx,\wlim{W_Y}{F_Y})$. So then simplicial functors $F\colon\gC[X\join Y]\to\eC$ that extend $F_Y\colon\gC[Y]\to\eC$ stand in bijective correspondence to simplicial functors $F_X\colon\gC[X\join\Del^0]\to \eC$ that map $\top$ to the limit $\wlim{W_Y}{F_Y}$.
\end{proof}

With this result in hand, we may return to analysing the lifting problem \eqref{eq:limit-lift-prob}. The next result characterises solutions to lifting problems generalising those of the form \eqref{eq:limit-lift-prob}.

\begin{prop}\label{prop:lift-analysis}
Solutions to lifting problems of the form
  \begin{equation}\label{eq:cone-lift}
    \xymatrix@R=2em@C=6em{
      {\gC[\boundary\Del^n\join Y]} \ar@{^(->}[d]\ar[r]^-{F} &
      {\eC}\ar[d]^{P} \\
      {\gC[\Del^n\join Y]}\ar[r]_-{G}\ar@{-->}[ur]_{L} & {\eD}
    }
  \end{equation}
  correspond bijectively to a pair of lifts making the following diagram commute:
  \begin{equation}\label{eq:unpacked-cone-lift} \xymatrix@C=7pt@R=5pt{ {\gC[\Horn^{n+1,n+1}]}(0,n)  \ar[dr] \ar@{^(->}[dd] \ar[rr] & & \Map_{\eC}(F0,Fn) \ar'[d][dd] \ar[dr] \\ &{\gC[\Horn^{n+1,n+1}]}(0,n+1)  \ar[rr]\ar@{^(->}[dd] & & \wlim{W_Y}{\Map_{\eC}(F0,F-)} \ar[dd] \\ {\gC[\Del^{n+1}]}(0,n) \ar[dr] \ar@/_3ex/@{-->}[]!R!U(0.4);[uurr] \ar[rr] & & \Map_{\eD}(G0,Gn) \ar[dr] \\ & {\gC[\Delta^{n+1}]}(0,n+1)  \ar@/_2.5ex/@{-->}[]!R!U(0.4);[uurr] \ar[rr] & &  \wlim{W_Y}{\Map_{\eD}(G0,G-)}}
  \end{equation}
\end{prop}

  \begin{proof}
Under the isomorphism established in Theorem~\ref{thm:Phi-iso} the vertical inclusion at the left of \eqref{eq:cone-lift} is isomorphic to the inclusion \[\gC[\boundary\Del^n\join\Del^0]\times_{\cattwo}\gC[\Del^0\join Y]\inc \gC[\Del^n\join\Del^0]\times_{\cattwo}\gC[\Del^0\join Y].\] Furthermore the canonical isomorphism $\Del^n\join\Del^0\cong\Del^{n+1}$ restricts to an isomorphism $\boundary\Del^n\join\Del^0\cong\Horn^{n+1,n+1}$, so this inclusion is in turn isomorphic to that obtained by applying the functor $-\times_{\cattwo}\gC[\Del^0\join Y]$ to the inclusion $\gC\Horn^{n+1,n+1}\inc\gC\Del^{n+1}$. Also Example~\ref{ex:subcomputad-of-simplex}\ref{itm:outer-horn} tells us that the simplicial subcategory $\gC\Horn^{n+1,n+1}$ differs from $\gC\Del^{n+1}$ only in as much as its functor spaces ${\gC\Horn^{n,n}}(0,n)$ and ${\gC\Horn^{n,n}}(0,n+1)$ are proper sub-spaces of the corresponding functor spaces of $\gC\Del^{n+1}$. On combining this fact with the description of simplicial functors with domain $\gC[X\join Y]$ given in Proposition~\ref{prop:simp-func-from-join}, we find that each lift in~\eqref{eq:cone-lift} corresponds to a pair consisting of a solution to the lifting problems
\[
    \xymatrix@R=2em@C=1.8em{
      {{\gC\Horn^{n+1,n+1}}(0,n)}
      \ar@{^(->}[d]\ar[r]^{F_{0,n}} & {\Map_{\eC}(F0,Fn)}\ar[d]^{P}\\
      {{\gC\Del^{n+1}}(0,n)}\ar[r]_{G_{0,n}}
      \ar@{-->}[ur]^{ L_{0,n}} & {\Map_{\eD}(G0,Gn)}
    }
        \xymatrix@R=2em@C=1.8em{
      {{\gC\Horn^{n+1,n+1}}(0,n+1)}
      \ar@{^(->}[d]\ar[r]^-{F_{0,*}} &
      {\wlim{W_Y}{\Map_{\eC}(F0,F-)}}\ar[d]^{\wlim{W_Y}{P}} \\
      {{\gC\Del^{n+1}}(0,n+1)}
      \ar@{-->}[ur]^(0.4){ L_{0,*}}\ar[r]_-{G_{0,*}} &
      {\wlim{W_Y}{\Map_{\eD}(G0,G-)}}
    }
\]
  together satisfying the condition that the following square commutes:
  \begin{equation*}
    \xymatrix@R=2em@C=4em{
      {{\gC\Del^{n+1}(0,n)}}\ar@{^(->}^-{\{n,n+1\}\circ-}[r]
      \ar@{-->}[d]_{L_{0,n}} &
      {{\gC\Del^{n+1}}(0,n+1)}\ar@{-->}[d]^{L_{0,*}} \\
      {\Map_{\eC}(F0,Fn)}\ar[r]_-{\Lambda_F\circ -} &
      {\wlim{W_Y}{\Map_{\eC}(F0,F-)}}
    }
  \end{equation*}
  Here the horizontal map at the bottom of this square is that induced by post-composition with the $W_Y$-weighted cone $\Lambda_F\colon W_Y\to \Map_{\eC}(Fn,F-)$ determined by the composite simplicial functor
  \begin{equation*}
    \xymatrix@R=0em@C=6em{
      {\gC[\Del^0\join Y]}\ar[r]^{\gC[\{n\}\join Y]} &
      {\gC[\boundary\Del^n\join Y]}\ar[r]^-{F} & {\eC}
    }
  \end{equation*}
  as in Corollary \ref{cor:collage-bijection}.
\end{proof}

The unpacked lifting property of Proposition \ref{prop:lift-analysis} may be simplified once we replace the function spaces on the left of the lifting problem~\eqref{eq:unpacked-cone-lift} with their explicit presentations discussed in Example~\ref{ex:subcomputad-of-simplex}, as follows:
  \begin{equation*}
    \vcenter{\xymatrix@R=1.5em@C=3em{
        {{\gC\Horn^{n+1,n+1}}(0,n)}
        \ar[r]^-{{\{n,n+1\}\circ-}}\ar@{^(->}[d] &
        {{\gC\Horn^{n+1,n+1}}(0,n+1)}\ar@{^(->}[d] \\
        {{\gC\Del^{n+1}}(0,n)}
        \ar[r]_-{{\{n,n+1\}\circ-}} &
        {{\gC\Del^{n+1}}(0,n+1)}
      }}\mkern20mu\cong\mkern20mu
    \vcenter{\xymatrix@R=1.5em@C=3em{
        {\boundary\Cube^{n-1}}\ar@{^(->}[r]\ar@{^(->}[d] &
        {\CHorn^{n,n}_0}\ar@{^(->}[d] \\
        {\Cube^{n-1}}\ar@{^(->}[r]_{\Cube^{n-1}\times\{1\}} &
        {\Cube^n}
      }}
  \end{equation*}
  This leads us to the following technical lemma:

\begin{lem}\label{lem:coherent-horn-lifting} There exists a bijective correspondence between lifting problems, and their solutions, of the forms depicted in the following display:
  \begin{equation}\label{eq:double-lift-equiv}
    \vcenter{\xymatrix@R=7pt@C=18pt{
        \boundary\Cube^{n-1} \ar@{^(->}[dr] \ar@{^(->}[dd] \ar[rr] & &
        A \ar[dr]^-{h}\ar[dd]|!{[dl];[dr]}\hole_(0.68){p} \\
        & \CHorn^{n,n}_0 \ar[rr] \ar@{^(->}[dd] & & B\ar[dd]^{q} \\
        \Cube^{n-1} \ar@{^(->}[dr]_{\Cube^{n-1}\times\fbv{1}}
        \ar@/_2.5ex/@{-->}[uurr]\ar [rr]|!{[ur];[dr]}\hole
        & & C\ar[dr]^-{k} \\
        & \Cube^{n} \ar@/_2ex/@{-->}[uurr]
        \ar[rr] & & D }}
    \qquad \leftrightsquigarrow \qquad
    \vcenter{\xymatrix{
        \boundary\Cube^{n-1} \ar@{^(->}[d] \ar[r] &
        B \comma h \ar[d]^-{\comma(B,q,p)} \\
        \Cube^{n-1} \ar[r]\ar@{-->}[ur] & q\comma k}}
  \end{equation}
\end{lem}

\begin{proof}
  Recall that $\Cube^n\cong\Cube^{n-1}\times\Del^1$ which restricts to sub-spaces to give isomorphisms $\CHorn^{n,n}_0 \cong \partial \Cube^{n-1} \times \Del^1 \cup \Cube^{n-1} \times \Del^{\fbv{0}}$; furthermore we also have that $\boundary \Cube^{n-1} \times \Del^1 \cap \Cube^{n-1} \times \Del^{\fbv{0}}=\boundary\Cube^{n-1}\times\Del^{\fbv{0}}$. So the left-hand vertical of the front square of the cube in~\eqref{eq:double-lift-equiv} is simply the Leibniz product $(\boundary\Cube^{n-1}\inc\Cube^{n-1})\leib\times(\Del^{\fbv{0}}\inc\Del^1)$. The Leibniz exponential of $\Del^{\fbv{0}}\inc\Del^1$ with $q \colon B \to D$ defines the map $\comma(B,q,q)$ of Proposition \ref{prop:trans-comma}, so  transposing across the two variable adjunction between Leibniz products and Leibniz exponential gives us a bijective correspondence between the lifting problem at the front of the cube in~\eqref{eq:double-lift-equiv} and the transposed lifting problem on the left of the following display:
  \begin{equation*}
    \vcenter{\xymatrix@=2em{
        {\boundary\Cube^{n-1}}\ar@{^(->}[d]\ar[r] &
        {B^{\cattwo}}\ar[d]^{\comma(B,q,q)} \\
        {\Cube^{n-1}}\ar[r]\ar@{-->}[ur] & {q\comma D}
      }}\mkern60mu
    \vcenter{\xymatrix@=2em{
        {\Cube^{n-1}}\ar@{-->}[r]\ar@{-->}[d] &
        {B^{\cattwo}}\ar[d]^{p_1} \\
        {A}\ar[r]_{h} & {B}
      }}
  \end{equation*}
  Under this transposition, the square on the right re-expresses the commutativity of the square within the cube on the left of~\eqref{eq:double-lift-equiv} comprised of the dashed lifts and the maps between their (co)domains. This is the only compatibility condition required of those lifts, so taking the pullback in that square we reduce those lifts and their compatibility condition to a single lift $\Cube^{n-1}\dashrightarrow B\comma h$. In a similar fashion, the commutativity of the squares comprising the lower and upper faces of the cube in~\eqref{eq:double-lift-equiv} transpose to conditions that posit the commutativity of squares
  \begin{equation*}
    \vcenter{\xymatrix@=2em{
        {\Cube^{n-1}}\ar[r]\ar[d] &
        {q\comma D}\ar[d]^{p_1} \\
        {C}\ar[r]_{k} & {D}
      }}\mkern25mu\text{and}\mkern25mu
    \vcenter{\xymatrix@=2em{
        {\boundary\Cube^{n-1}}\ar[r]\ar[d] &
        {B^{\cattwo}}\ar[d]^{p_1} \\
        {A}\ar[r]_{h} & {B}
      }}
  \end{equation*}
respectively. Taking pullbacks in these squares, we again see that they correspond to maps $\Cube^{n-1}\to q\comma k$ and $\boundary\Cube^{n-1}\to B\comma h$. It is now a routine matter to see that the commutativity conditions encapsulated in the sides of the cube on the left of~\eqref{eq:double-lift-equiv} correspond to the commutativity of the square on its right in which the maps going from its left-vertical to the right-vertical are those just constructed.
\end{proof}
 
Note that the lifting problem on the right of \eqref{eq:double-lift-equiv} can be solved whenever the map $\comma(B,q,p) \colon B \comma h \to q \comma k$ is a trivial fibration. The following lemma characterises those squares for which this is the case.

\begin{lem}\label{lem:coherent-horn-lifting-suff}
Consider a commutative square whose verticals are Kan fibrations between Kan complexes:
  \begin{equation*}
    \xymatrix@=2em{
      {\qA}\ar[r]^h\ar@{->>}[d]_{p} &
      {\qB}\ar@{->>}[d]^q \\
      {\qC}\ar[r]_{k} & {\qD}
    }
  \end{equation*}
Then the induced map $\comma(\qB,q,p)\colon \qB\comma h\to q\comma k$ is a trivial fibration if and only if the given square is a homotopy pullback, in the sense that the map $\qA \we\qB\times_{\qD}\qC$ is an equivalence. \end{lem}

\begin{proof}
  Under conditions of the statement, Proposition~\ref{prop:trans-comma} applied in the $\infty$-cosmos of Kan complexes of Example \ref{ex:Kan-infty-cosmos} demonstrates that the map $\comma(\qB,q,p)\colon \qB\comma h\to q\comma k$ is a Kan fibration between Kan complexes. Now $\qB$ and $\qD$ are Kan complexes so it follows that the comma objects $\qB\comma h$ and $q\comma k$ are homotopy pullbacks of top and bottom cospans of the diagram
 \[
  \xymatrix@1{{\qA} \ar[r]^h  \ar@{->>}[d]_p & {\qB} \ar@{->>}[d]^q & {\qB}\ar@{=}[l] \ar@{=}[d] \\ {\qC}\ar[r]^{k} & {\qD} & {\qB}\ar@{->>}[l]_{q}}\] respectively. The right-hand legs of those cospans are Kan fibrations, so their pullbacks are also homotopy pullbacks and those pullback squares induce equivalences $\qA\we\qB\comma h$ and $\qB\times_{\qD}\qC\we q\comma k$. By the universal properties that define these maps,  these equivalences fit into a commutative square
  \begin{equation*}
    \xymatrix@R=1.8em@C=4em{
      {\qA}\ar[r]^-{(h,p)}\ar[d]_{\simeq} &
      {\qB\times_{\qD}\qC}\ar[d]^{\simeq} \\
      {\qB\comma h}\ar@{->>}[r]_-{\comma(\qB,q,p)} &
      {q\comma k}
    }
  \end{equation*}
  from which the stated result follows by application of the laws of composition and cancellation of equivalences.
\end{proof}

Combining these results, we conclude

\begin{cor}\label{cor:coherent-horn-lifting}
If $P \colon \eC \to \eD$ is a simplicial functor between Kan-complex-enriched categories that is a levelwise Kan fibration, then a lifting problem
\[    \xymatrix@R=2em@C=6em{
      {\gC[\boundary\Del^n\join Y]} \ar@{^(->}[d]\ar[r]^-{F} &
      {\eC}\ar[d]^{P} \\
      {\gC[\Del^n\join Y]}\ar[r]_-{G}\ar@{-->}[ur]_{L} & {\eD}
    }\]
 has a solution under the condition that the square
  \begin{equation}\label{eq:coh-horn-lift-htypb}
    \xymatrix@R=2em@C=5em{
      {\Map_{\eC}(F0,Fn)}\ar@{->>}[d]_{P}\ar[r]^-{\Lambda_F\circ{-}} &
      {\wlim{W_Y}{\Map_{\eC}(F0,F-)}}\ar@{->>}[d]^{\wlim{W_Y}{P}}\\
      {\Map_{\eD}(G0,Gn)}\ar[r]_-{\Lambda_G\circ{-}} &
      {\wlim{W_Y}{\Map_{\eD}(G0,G-)}}
    }
  \end{equation}
  is a homotopy pullback of Kan complexes. 
\end{cor}

\begin{proof}
Proposition~\ref{prop:lift-analysis} shows that the lifting problem in~\eqref{eq:cone-lift} has a solution if and only if a lifting problem of the kind analysed in Lemma~\ref{lem:coherent-horn-lifting} with right hand square~\eqref{eq:coh-horn-lift-htypb} has a solution. Note also that $W_Y$ is a flexible weight so it follows, by Proposition~\ref{prop:flexible-weights-are-htpical}, that the weighted limits on the right of~\eqref{eq:coh-horn-lift-htypb} are Kan complexes and the map between them is a Kan fibration. Consequently we may apply Lemma~\ref{lem:coherent-horn-lifting-suff} to conclude that this latter lifting problem has a solution so long as the square~\eqref{eq:coh-horn-lift-htypb} is a homotopy pullback of Kan complexes as stated. 
\end{proof}

As a special case, we now have a criterion which allows us to solve the lifting problem \eqref{eq:limit-lift-prob}:

\begin{cor}\label{cor:cone-lift}
  Suppose that $\eC$ is a Kan-complex-enriched category and that $Y$ is a simplicial set. A lifting problem
  \begin{equation}\label{eq:final-cone-lift}
    \xymatrix@R=2em@C=3.5em{
      {\gC[\Del^0\join Y]}\ar@/^2ex/[]!R!UR(0.5);[rr]!L!UL(0.5)^-{\Lambda_F}
      \ar[r]_-{\gC[\fbv{n}\join Y]} &
      {\gC[\boundary\Del^n\join Y]}\ar[r]_(0.55){F}\ar@{^(->}[d] &
      {\eC} \\
      {} & {\gC[\Del^n\join Y]}\ar@{-->}[ur] & {}
    }    
  \end{equation}
  has a solution whenever the cone $\Lambda_F$ presents $Fn$ as the pseudo homotopy limit of the restricted diagram $F\colon\gC Y\to \eC$.
\end{cor}

\begin{proof}
In the case of the unique simplicial functor  $!\colon\eC\to\catone$, the square \eqref{eq:coh-horn-lift-htypb} reduces to
  \begin{equation*}
    \xymatrix@R=2em@C=4em{
      {\Map_{\eC}(F0,Fn)}\ar[r]^-{\Lambda_F\circ-}\ar@{->>}[d] &
      {\wlim{W_Y}{\Map_{\eC}(F0,F-)}}\ar@{->>}[d] \\
      {1}\ar@{=}[r] & {1}
    }
  \end{equation*}
  which is a homotopy pullback when its upper horizontal map is an equivalence. By definition, however, this happens precisely when the cone $\Lambda_F$ presents $Fn$ as the pseudo homotopy limit of $F\colon\gC Y\to\eC$ as required.
\end{proof}

Corollary \ref{cor:cone-lift} proves Theorem \ref{thm:nerve-completeness}. For the reader's convenience, we summarise the argument just given.

\begin{proof}[Proof of Theorem \ref{thm:nerve-completeness}]
  The lifting property characterising the limit of $d\colon Y\to \qC$ in \eqref{eq:limit-lift-prob'} is dual to the lifting property in \eqref{eq:final-cone-lift} under the homotopy coherent nerve / realisation adjunction. Furthermore Corollary~ \ref{cor:cone-lift} tells us that this latter lifting property pertains so long as the cone $\Lambda_D\colon \gC[\Del^0\join Y]\to \eC$ displays its apex as a pseudo homotopy limit of the simplicial functor $D\colon\gC Y\to\eC$.
\end{proof}

The dual statement, recognizing quasi-categorical colimits in homotopy coherent nerves, is proven similarly or can be deduced by applying the previous result to the Kan-complex-enriched category $\eC\op$ with $\Map_{\eC\op}(A,B)\defeq \Map_{\eC}(B,A)$ and its homotopy coherent nerve $\qC\op$.

\subsection{Complete and cocomplete quasi-categories}\label{ssec:complete-qcats}

Applying Theorem~\ref{thm:nerve-completeness} and its dual we can prove directly that certain quasi-categories constructed as homotopy coherent nerves are complete and/or cocomplete.

For any $\infty$-cosmos $\eK$, we write $\qK \defeq \hN(g_*\eK)$ for the homotopy coherent nerve of its $(\infty,1)$-core $g_*\eK$ and refer to it as the \emph{quasi-category of $\infty$-categories in $\eK$.}

\begin{prop}[completeness of the quasi-category associated with an $\infty$-cosmos]\label{prop:qcat-of-cosmos-complete} For any $\infty$-cosmos $\eK$, the large quasi-category $\qK$ of $\infty$-categories in $\eK$ is small complete.
\end{prop}
\begin{proof}
  Given an $\infty$-cosmos $\eK$, the associated Kan-complex enriched category $g_*\eK\subset\eK$, constructed by taking groupoidal cores of functor spaces as in Definition~\ref{defn:gpd-cores}, admits all flexible weighted homotopy limits by Corollary \ref{cor:infty-one-core-flex}. It follows, from Theorem~\ref{thm:nerve-completeness}, that its homotopy coherent nerve $\qK\defeq\hN(g_*\eK)$ is complete.
\end{proof}

Our next result considers the sub $\infty$-cosmos of groupoidal objects $\eK\gr\inc\eK$ of Proposition \ref{prop:gpdal-infty-cosmos}. Motivated by the fact that the Kan complexes define the groupoidal objects of $\qCat$, we might call the large quasi-category  of groupoidal $\infty$-categories in $\eK$ the quasi-category $\qS_{\eK}\defeq\hN(\eK\gr)$ of \emph{spaces in $\eK$}.

\begin{prop}[completeness of the quasi-category of spaces in an $\infty$-cosmos]\label{prop:qcat-of-space-complete} For any $\infty$-cosmos $\eK$, the large quasi-category $\qS_{\eK}$ of groupoidal $\infty$-categories in $\eK$ is complete and closed under all small limits in the quasi-category $\qK$.
\end{prop}

\begin{proof}
Proposition~\ref{prop:gpdal-infty-cosmos} tells us that the full simplicial category $\eK\gr$ of groupoidal objects in an $\infty$-cosmos $\eK$ is closed in there under flexible weighted limits. So Theorem~\ref{thm:nerve-completeness} applies to show that its nerve, the quasi-category $\qS_{\eK}\defeq\hN(\eK\gr)$ of spaces in $\eK$, is closed under all limits in the complete quasi-category $\qK$.
\end{proof}

\begin{prop}\label{prop:qcat-simplicial-model}
  If $\eM$ is a simplicial model category then the quasi-category $\qM\defeq\hN(\eM_{\mathrm{cf}})$, defined as the homotopy coherent nerve of the full simplicial subcategory of fibrant-cofibrant objects, is small complete and cocomplete.
\end{prop}

A similar result, under the additional hypothesis that $\eM$ is combinatorial, appears as \cite[4.2.4.8]{Lurie:2009fk}. At the present level of generality, this result was first proven by Barnea, Harpaz, Horel \cite[2.5.9]{BarneaHarpazHorel:2017pc}.

\begin{proof}
  In any simplicial model category $\eM$, the full subcategory $\eM_{\mathrm{cf}}$ spanned by the fibrant-cofibrant objects is Kan complex enriched.

  By the argument of Proposition~\ref{prop:flexible-weights-are-htpical}, strictly defined limits of diagrams weighted by flexible weights exist in the simplicial subcategory of fibrant objects $\eM_{\mathrm{f}}$; see also \cite{Gambino:2010wl}. But the limit object might not be cofibrant. So if $W\colon\eA\to\sSet$ is a flexible weight and $D\colon\eA\to\eM_{\mathrm{cf}}$ is a diagram of fibrant-cofibrant objects then the weighted limit $\{W,D\}$ exists in $\eM_{\mathrm{f}}$ and we may take its cofibrant replacement to give an object $\{W,D\}_{\mathrm{c}}$ in $\eM_{\mathrm{cf}}$. Notice, however, that for each cofibrant $C$ in $\eM$ the representable $\Map_{\eM}(C,-)$ is a right Quillen functor, so Ken Brown's argument applies to demonstrate that it maps the cofibrant replacement map $e\colon\{W,D\}_c\we\{W,D\}$, which is a weak equivalence of fibrant objects, to an equivalence of Kan complexes:
  \begin{equation*}
    \xymatrix@R=0ex@C=5em{
      {\Map_{\eM}(C,\{W,D\}_c)}\ar[r]^{\Map_{\eM}(C,e)}_{\simeq} &      {\Map_{\eM}(C,\{W,D\})}\ar[r]^-{\cong} & {\{W,\Map_{\eM}(C,D(-))\}}
    }
  \end{equation*}
  This suffices to demonstrate that $\{W,D\}_c$ is a $W$-weighted homotopy limit of the diagram $D$ in $\eM_{\mathrm{cf}}$. In particular, given any simplicial set $X$ and any diagram $d\colon X\to\hN(\eM_{\mathrm{cf}})$ we may deduce the existence of the pseudo homotopy limit of the dual diagram $D\colon\gC[X]\to\eM_{\mathrm{cf}}$ featured in the statement of Theorem~\ref{thm:nerve-completeness}. It follows, by application of that theorem, that the object $\{W_X,D\}_c$ provides a limit for the diagram $d$ in the quasi-category $\hN(\eM_{\mathrm{cf}})$, and thus that $\hN(\eM_{\mathrm{cf}})$ admits all limits.

  To prove the dual statement about colimits, simply observe that the homotopy coherent nerve of the dual simplicial category $\eM\op_{\mathrm{cf}}$ is isomorphic to the (simplicial) dual of the quasi-category $\hN(\eM_{\mathrm{cf}})$. Consequently, the stated result providing for the existence of colimits in $\hN(\eM_{\mathrm{cf}})$ follows by applying the argument just given to the dual model category $\eM\op$.
\end{proof}

\begin{ex}\label{ex:complete-presheafs}
  Suppose that $\eC$ is a small Kan-complex-enriched category. We equip the simplicial functor category $\SSet^{\eC\op}$ with the projective model structure, relative to the Kan model structure on $\SSet$, and observe that this makes it into a simplicial model category. Its weak equivalences, fibrations, and fibrant objects are given pointwise in the Kan model structure, and the projective cells of Definition~\ref{defn:flexible-weight} provide a generating set for its cofibrations. We shall use the notation $\hat\eC$ to denote the full simplicial subcategory of $\SSet^{\eC\op}$ spanning its projective fibrant-cofibrant objects. Proposition \ref{prop:qcat-simplicial-model} then applies to show that its homotopy coherent nerve $\hat\qC\defeq\hN\hat\eC$ is a complete and cocomplete quasi-category.
  \end{ex}
  
  Each representable functor on $\eC$ is both projective cofibrant, by definition, and projective fibrant, since the hom-spaces of $\eC$ are Kan complexes, so it follows that the Yoneda embedding restricts to a simplicial functor $\yoneda\colon\eC\to\hat\eC$, which, by the simplicial Yoneda lemma is fully faithful.   Now suppose that $e\colon F\to G$ is a simplicial natural transformation in $\hat\eC$ for which $\Map_{\hat\eC}(\yoneda c,e)\colon \Map_{\hat\eC}(\yoneda c, F)\to \Map_{\hat\eC}(\yoneda c ,G)$ is an equivalence of Kan complexes for all objects $c\in\eC$. By the simplicial Yoneda lemma that latter map is isomorphic to the component $e_c\colon F c \to G c$, so the given condition simply posits that each component of the natural transformation $e$ is an equivalence of Kan complexes. In other words, it is a weak equivalence between fibrant and cofibrant objects in the projective model structure and is thus an equivalence in $\hat\eC$.
  
  The following definitions give names for these properties:

\begin{defn}\label{defn:ho-notions}
  Suppose that $F\colon\eC\to\eD$ is a simplicial functor between Kan-complex-enriched categories. We say that $F$ is
  \begin{enumerate}[label=(\roman*)]
  \item \emph{homotopically fully-faithful} if its action $F\colon\Map_{\eC}(A,B)\to\Map_{\eD}(FA,FB)$ on each hom-space is an equivalence of Kan complexes, and
  \item \emph{homotopically strongly generating} if a $0$-arrow $e\colon A\to B$ in $\eD$ is an equivalence if and only if for all objects $C$ in $\eC$ the map $\Map_{\eD}(FC,e)\colon\Map_{\eD}(FC, A)\to\Map_{\eD}(FC,B)$ is an equivalence of Kan complexes.
  \end{enumerate}
\end{defn}

\begin{rmk}
These concepts are related to, but not the same as, the simplicially enriched variants of the usual fully-faithfulness and strong generation properties: these homotopical notions insist or infer that certain maps are equivalences, whereas their simplicially enriched cousins ask for those maps to be isomorphisms. While it is certainly the case  that any fully-faithful simplicial functor is also homotopically fully-faithful, it is not however the case, in general, that a strongly generating functor is homotopically strongly generating (or vice versa). 
\end{rmk}

Since the Yoneda functor $\yoneda\colon\eC\to\hat\eC$ is homotopically fully faithful and homotopically strongly generating it follows that  the homotopy coherent nerve $\yoneda\colon\qC\to\hat\qC$ of the Yoneda functor is a fully faithful and strongly generating functor of quasi-categories, see \cite[5.1.3.1]{Lurie:2009fk} or \cite[5.1.9]{RiehlVerity:2018oc}. This defines an embedding of the homotopy coherent nerve of any small Kan-complex-enriched category into a (small) complete and cocomplete quasi-category. In general, any fully faithful and strongly generating functor of $\infty$-categories preserves all limits known to exist in both the domain or codomain --- see \cite[5.2.9]{RiehlVerity:2018oc} --- but we don't need the full strength of that result, so we don't take time to prove this here. We apply these observations to provide the following converse to Theorem~\ref{thm:nerve-completeness}.

\begin{thm}\label{thm:nerve-completeness-converse}
  Suppose that $\eC$ is a small Kan-complex-enriched category, that $X$ is a simplicial set and that $d\colon X\to \qC$ is a diagram in the associated quasi-category $\qC\defeq \hN\eC$. If the diagram $d$ admits a limit in $\qC$ then the transposed diagram $D\colon\gC X\to\eC$ admits a pseudo homotopy limit in $\eC$. 
\end{thm}

\begin{proof}
  Let $L$ denote the pseudo homotopy limit of the diagram $\yoneda D\colon \gC X\to\hat\eC$, whose existence in $\hat\eC$ is guaranteed by Proposition \ref{prop:qcat-simplicial-model} applied to Example \ref{ex:complete-presheafs}, as displayed by a pseudo cone $\Lambda\colon W_X\to \Map_{\hat\eC}(L,\yoneda D-)$. Let $\ell$ denote the limit of the diagram $d\colon X\to\qC$ in the quasi-category $\qC$, equipped with an equivalence $\qC\comma\ell \simeq \Delta\comma d$ over $\qC$, which pulls back to define an equivalence $c\comma \ell \simeq \Delta c \comma d$ for any object $c \in \eC$ considered as $c \colon 1 \to \qC$. Applying $\yoneda\colon\qC\to\hat\qC$, there is a canonical cone with summit $\yoneda\ell$ over the diagram $\yoneda d\colon X\to\hat\qC$. As Theorem~\ref{thm:nerve-completeness} shows that the pseudo homotopy limit $L$ gives a limit of the diagram $\yoneda d$ in $\hat\qC$, this cone induces a comparison map $f \colon \yoneda \ell \to L$ in $\hat\qC$ as well as an equivalence $\hat\qC\comma{L} \simeq \Delta \comma\yoneda d$ over $\hat\qC$, which pulls back along $\yoneda c \colon 1 \to \hat\qC$ to an equivalence $\yoneda c \comma {L} \simeq  \Delta\yoneda c \comma \yoneda d$. By Corollary \refVI{cor:fun-to-hom}, the internal mapping spaces in the homotopy coherent nerve of a Kan-complex-enriched category are equivalent to the simplicial mapping spaces. Thus, by fully faithfulness of the simplicial Yoneda embedding $\yoneda\colon\eC\to\hat\eC$,
  \[ \Map_{\hat\eC}(\yoneda c, \yoneda \ell) \simeq \Map_{\eC}(c,\ell) \simeq  c \comma \ell \simeq \Delta c \comma d  \simeq \Delta\yoneda c \comma \yoneda d \simeq \yoneda c \comma L \simeq \Map_{\hat\eC}(\yoneda c,L).\]

  Now we may pre-compose the pseudo cone by the arrow $f$ to give the following commutative square
  \begin{equation*}
    \xymatrix@=2em{
      {W_X}\ar[r]^-{\Lambda}\ar@{-->}[d]_{\exists!\bar\Lambda} &
      {\Map_{\hat\eC}(L,\yoneda D-)}\ar[d]^{-\circ f} \\
      {\Map_{\eC}(\ell,D-)}\ar[r]_-{\yoneda}^-{\cong} &
      {\Map_{\hat\eC}(\yoneda\ell,\yoneda D-)}
    }
  \end{equation*}
  in which the unique existence of the pseudo cone on the left is guaranteed by the fact that the simplicial Yoneda embedding at the bottom is fully-faithful. To see that $\bar\Lambda$ presents $\ell$ as the pseudo homotopy limit of $D$ in $\eC$, we must verify that the map induced by post-composition with $\bar\Lambda$
  \[
    \xymatrix@R=0em@C=5em{
\Map_{\eC}(c,\ell) \ar[r] &
    \{W_X,\Map_{\eC}(c,D-)\}_{\gC[X]}.
  }
  \]
  is an equivalence. By construction this is equivalent to the map 
  \[
    \xymatrix@R=0em@C=5em{
\Map_{\yoneda \eC}(\yoneda c,L) \ar[r] &
    \{W_X,\Map_{\hat\eC}(\yoneda c,\yoneda D-)\}_{\gC[X]}.
  }
  \]
  induced by post-composition with  the pseudo homotopy limit cone $\Lambda$, and hence is an equivalence as desired.
\end{proof}

  \bibliographystyle{special}
  \bibliography{../../common/index}


\end{document}